\documentclass[11pt,a4paper]{article}
\usepackage{amsmath}
\usepackage{amsfonts}
\usepackage{amsthm}
\usepackage{amssymb}
\usepackage{url}
\usepackage[usenames,dvips]{color}
\usepackage{graphicx}
\usepackage{psfrag, float}  %per al text dels gr�fics

\newtheorem{theorem}{Theorem}
\newtheorem{lemma}{Lemma}[section]
\newtheorem{corollary}[lemma]{Corollary}
\newtheorem{definition}[lemma]{Definition}
\newtheorem{remark}[lemma]{Remark}

\newtheorem{proposition}[lemma]{Proposition}

\def\be{\begin{eqnarray}}
\def\ee{\end{eqnarray}}
\def\beal{\begin{aligned}}
\def\enal{\end{aligned}}
\newcommand{\eps}{\varepsilon}
\newcommand{\ga}{\gamma}
\newcommand{\dps}{\displaystyle}
\newcommand{\RR}{\mathbb{R}}

\newcommand{\CC}{\mathbb{C}}
\newcommand{\TT}{\mathbb{T}}
\newcommand{\Q}{\mathbb{Q}}
\newcommand{\ZZ}{\mathbb{Z}}
\newcommand{\MM}{\mathcal{M}}
\newcommand{\WW}{\mathcal{W}}
\newcommand{\PP}{\mathcal{P}}
\newcommand{\GG}{\mathcal{G}}
\newcommand{\II}{\mathcal{I}}
\newcommand{\BB}{\mathcal{B}}
\newcommand{\LL}{\mathcal{L}}

\newcommand{\YY}{\mathcal{Y}}
\newcommand{\KK}{\mathcal{K}}
\newcommand{\KKK}{\mathbb{K}}
\newcommand{\SSS}{\mathcal{S}}
\newcommand{\DD}{\mathbb{D}}
\newcommand{\XX}{\mathcal{X}}
\newcommand{\OO}{\mathcal{O}}
\newcommand{\FF}{\mathcal{F}}
\newcommand{\HH}{\mathcal{H}}
\newcommand{\VV}{\mathcal{V}}
\newcommand{\RRR}{\mathcal{R}}

\newcommand{\UU}{\mathcal{U}}
\newcommand{\NNN}{\mathcal{N}}
\newcommand{\EE}{\mathcal{E}}
\newcommand{\JJ}{\mathcal{J}}
\newcommand{\AAA}{\mathcal{A}}
\newcommand{\ZZZ}{\mathcal{Z}}
\newcommand{\CCC}{\mathcal{C}}
\newcommand{\QQQ}{\mathcal{Q}}
\newcommand{\Id}{\mathrm{Id}}
\newcommand{\DDD}{\mathcal{D}}

\newcommand{\ii}{^{-1}}
\newcommand{\de}{\delta}
\newcommand{\pa}{\partial}
\newcommand{\la}{\lambda}
\newcommand{\inn}{\mathrm{in}}
\newcommand{\out}{\mathrm{out}}
\newcommand{\al}{\alpha}
\newcommand{\kk}{\kappa}
\newcommand{\rr}{\rho}

\newcommand{\tet}{\theta}
\newcommand{\ol}{\overline}
\newcommand{\La}{\Lambda}
\newcommand{\bet}{\beta}

\renewcommand{\Re}{\mathrm{Re\, }}
\renewcommand{\Im}{\mathrm{Im\,}}
\newcommand{\wt}{\widetilde}
\newcommand{\wh}{\widehat}

\newcommand{\hyp}{\mathrm{hyp}}
\newcommand{\el}{\mathrm{ell}}
\newcommand{\mix}{\mathrm{mix}}
\newcommand{\adj}{\mathrm{adj}}

\newcommand{\loc}{\mathrm{loc}}
\newcommand{\glob}{\mathrm{glob}}

\begin{document}

\title{Growth of Sobolev norms in the cubic defocusing nonlinear
Schr\"odinger equation}
\maketitle

\author{\qquad \qquad \qquad \qquad
\qquad M. Guardia
 \footnote{ IAS and University of Maryland at College Park
    (\url{marcel.guardia@upc.edu})}
\qquad V. Kaloshin \footnote{ IAS and University of
    Maryland at College Park (\url{vadim.kaloshin@gmail.com})} }

%\begin{center}
%Preliminary version
%\end{center}

\begin{abstract}
We consider the cubic defocusing nonlinear Schr\"odinger
equation in the two dimensional torus. Fix $s>1$.
Colliander, Keel, Staffilani, Tao and Takaoka proved in
\cite{CollianderKSTT10} the existence of solutions with
$s$-Sobolev norm growing in time.

We establish the existence of solutions with polynomial time
estimates. More exactly, there is $c>0$ such that for any
$\KK\gg 1$ we find a solution $u$ and a time $T$
such that $\| u(T)\|_{H^s}\geq\KK \| u(0)\|_{H^s}$. Moreover,
time $T$ satisfies polynomial bound $0<T<\KK^c$.
\end{abstract}

\tableofcontents

\section{Introduction}
Let us consider the periodic cubic defocusing nonlinear Schr\"odinger
equation (NLS),
\begin{equation}\label{def:NLS}
\left\{\begin{aligned}
&-i \pa_t u+\Delta u=|u|^2 u\\
&u(0,x)=u_0(x)
\end{aligned}\right.
\end{equation}
where $x\in\TT^2=\RR^2/(2\pi\ZZ)^2$, $t\in\RR$ and
$u:\RR\times\TT^2\rightarrow\CC$.

The solutions of equation \eqref{def:NLS}
conserve two quantities: the Hamiltonian
\begin{equation*}
 E[u](t)=\int_{\TT^2}\left(\frac{1}{2}\left|\nabla
u\right|^2+\frac{1}{4}|u|^4\right)
 dx(t)
\end{equation*}
and mass
\begin{equation}\label{def:NLS:mass}
 \MM[u](t)=\int_{\TT^2}|u|^2dx(t)=\int_{\TT^2}|u|^2dx(0),
\end{equation}
which is just the square of the $L^2$-norm of the solution for any $t>0$.
It is useful to study solutions $u(t)$ in a family of Sobolev spaces $H^s$
with the corresponding $H^s$-norms
\[
\|u(t)\|_{H^s(\TT^2)}:=\|u(t,\cdot)\|_{H^s(\TT^2)}:=
\left(\sum_{n\in\ZZ^2} \langle n\rangle ^{2s}|\hat u(t,n)|^2\right)^{1/2},
\]
where $\langle n\rangle=(1+|n|^2)^{1/2}$ and,
\[
 \hat u(t,n):= \int_{\TT^2} u(t,x)e^{-in\cdot x}\ dx.
\]
The local-in-time well-posedness for any $u_0\in H^s(\TT^2),\ s>0$
was proven by Bourgain \cite{Bourgain93}.
This along with the two conservation
laws, implies existence of a smooth solution (\ref{def:NLS}) for all time.
It follows from conservation of energy $E[u](t)$ that the $H^1$-norm
of any solution of (\ref{def:NLS}) is uniformly bounded. Our main
goal is {\it to look for solutions whose higher Sobolev norms
$\|u(t)\|_{H^s(\TT^2)},\ s>1$, can grow in time.}

If the $H^s$-norm can grow indefinitely for some given $s>1$, while
the $H^1$-norm stays bounded, then we have solutions which initially
oscillate only on the scales comparable to the spatial period and
eventually oscillate on arbitrarily small scales. To see that compare
these norms. The only possibility for $H^s$ to grow indefinitely is
that the energy of a solution of (\ref{def:NLS}) can penetrate to
higher and higher Fourier modes.

On the one-dimensional torus, equation (\ref{def:NLS})
is completely integrable due to the famous result of
Zakharov-Shabat \cite{ZakharovS71} (see also \cite{Grebert12}).
As a corollary $\|u(t)\|_{H^s(\TT^1)}\le C \|u(0)\|_{H^s(\TT^1)},\
s\ge 1$ for all $t>0$. If one replaces the nonlinearity
$|u|^2 u=\partial_{\bar u} P(|u|^2)$ in (\ref{def:NLS}) with
a more general polynomial, then Bourgain \cite{Bourgain96} and
Staffilani \cite{Staffilani} proved at most polynomial growth
of Sobolev norms. Namely, for some $C>0$ we have
\[
\|u(t)\|_{H^s}\le t^{C(s-1)}\|u(0)\|_{H^s} \quad \text{ for }\qquad t\to \infty.
\]
In \cite{Bourgain00a} Bourgain applied a version of Nekhoroshev theory.
He proved that for a 1-dimensional NLS with a polynomial nonlinearity
$P(|u|^2)$ satisfying $P(0)=P'(0)=P''(0)=0$ for $s$ large
and a typical initial data $u(0) \in H^s(\TT)$ of small size $\eps$,
i.e. $\|u(0)\|\le \eps$ we have
\[
\sup_{|t|<T}\|u(t)\|_{H^s}\le C \eps,
\]
where $T\le \eps^{-A}$ with $A=A(s)\to 0$ as $s\to +\infty$.
This is an indication of absence of a polynomial growth and
motivated Bourgain \cite{Bourgain00b} to pose the following question:

\vskip 0.1in

{\it Are there solutions in dimension $2$ or higher with unbounded
growth of $H^s$-norm for $s>1$? }

\vskip 0.1in

Moreover, he conjectured, that in case this is true, the growth should
be subpolynomial in time, that is,
\[
\|u(t)\|_{H^s}\ll  t^{\eps}\|u(0)\|_{H^s} \quad \text{ for }
\qquad t\to \infty, \text{ for all }\eps>0.
\]
There are several papers obtaining improved polynomial
{\it upper} bounds for the growth of Sobolev norms for equation \eqref{def:NLS} and
also generalizing these results to other nonlinear Schr\"odinger equations either on $\RR$,
or $\RR^2$, or on compact manifolds
\cite{Staffilani97, CollianderDKS01,Bourgain04,Zhong08, CatoireW10, Sohinger11,CollianderKO12}.
Similar results have been obtained for the wave equation \cite{Bourgain96} and
for the Hartree equation \cite{Sohinger10a,Sohinger10b}.

All of the cited above papers give {\it upper} bounds of the growth
but {\it do not obtain} orbits which undergo growth. Indeed, there
are few results obtaining such orbits. In \cite{Bourgain96}, Bourgain
constructs orbits with unbounded growth of the Sobolev norms for
the wave equation with a cubic nonlinearity but with a spectrally
defined Laplacian. In \cite{GerardG10,Pocovnicu11}, it is shown growth
of Sobolev norms for the Szeg\"o equation, and in \cite{Pocovnicu12}
for certain nonlinear wave equation.

Concerning the nonlinear Schr\"odinger equation, Kuksin in \cite{Kuksin97b}
(see related works \cite{Kuksin95, Kuksin96, Kuksin97, Kuksin99})
studied the growth of Sobolev norms but for the equation
\[
-i\dot w=-\de \Delta w+|w|^{2p}w, \,\,\de\ll 1,\ p\geq 1.
\]
He obtained solutions whose Sobolev norms grow by an inverse power of $\de$.
Note that $u_\de(t,x)=\de^{-\frac{1}{2}}w(\de\ii t,x)$ is a solution of  \eqref{def:NLS}.
Therefore, the solutions that he obtains correspond to orbits of equation
\eqref{def:NLS} with large initial data. The present paper is closely related
to  \cite{CollianderKSTT10}. In this paper, it was shown that for any $s>1$
the $H^s$-norm can grow by any predetermined factor.
The initial data there are not required to be large as \cite{Kuksin97b},
but rather have a small initial $H^s$-norm with $s>1$.  {\it Essentially using
construction from this paper} \cite{CollianderKSTT10} we not only construct
solutions with similar properties, but also estimate their speed of diffusion.

The main result of this paper is
\begin{theorem}\label{thm:main}
Let $s>1$. Then there exists $c>0$ with the following property:
for any large $\KK\gg 1$ there exists a a global solution $u(t,x)$
of (\ref{def:NLS}) and a time $T$ satisfying
\[
0 < T \leq \KK^c
\]
such that
% such that for any $t$ with $1< t < T$ we have
%\[
% \|u(t)\|_{H^s} \ge t^{\frac{1}{c}}  \|u(0)\|_{H^s}.
%\]
%In particular,
\[
\|u(T)\|_{H^s}\ge \KK\,\|u(0)\|_{H^s}.
\]
Moreover, this solution can be chosen to satisfy
\[
 \|u(0)\|_{L^2}\le \KK^{-(s-1)c/4+2/(s-1)}.
\]
\end{theorem}

Note that Theorem \ref{thm:main} does not contradict Bourgain conjecture
about the subpolynomial growth. Indeed, Theorem \ref{thm:main} only obtains
solutions with arbitrarily large but finite growth in the Sobolev norms
whereas Bourgain conjecture refers to unbounded growth.

%
%I reorder remarks and put an additional one about speed of growth.
\begin{remark}
Even if Theorem \ref{thm:main} is stated for \eqref{def:NLS} in
the two torus, it can be applied to the $d$ dimensional torus
with $d\geq 2$, since the solution we obtain is also a solution
for equation \eqref{def:NLS} in the $\TT^d$ setting all
the other harmonics to zero.
\end{remark}
\begin{remark}
In fact, we can obtain more detailed information about
the distribution of the Sobolev norm of the solution $u(T)$
from Theorem \ref{thm:main} among its Fourier modes.
More precisely, we can ensure that there exist $n_1,n_2\in\ZZ^2$
such that
\[
  \|u(T)\|_{H^s}\ge
  |n_1|^{2s}|u_{n_1}(T)|^2+|n_2|^{2s}|u_{n_2}(T)|^2\geq \KK.
\]
That is, when $t=T$ the Sobolev norm is essentially localized
on {\bf two} Fourier coefficients.
\end{remark}
\begin{remark}
Using more careful analysis of the proof we can establish
existence of solutions whose Sobolev norms are lower bounded
for each time $t\in [1,T].$ Namely,
\[
\ln \|u(t)\|_{H^s} \ge \dfrac{t\,\ln \KK}{\KK^c} + \ln \|u(0)\|_{H^s}.
\]
\end{remark}
\begin{remark}\label{remark:SlowerTime}
Our solutions differ from solutions studied in \cite{CollianderKSTT10} in
a substantial way. If one applies to information about dynamics contained
in \cite{CollianderKSTT10} supplied with the theory of normal forms
and a beautiful trick of Shilnikov \cite{Shilnikov67}, then it is
possible to compute certain ``local maps'' and diffusion time.
It turns out to be super-exponential in $\KK$, namely, it grows
as $C^{\,\KK^\al}$ for some $C>0$ and $\al\geq 2$ (see Section
\ref{sec:MainIdeasSaddle} for more details).
Even equipped with the aforementioned dynamical technique in order to obtain
polynomial diffusion time we need to achieve $\sim \ln \KK$ cancelations.
These cancellations are spilled out in Section \ref{sec:MainIdeasSaddle}
on an heuristic level and then worked out in Sections \ref{sec:HypToyModel}
and \ref{sec:FullSystem}.
\end{remark}

In \cite{CollianderKSTT10} initial conditions of solutions with growth of
Sobolev norms can be chosen with small $\|u(0)\|_{H^s}$\footnote{As Terence
Tao pointed to us, our solutions have small $L^2$-norm, but not $H^s$-norm}.
In our case it is also possible, but leads to slowing down of time of
 growth. This fact is explained in Appendix \ref{app:SmallSobolev}.

The present paper deals with growth of Sobolev norms for a Hamiltonian partial
differential equation.
We show the existence of unstable solutions.
As we have explained, there have not been many results showing the existence of
these instabilities. In \cite{CarlesF10} a solution of \eqref{def:NLS} with spreading of mass among modes is constructed. Nevertheless the spreading does not lead to growth of Sobolev norms. In \cite{Hani11} a progress toward infinite
growth of Sobolev norms is made. Let us say also that in the past decades there
has been a considerable progress in the study of other types of dynamics for
Hamiltonian partial differential equations. For instance, in the existence of
periodic, quasi-periodic or almost-periodic solutions (see e.g. \cite{Rabinowitz78,Wayne90,CraigW93,KappelerP03,Kuksin93,KuksinP96,
Berti07, BertiB11}), in Nekhoroshev type results (see e.g.
\cite{Bambusi97,Bambusi99b}) and normal forms
(see e.g. \cite{Bambusi031, BambusiG06,GrebertIP09, Grebert12, ProcesiP12}). Of particular
interest for the present paper are \cite{Bourgain98,Eliasson10} since, in
these papers, the authors study the existence of quasi-periodic solutions for
the nonlinear Schr\"odinger equation in the 2-dimensional torus \cite{Bourgain98}
and in a torus of any dimension \cite{Eliasson10}. Nevertheless, they consider
slightly different equations  containing a convolution potential.

\section{Main ideas and structure of the proof}\label{sec:MainIdeas}

One of remarkable contributions in \cite{CollianderKSTT10}
is the formulation of a finite-dimensional toy model, which
after a certain lift approximates solutions of (\ref{def:NLS}).
The Hamiltonian of the toy model from \cite{CollianderKSTT10} has a specific
form. It has a nearest neighbors interaction
and is integrable inside a certain family of $4$-dimensional planes.
In this section we present a class of Hamiltonians with a nearest
neighbors interaction for which our method applies.
It is specified at the end of Section \ref{sec:FeaturesModel}.

\subsection{Features of the model}\label{sec:FeaturesModel}
\begin{itemize}
\item Write (\ref{def:NLS}) as infinitely ODE's for Fourier
coefficients of solutions. It is a Hamiltonian system with Hamiltonian $\HH$ (see \eqref{def:HamForFourier}).

\item ({\it Two step reduction})

--- Obtain a Normal Form of the original Hamiltonian
near the origin by removing non-resonant terms (see Theorem \ref{thm:NormalForm}).

--- Use gauge freedom to remove linear and some
non-linear terms (see \eqref{eq:InftyODE:VariationConstants}).

\item ({\it The Toy Model})

Select a finite subset of Fourier coefficients $\La$ in
$\ZZ^2$ so that they can be split into pairwise disjoint
generations $\La=\cup_{j=1}^N \La_j$ and
{\it only neighboring}
generations $\La_j$ and $\La_{j+1}$ interact.

This can be done so that dynamics of each element in each
generation has exactly the same as dynamics of any other
member of this generation (see Corollary \ref{coro:Invariant}).
Truncating we are reduced to a complex $N$-dimensional system
given by a Hamiltonian
\[
h(b)=
\dfrac  14 \sum_{j=1}^N |b_j|^4-
\dfrac  12 \sum_{j=2}^{N-1}
\left( b_j^2 \overline b^2_{j-1}+
\overline b_j^2 b^2_{j-1}\right),
\]
where each $b_j$ is complex valued, and the symplectic
form $\Omega=\frac{i}{2}db_j\wedge \ol b_j$. The system
conserves mass $\MM(b)=\sum_{j=1}^N |b_j|^2$. We study
the dynamics
restricted to mass $\MM(b)=1$. Dynamics of this Hamiltonian
is called in \cite{CollianderKSTT10} {\it the Toy Model} and
is the focal point of analysis. It is convenient to study
this system in real coordinates and identify $\CC\cong \RR^2$.

Notice also that the Hamiltonian $h(b)$ can be viewed as
a Hamiltonian on a lattice $\ZZ$ with nearest neighbor
interactions. Our main result relies on the construction of
energy transfer from $b_3\approx 1, b_j \approx 0,\ j\ne 3$
to $b_{N-2}\approx 1,\ j\ne N-1$ for this Hamiltonian.
Construction of a somewhat similar energy transfer for
the pendulum lattice is done in \cite{Kaloshin-Levi-Saprykina}.

\item ({\it Invariant low-dimensional subspaces})

Notice that each $4$-dimensional plane
\[
L_j=\{b_1=\cdots=b_{j-1}=b_{j+2}=\cdots=b_N=0\}
\]
is invariant. Moreover, dynamics in $L_j$ is given by a simple
Hamiltonian
\[
h_j(b_{j},b_{j+1})=\dfrac 14 \left(|b_j|^4+|b_{j+1}|^4\right)- \dfrac 12
\left(b_j^2 \overline b^2_{j+1}+
\overline b_j^2 b^2_{j+1}\right).
\]
 Denote
$\MM_j(b_j,b_{j+1})=|b_j|^2+|b_{j+1}|^2$. Both $h_j$ and $\MM_j$
are conserved. The mass $\MM_j$ is assumed to be $1$.

\begin{center}
{\it The solutions constructed stays close to the planes
$\{L_j\}_{j=2}^{N-1}$  \newline
\qquad \qquad \qquad and  go from one intersection
$l_j=L_j\cap L_{j+1}$ to the next one
\newline
 $l_{j+1}=L_{j+1}\cap L_{j+2}$ consequently for $j=3, \dots, N-2$
 (see Figure \ref{fig:plane-approx}).\qquad \qquad \qquad }
\end{center}

To make a closer look at solutions we need to understand dynamics
in the planes $L_j$'s.

\begin{figure}[t]
  \centering
  \includegraphics[width=5in]{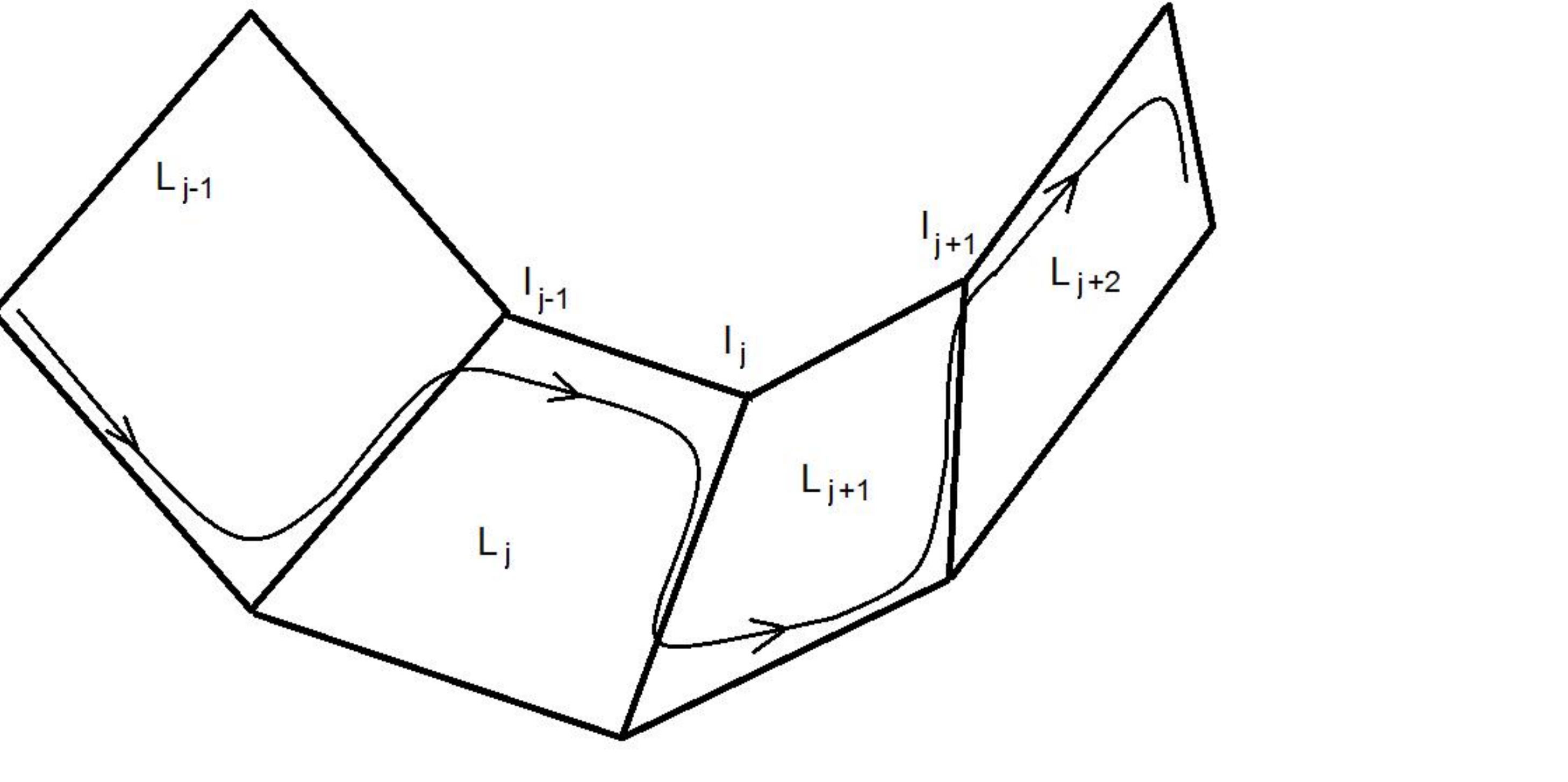}
  \caption{Planes approximating solutions}
  \label{fig:plane-approx}
\end{figure}
%[width=4in]

\item ({\it Integrable dynamics in each plane $L_j$})

Dynamics in each $2$-dimensional plane $L_j$ is integrable.
Indeed, there are two first integrals $h_j$  and $M_j$ in
involution. By Arnold-Liouville theorem away from degeneracies
the $4$-dimensional plane $L_j$ is foliated by $2$-dimensional
invariant tori with dynamics smoothly conjugated to a constant flow.

We are interested in two specific periodic orbits: $\tet_j$-direction
$\{|b_j|=1,\ b_{j+1}=0\}$ and $\tet_{j+1}$-direction
$\{|b_{j+1}|=1,\ b_j=0\}$ and in a family of heteroclinic orbits
$\{\ga_j\}$ connecting the former with the later.
All these orbits can be found explicitly, but {\it their
existence can be predicted having $h_j$ and $\MM_j$
satisfying some properties}.

\begin{itemize}
\item Having the mass $\MM_j=|b_j|^2+|b_{j+1}|^2$ conserved it is natural
to expect that the boundary is invariant. The boundary consists of
$b_j=0$ and $b_{j+1}=0$ (both periodic orbits) and belong to
the same $h_j$-energy surface.

\item It is a straightforward calculation
to check that both orbits are hyperbolic, i.e. of saddle type.

\item Notice that $\{h_j=\frac 14,\ \MM_j=1\}$ is a $2$-dimensional
surface with the boundary given by periodic orbits $b_j=0$ and $b_{j+1}=0$.
Away from these periodic orbits it is a locally analytic surface, i.e.
gradients $\nabla h_j$ and $\nabla \MM_j$ are linearly independent.

\item Away from the periodic orbits $b_j=0$ and $b_{j+1}=0$
the surface $\{h_j=\frac 14,\ \MM_j=1\}$ consists of stable and
unstable $2$-dimensional manifolds. Unless the periodic orbits $b_j=0$
and $b_{j+1}=0$ on $\{h_j=\frac 14,\ \MM_j=1\}$ are separated by a degenerate
periodic orbit, they have to be connected by these manifolds.

\item Now we verify that there is no such a degenerate periodic orbit.
Moreover,
we find explicitly the family of connecting heteroclinic orbits.
Even though these explicit formulas is not used in our proof.
\end{itemize}

Write in polar coordinates $b_k=\sqrt r_k\, e^{i\theta_k},\ k=j,j+1$.
The mass conservation becomes $\MM_j(b)=r_j+r_{j+1}$, the symplectic form $\Omega=\dfrac 12 dr_j\wedge d\theta_j$ and the Hamiltonian
\[
h_j\left(\sqrt{r_j}\ e^{i\theta_j},\sqrt{r_{j+1}}\ e^{i\theta_{j+1}}\right)= \dfrac
14 \left[r_j^2+r_{j+1}^2+4r_jr_{j+1}\cos 2(\theta_j-\theta_{j+1}))\right].
\]
Then the equation of motion are
\begin{align*}
\dot \tet_j &=  r_j-2r_{j+1}\cos 2(\theta_j-\theta_{j+1}) \\
\dot \tet_{j+1}&=  r_{j+1}-2r_j\cos 2(\theta_j-\theta_{j+1})\\
\dot r_j &= 4r_j \,r_{j+1} \sin 2(\theta_j-\theta_{j+1}) \\
\dot r_{j+1}&= - 4r_j \,r_{j+1} \sin 2(\theta_j-\theta_{j+1}).
\end{align*}
%It easily follows from equation of motion that
For the energy surface $h_j=\dfrac 14$  we have
\begin{itemize}
\item Two families of periodic solutions \newline
$\{(\tet_j,\tet_{j+1},r_j,r_{j+1}):\ r_j=0\}$  and
$\{(\tet_j,\tet_{j+1},r_j,r_{j+1}):\ r_{j+1}=0\}$.

\begin{figure}[t]
  \centering
  \includegraphics[width=5in]{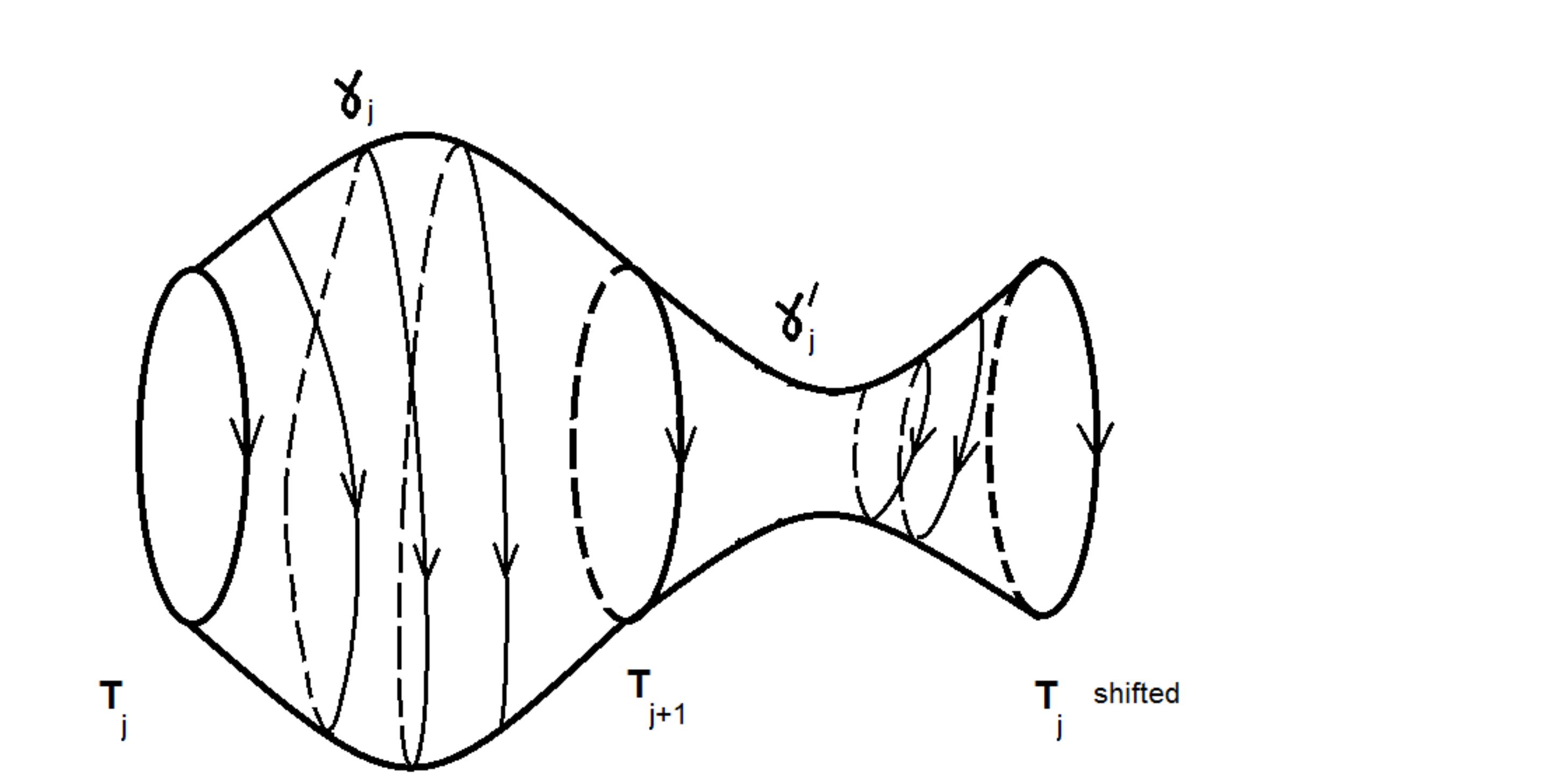}
  \caption{Heteroclinic orbits}
  \label{fig:heteroclinics}
\end{figure}
%[width=4in]

\item Each family has two special solutions: $2(\tet_j-\tet_{j+1})$
equals either $\frac {2\pi}{3}$ and $\frac {4\pi}{3}$. Both planes
are invariant:
$\frac{d}{dt}{(\tet_j-\tet_{j+1})}=
- (r_{j}+r_{j+1}) (1+2\cos 2(\theta_j-\theta_{j+1}))=0.$
Denote $\TT_j=\{2(\tet_j-\tet_{j+1})=\frac {2\pi}{3}\ (\text{ mod }2\pi),\ r_j=0\}$.

\item On $\MM_j=1,\ h_j=\frac 14,\ \tet_j-\tet_{j+1}=\frac {2\pi}{3}$
we have $\dot r_j= r_j r_{j+1}=-\dot r_{j+1}$. Thus, there
is a heteroclinic orbit $\ga_j$ connecting $\TT_{j}$ with the second
family $r_{j+1}=0$.
\end{itemize}
Now we can be more specific in location of orbits:
\be \label{plane-location-orbit-bis}
\beal
\textit{ The solutions constructed go from one periodic orbit }
\  \TT_{2} \textit{ to the next }\ \TT_{3}
\\
 \textit{ along }\ga_{2}, \textit{ then from }\ \TT_{3}
 \textit{ to }\ \TT_{4}\ \textit{ along }\ga_{3}
  \textit{ and so on for }j=4, \dots, N-2.
\enal
\ee

In a view of the above discussion we have the following  description:
{\small \be
\beal
\leadsto\qquad
&  \leadsto\ \TT_{j}\ \leadsto & \qquad \leadsto\ \ga_j \ \leadsto
\qquad \qquad  & \qquad \leadsto\ \TT_{j+1}\qquad \leadsto  \\
& \dot \tet_i\approx 0,\ |i-j|>1 \qquad &\qquad  \tet_j-\tet_{j+1}\approx \frac{\pi}{3}\qquad  &
\qquad \dot \tet_i\approx 0,\ |i-j-1|>1
\\
& |b_i|\approx 0,\ i\ne j,j+1 \qquad &  |b_j| \leadsto |b_{j+1}|
\qquad \quad & \qquad |b_i| \approx 0,\ i\ne j+1,j+2.
\enal
\ee }

\item ({\it Local behavior of periodic orbits} $\TT_j$)
Due to the above analysis, the periodic orbits $\TT_j$ viewed in $\RR^{2N}$
have at least two expanding and
two contracting directions: one pair from $L_{j-1}$-plane
and the other from $L_j$-pane. Due to symmetry of
the restricted systems in $L_{j-1}$-plane and $L_{j}$-plane
these periodic orbits have {\bf multiple} hyperbolic eigenvalues.
Multiplicity turns out to be {\bf exactly} $2$.

\item ({\it Resonant normal forms near} $\TT_j$) Presence of
resonance complicated analysis and as formulas (\ref{local-map})
show resonance changes local behavior compare to the linear case.
To resolve it we use a beautiful trick of Shilnikov \cite{Shilnikov67}
and obtain precise information about local behavior, which is
explained in Section \ref{sec:MainIdeasSaddle}.

\item ({\it Connecting heteroclinic orbits})
As we showed above there are orbits $\ga_j$ connecting $\TT_j$
with $\TT_{j+1}$ for each $j=3,\dots, n-2$. We need
to analyze dynamics near these heteroclinic orbits.

\item ({\it Local
almost product structure}) Once we obtain information
about behavior near $\TT_j$'s and near connecting orbits $\ga_j$,
we can describe dynamics of the Toy Model as if it close to
the direct product of $(N-3)$ planes $L_j,\ j=3,\dots,N-1$.
\end{itemize}

\begin{center}
Properties of the Hamiltonian $h(b)$ used in the proof.
\end{center}
As we mentioned in the introduction to this section
we do not use a specific form of $h$. Here is the list
of properties that we need.
\begin{itemize}
 \item $h$ has nearest neighbors interaction;
 \item $h$ has $2$-dimensional (complex)
 invariant planes intersecting transversally;
 \item there are two first integrals (coming from
 two conserved quantities: energy and mass);
 \item some generic properties of $h$ and $\MM$.
\end{itemize}

\subsection{The dynamics close to the periodic orbits: a heuristic model} \label{sec:MainIdeasSaddle}
One of the crucial steps in analyzing the toy model $h(b)$ is the study of
the dynamics in a neighborhood of the periodic orbits $\TT_j$. Namely,
we want to analyze how points which lie close their stable invariant
manifold evolve under the flow until reaching points close to their unstable
one (see Figure \ref{fig:local-map}). As we have explained, these periodic orbits are
of mixed type (four eigenvalues are hyperbolic and the rest are elliptic).
Since in each plane $L_j$ dynamics is the same explained in the previous
section, the hyperbolic eigenvalues have multiplicity two and, therefore, are
equal to $\la,\la,-\la,-\la$ for some $\la>0$.
Since in this section serves exposition purposes
we let $\la=1$ and set the elliptic modes to zero.
\footnote{To be more precise near each saddle,
the elliptic directions remain almost constant and, since they will be taken
small enough, it turns out they do not make much influence in the dynamics
of hyperbolic components. Thus, to simplify the exposition, we set
the elliptic modes to zero
and study how the hyperbolic ones evolve. This implies that we only need
to study three modes $b_{j-1}$, $b_j$ and $b_{j+1}$. This analysis is performed
in Section \ref{sec:HypToyModel} in great detail.}

Essentially the study has three steps:
\begin{itemize}
\item Using conservation of $\MM$, make a simplectic
reduction so the periodic orbit $\TT_j$ becomes a fixed point.
\item Perform a normal form procedure to reduce the size of the higher
order non-resonant terms.
\item Analyze the dynamics of the new vector field and achieve a cancelation
for a local map.
\end{itemize}

The first step is performed in Section \ref{sec:SaddleMapFirstChanges}.
It leads to  a Hamiltonian of two degrees of freedom of the form
\[
 H(p,q)=p_1q_1+p_2q_2+H_4(q,p),
\]
where $H_4$ is a homogeneous polynomial of degree four.
The variables $(p_1,q_1)$ correspond to the variable $b_{j-1}$
after diagonalizing the saddle and the variables $(q_2,p_2)$
correspond to $b_{j+1}$.

%We want to obtain orbits which travel from points in a neighborhood
%of the connection between the saddle and the previous one $\TT_{j-1}$
%to points in a neighborhood of the connection with the following saddle
%$\TT_{j+1}$.

Fix a small $\sigma>0$. To study the local dynamics, it suffices
to analyze a map from  a section $\Sigma_-=\{q_1=\sigma,\
|p_1|,|q_2|,|p_2|\ll \sigma\}$, to a section
$\Sigma_+=\{p_2=\sigma,\ |p_1|,|q_1|,|q_2|\ll \sigma\}$
(see Figure \ref{fig:local-map}). Using rescaling assume $\sigma=1$.
This can change time by a fixed factor.

\begin{figure}[t]
  \centering
  \includegraphics[scale=0.5]{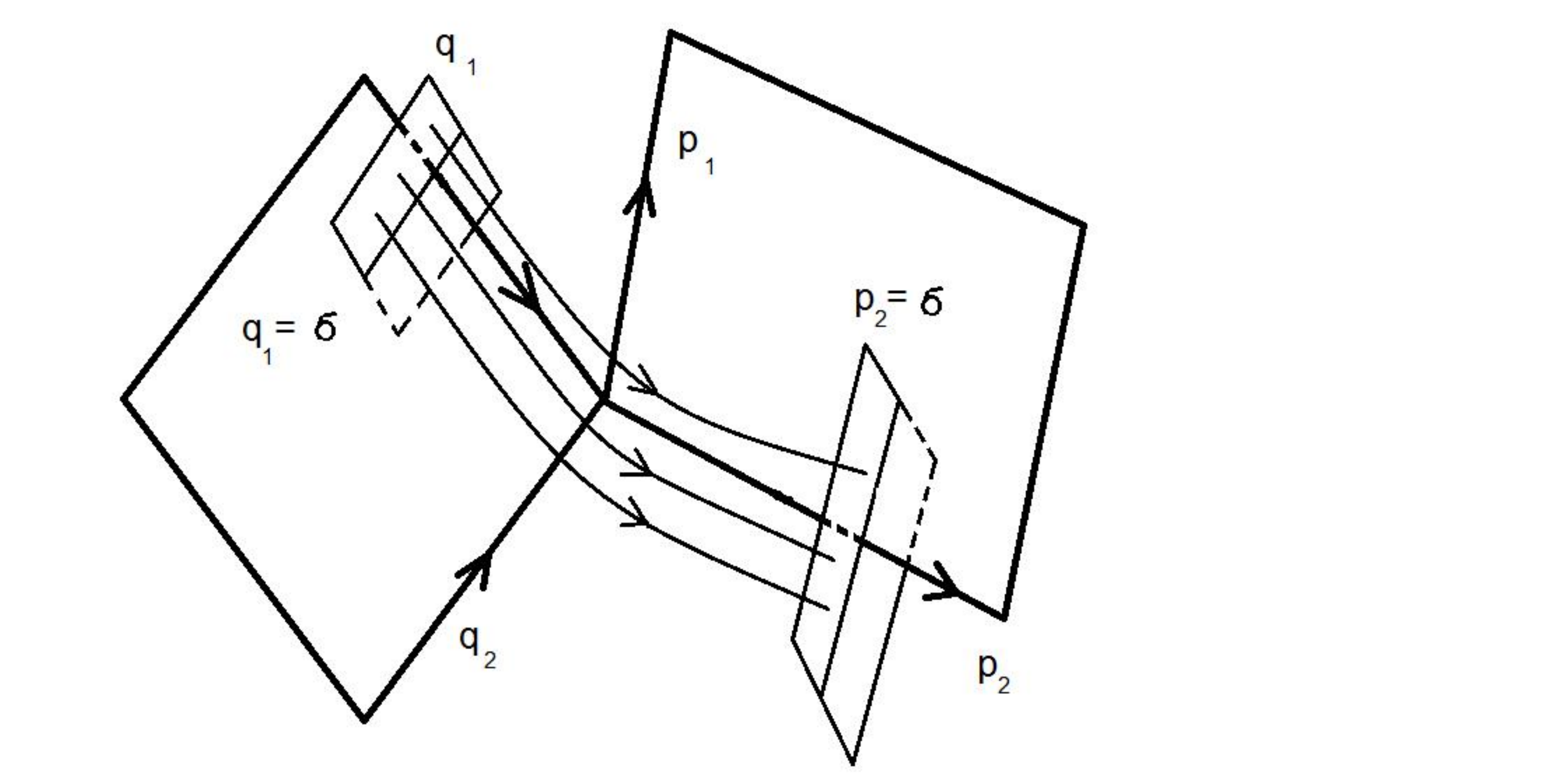}
  %,scale=0.5in
  %[width=4in]
  \caption{Local map}
  \label{fig:local-map}
\end{figure}

Since we are in a neighborhood of the origin, one would expect that
the dynamics of the system associated to this Hamiltonian is well
approximated by its first order, that is, by a linear equation.
%Let $r:[0,1]\to \RR_+$ be any non-negative function. If $H_4=0$,
Then, the solutions are just given by
\[
\begin{split}
p_1(t)&=p_1^0e^t,\,\,\,q_1(t)=q_1^0e^{-t}\\
p_2(t)&=p_2^0e^t,\,\,\,q_2(t)=q_2^0e^{-t}\\
\end{split}
\]
and then the local map $\BB_0$ from $U\subset \Sigma_-$ to
$\Sigma_+$ for this system  sends points
\[
\left(p_1^0,q_1^0,p_2^0,q_2^0\right)\sim \left(
\de,1,\sqrt{\de},\sqrt{\de}\right)
\]
to
\[
\BB_0\left(p_1^0,q_1^0,p_2^0,q_2^0\right)\sim
\left(\sqrt{\de},\sqrt{\de},1,\de\right),
\]
where $0<\de\ll 1$. Moreover, the travel time of orbits
by this map is always $T=-\ln\sqrt{\de}+\OO(1)$.

We will see that the image point {\it changes substantially} when
we add $H_4$ to the system, due to {\it both resonant and nonresonant terms}.
To exemplify this, we consider  \emph{a simplified model} which in fact contains
\emph{all the difficulties} that the true model has,
\begin{equation}\label{def:SimplifiedHam}
 H(p,q)=p_1q_1+p_2q_2+q_1^2p_2^2+p_1^2p_2^2.
\end{equation}
Since the term $p_1^2p_2^2$ is nonresonant, we first perform
one step of normal form $(x,y)=\Psi(p,q)$ (see Section \ref{sec:HypToyModel}
for details). It can be easily seen that the change $\Psi$ is of the form
\begin{equation}\label{def:Simplified:SizeNormalForm}
\Psi(p,q)=\left(p_1,q_1+\OO(p_1p_2^2),p_2,q_2+\OO(p_1^2p_2)\right)
\end{equation}
and, therefore, keeps the size of initial points of the form
\[
\left(p_1^0,q_1^0,p_2^0,q_2^0\right)\sim \left(
\de,1,\sqrt{\de},\sqrt{\de}\right).
\]
That is, $(x^0,y^0)=\Psi(p^0,q^0)$ satisfies
\[
\left(x_1^0,y_1^0,x_2^0,y_2^0\right)\sim \left(
\de,1,\sqrt{\de},\sqrt{\de}\right).
\]
The change to normal form leads to a Hamiltonian system of the form
\[
 H'(x,y)=x_1q_1+x_2y_2+ y_1^2x_2^2+\text{higher order terms}.
\]
Drop the higher order terms. Then, the solutions of the system associated
to this Hamiltonian can be computed explicitly and are given by
\[
 \begin{split}
  x_1&=x_1^0e^t+2y_1^0(x_2^0)^2te^t=\left(x_1^0+2y_1^0(x_2^0)^2t\right)e^t\\
  y_1&=y_1^0e^{-t}\\
  x_2&=x_2^0e^t\\
  y_2&=y_2^0e^{-t}-2(y_1^0)^2x_2^0te^t.\\
 \end{split}
\]
Thus, since the travel time is $t=-\ln\,\sqrt{\de}+\OO(1)$, it is clear
that {\it the nonlinear terms are bigger than the linear ones}, leading
to an image point of the form
\[
\left(x_1^f,y_1^f,x_2^f,y_2^f\right)\sim
\left(\sqrt{\de}\ln(1/\de), \sqrt{\de},
1,
\de\ln(1/\de)\right).
\]
Using \eqref{def:Simplified:SizeNormalForm}, in the original variables the image point of the map $\BB_1$
associated to Hamiltonian $H$  is of the form
\[
\BB_1\left(p_1^0,q_1^0,p_2^0,q_2^0\right)\sim
\left(\sqrt{\de}\ln(1/\de), \sqrt{\de},
1,\de\ln^2(1/\de)\right).
\]
We want to emphasize that the presence of these
{\it logarithmic terms is
a serious problem} we need to deal with. Recall that we need to travel
through $N-3$ saddles ($\TT_3\leadsto\TT_4\leadsto\ldots\leadsto\TT_{N-1}$).
Roughly speaking, this implies that we need to compose $N-4$ local maps.
Thanks to the symmetries, at each saddle we can consider a system of coordinates
such that the dynamics is essentially given by a Hamiltonian of
the form \eqref{def:SimplifiedHam}. Moreover, since at each local map
we gain some logarithms,  the initial points of the local map associated to the $j$ saddle are of the form
\[
\left(p_1^0,q_1^0,p_2^0,q_2^0\right)\sim
\left(\de \ln^{2^{j-1}}(1/\de),1,\sqrt{\de},\sqrt{\de}\right),
\]
which, thanks to \eqref{def:Simplified:SizeNormalForm}, in the normal form variables satisfy
\[
\left(x_1^0,y_1^0,x_2^0,y_2^0\right)\sim
\left(\de \ln^{2^{j-1}}(1/\de),1,\sqrt{\de},\sqrt{\de}\right).
\]
Then, proceeding as before, these points are mapped to points of the form
\[
\left(x_1^f,y_1^f,x_2^f,y_2^f\right)\sim
\left(\sqrt \de \ln^{2^{j-1}}(1/\de), \de^{1/2},
1,\de\ln(1/\de)\right)
\]
which in the original variables read
\[
\BB_1\left(p_1^0,q_1^0,p_2^0,q_2^0\right)\sim
\left(\sqrt \de \ln^{2^{j-1}}(1/\de), \sqrt \de, 1,\de\ln^{2^{j}}(1/\de)\right).
\]
That is, the amount of logarithms doubles at each step and thus grows exponentially.
This accumulation of  logarithmic terms  leads to \emph{very bad estimates}.
Indeed, to keep track of the orbit after $N-4$ local maps, we would need that
\[
 \de\ln^{2^{N-4}}(1/\de)\ll 1.
\]
Therefore, we would need to choose $\de$ extremely small with respect to $N$.

For example, if $\de \gtrsim C^{-\KK^{2a}} \sim C^{-\,2^{aN}}$ for some $C>0$
independent of $N$, then the above expression gives
\[
C^{-\,2^{aN}} (2^{aN}\,\ln C)^{2^{N-4}}\gg 1 \text{ for }a\le 1.
\]
In this case,  the constant $\la$ appearing in Theorem \ref{thm:Approximation} would need to satisfy  $\la \sim \de^{-b}$ for some $b>0$ and independent of $N$.
As a result, Theorem \ref{thm:ToyModelOrbit} gives a
diffusion time $T\sim \la^2 \KKK \ga N \ln 1/\de \gtrsim C^{\KK^2}$ (see formula \eqref{def:Time:Rescaled}).
%Since the time spent at each local map is $t=-\ln\,\sqrt{\de}+\OO(1)$,
Thus, choosing such a small $\de$ would lead to very bad estimates for
the diffusion time of Sobolev norms as we pointed out
in Remark \ref{remark:SlowerTime}.

To overcome this problem, we modify slightly the initial
conditions. Notice that if we choose $x_1^0$ such that
\[
 x_1^0-2y_1^0(x_2^0)^2\ln\,\sqrt{\de}=0,
\]
we obtain that at the end $x_1^f\sim\sqrt{\de}$ and thus we avoid
the logarithmic term. This cancelation \emph{will be crucial} in our proof.
If we restrict $x_1^0$ to this set, we are taking $x_1^0\sim \de\ln(1/\de)$
and therefore we will be sending points
\[
\left(x_1^0,y_1^0,x_2^0,y_2^0\right)\sim
\left(\de\ln(1/\de),1,\sqrt{\de},\sqrt{\de}\right),
\]
to points
\[
\left(x_1^f,y_1^f,x_2^f,y_2^f\right)\sim
\left(\sqrt{\de},\sqrt{\de},1,\de\ln(1/\de)\right).
\]
The map will keep the same form expressed in the original variables, and,
therefore, we will avoid having increasing separation from the invariant manifolds.

\subsection{Outline of the Proof}
\begin{itemize}
\item Find symplectic coordinates near the origin in $\ell^1$, where
the original Hamiltonian $\HH$ simplifies (see Theorem \ref{thm:NormalForm}).
Namely,  $\HH \circ \Gamma=\DDD+\wt \GG+\RRR$, where $\DDD$ is a quadratic Hamiltonian, $\wt\GG$ is of degree
four and only contains resonant terms, and $\RRR$ is smaller.

\item  Dynamics of $\DDD+\wt \GG$ has invariant finite-dimensional
subspaces, which give rise to a simpler (no simple!) finite-dimensional
Hamiltonian $h(b)$ given by (\ref{def:Hamiltonian}). In terminology of
\cite{CollianderKSTT10} this Hamiltonian defines the {\it Toy Model}. In Theorem  \ref{thm:ToyModelOrbit} we obtain orbits of the toy model which have  transfer of energy.

\item We show that are solutions of the system associated to $\HH$ which are close to those of the toy model for long enough time (Theorem \ref{thm:Approximation}). These orbits undergo the wanted growth of the Sobolev norm.

\item The proof of Theorem \ref{thm:ToyModelOrbit} occupies most of the paper. Theorems \ref{thm:NormalForm} and \ref{thm:Approximation} are deferred to Appendices \ref{app:NormalForm} and \ref{app:Approx} respectively. Now we describe the plan of the proof of Theorem \ref{thm:ToyModelOrbit}.

\item Following \cite{CollianderKSTT10} we detect a collection of
periodic orbits $\{\TT_j\}_{j=1}^{N-1}$ of $h(b)$, defined in
(\ref{def:periodic-orbits}), and heteroclinic orbits $\{\ga_j\}_{j=1}^{N-2}$
connecting them (see (\ref{def:heteroclinic})).

The whole proof consists in a careful analysis of dynamics near
the union of these periodic orbits and their connecting orbits.
Our analysis naturally splits into
\begin{itemize}
\item {\it local dynamics} near periodic orbits $\{\TT_j\}_{j=1}^{N-1}$ and
\item {\it global dynamics} near heteroclinic orbits $\{\ga_j\}_{j=1}^{N-2}$.
\end{itemize}
\item More formally, Theorem \ref{thm:ToyModelOrbit} follows from
Theorem \ref{theorem:iterative}. The latter Theorem in turn follows
from Lemmas \ref{lemma:iterative:saddle} and \ref{lemma:iterative:hetero}.

\item The Local Lemma \ref{lemma:iterative:saddle} provides
refined information about local behavior near periodic orbits $\{\TT_j\}_j$
with quantitative estimates.

\item Global Lemma \ref{lemma:iterative:hetero} provides
refined information about local behavior near heteroclinic orbits
from (\ref{def:heteroclinic}) with quantitative estimates.

\item The proof of Local Lemma \ref{lemma:iterative:saddle} consists of
several steps. As we have explained in Section \ref{sec:FeaturesModel}, the periodic orbits $\{\TT_j\}_j$ have mixed type.
Namely, in some directions the local behavior is hyperbolic, while in
others it is elliptic. It turns out that the closer orbits under
investigation pass to the periodic orbits $\{\TT_j\}_j$, the more
decoupled (direct product-like) behavior they have.

\item In Section \ref{sec:HypToyModel} we set all the elliptic
variables zero and study the ($4$-dimensional) hyperbolic Toy Model.

\item In Section \ref{sec:FullSystem} we use these results
to deal with the full hyperbolic-elliptic system and prove Lemma
\ref{lemma:iterative:saddle}.

\item In Section \ref{sec:ProofHeteroMap} we prove
Global Lemma \ref{lemma:iterative:hetero}. As we pointed out, this implies Theorem \ref{theorem:iterative},
which in turn, implies Theorem \ref{thm:ToyModelOrbit}.

\item Combining this result with Theorem \ref{thm:NormalForm}, proved in Appendix \ref{app:NormalForm}, and Theorem \ref{thm:Approximation} proved in Appendix \ref{app:Approx}, we complete the proof of the main result (Theorem \ref{thm:main}).
\end{itemize}
We summarize this in the following diagram:
\be
\beal
\qquad  \boxed{\text{ Theorem \ref{thm:main} }}  \qquad \qquad  \qquad \qquad \\
  \Uparrow \quad  \qquad \qquad  \qquad \qquad \qquad
\\
\overbrace{\boxed{\text{ Theorem \ref{thm:NormalForm}} }+\quad
\underbrace{\boxed{\text{ Theorem \ref{thm:ToyModelOrbit} }}}\quad +\
\boxed{\text{ Theorem \ref{thm:Approximation} }}}\ \\
 \Uparrow \quad  \qquad \qquad  \qquad \qquad \qquad
\\
 \qquad
\boxed{\text{ Theorem \ref{theorem:iterative} }}  \qquad \qquad  \qquad \qquad \\
  \Uparrow \quad  \qquad \qquad  \qquad \qquad \qquad
\\
\overbrace{\boxed{\text{ Local Lemma \ref{lemma:iterative:saddle}} }\ \ +\
\boxed{\text{ Global Lemma \ref{lemma:iterative:hetero} }}}\qquad
\enal
\ee

\subsection{Major ingredients of the proof}
We summarize here the new set of tools that we apply to the problem compared to  \cite{CollianderKSTT10}.
\begin{itemize}
 \item In Theorem \ref{thm:NormalForm}, we use a standard normal form (e.g. see \cite{KuksinP96}).
\item Theorem \ref{thm:ToyModelOrbit} requires several new ideas:
\begin{itemize}
\item Finitely smooth resonant normal form for hyperbolic saddles \cite{BronsteinK94}.
\item Shilnikov boundary value problem \cite{Shilnikov67} to study the local behavior close to the periodic orbits $\TT_j$.
\item As we explained for the model case in Section \ref{sec:MainIdeasSaddle}, to control the dynamics of the toy model we need a peculiar cancellation (see Section \ref{sec:HypToyModel}).
\item To have cancellations at each stage, we need to establish local product structure for the orbits we are interested in (see Definition \ref{def:ProductLikeSet}).
\end{itemize}
\item Due to
the good control of the solutions of the toy model, we are able to approximate the solutions of the original systems with the ones of the toy model for longer time compared with  \cite{CollianderKSTT10} (see Theorem \ref{thm:Approximation}). To achieve this, we also modify the set $\Lambda$ (see condition $6_\Lambda$).  This modification
allows to slow down spreading outside $\Lambda$.
\end{itemize}

\section{The three key theorems}\label{sec:SketchProof}

We start the proof analyzing the infinite system
of equations which describe the behavior of Fourier coefficients.
Namely, consider the Fourier series of $u$,
\[
 u(t,x)=\sum_{n\in\ZZ^2}a_n(t) e^{inx}, \qquad
 a_n(t):=\hat u(t,n).
\]
Therefore, the equation \eqref{def:NLS} becomes an infinite system of
equations for $\{a_n\}_{n\in\ZZ^2}$, which are given by
\begin{equation}\label{eq:NLSForFourierCoefs}
 -i\dot a_n=|n|^2 a_n+\sum_{\substack{n_1,n_2,n_3\in \ZZ^2\\n_1-n_2+n_3=n}}
 a_{n_1}\ol{a_{n_2}}a_{n_3}.
\end{equation}
Note that this equation is Hamiltonian. Indeed, it can be written as
\[
 \dot a_n=2i \pa_{\ \ol{a_n}}\ \HH(a, \ol a),
\]
where
\begin{equation}\label{def:HamForFourier}
\HH (a,\ol a)=\DDD (a,\ol a)+\GG  (a,\ol a)
\end{equation}
with
\begin{align*}
 \DDD (a,\ol a)&=\frac{1}{2}\sum_{n\in\ZZ^2}|n|^2|a_n|^2\\
\GG  (a,\ol a)&=\frac{1}{4}\sum_{\substack{n_1,n_2,n_3,n_4\in \ZZ^2\\
n_1-n_2+n_3=n_4}}a_{n_1}\ol{a_{n_2}}a_{n_3}\ol{a_{n_4}}.
\end{align*}
We will study equation \eqref{eq:NLSForFourierCoefs} in
a family of Banach spaces: all $H^s$-Sobolev spaces with $s> 1$
as well as in the $\ell^1$-space. The $\ell^1$ space is defined as
\[
 \ell^1=\left\{a:\ZZ^2\rightarrow \CC: \|a\|_{\ell^1}=
 \sum_{n\in\ZZ^2}|a_n|<\infty\right\}.
\]
Note that, $\ell^1$ is a Banach algebra with respect to the convolution
product. Namely, if $a,b\in \ell^1$ its convolution product $a\ast b$,
which is defined by
\[
( a\ast b)_n=\sum_{n_1+n_2=n} a_{n_1}b_{n_2}
\]
satisfies
\[
 \|a\ast b\|_{\ell^1}\leq \|a\|_{\ell^1}\| b\|_{\ell^1}.
\]

Finally, let us point out that the $L^2$-norm conservation of \eqref{def:NLS},
becomes now conservation of the $\ell^2$-norm of $a$, defined as above.
Namely, we have that $  \|a(t)\|_{\ell^2}=\|a(0)\|_{\ell^2}$ for all $t\in\RR$.

We want to study the evolution of certain solutions of  equation
\eqref{eq:NLSForFourierCoefs}, which will be small in the $\ell^1$ norm.
Now we make an outline of the proof.

The first step is to find out which terms make the biggest
contribution to this evolution. To this end, we perform one step
of normal form and bound the remainder in the $\ell^1$-norm.

\begin{theorem}\label{thm:NormalForm}
For the Hamiltonian $\HH$ in \eqref{def:HamForFourier} there exists a symplectic
change of coordinates $a=\Gamma(\alpha)$ in a neighborhood of $0$ in $\ell^1$
which takes it into its Birkhoff normal form up to order four,
that is,
\[
 \HH\circ\Gamma=\DDD+\wt \GG+\RRR,
\]
where $\wt\GG$ only contains  resonant terms, namely
\[
 \wt \GG(\al, \ol\al)=\frac{1}{4}\sum_{\substack{n_1,n_2,n_3,n_4\in
\ZZ^2\\n_1-n_2+n_3=n_4\\|n_1|^2-|n_2|^2+|n_3|^2=|n_4|^2}}\al_{n_1}\ol{\al_{n_2}}
\al_{n_3}\ol{\al_{n_4}}
\]
and $X_\RRR$, the vector field associated to
the Hamiltonian $\RRR$, satisfies
\[
 \|\XX_\RRR\|_{\ell^1}\leq \OO\left(\|\al\|_{\ell^1}^5\right).
\]
Moreover, the change $\Gamma$ satisfies
\[
 \left\|\Gamma-\mathrm{Id}\right\|_{\ell^1}\leq
\OO\left(\|\al\|_{\ell^1}^3\right).
\]
\end{theorem}
The proof of this theorem is postponed to Appendix \ref{app:NormalForm}.

Once we
perform one step of normal form, we have a new vector field
\begin{equation}\label{eq:InfiniteODEAfterNormalForm}
- i \dot \alpha_n = |n|^2 \alpha_n +
\sum_{(n_1,n_2,n_3)\in \AAA_0 (n)} \al_{n_1} \overline { \al_{n_2}}\al_{n_3}+
\partial_{\overline \al_n} \RRR,
\end{equation}
where
\begin{equation}\label{def:ResonantLatticeBeforeGauge}
\begin{split}
 \AAA_0(n)=\Big\{&(n_1,n_2,n_3)\in\left(\ZZ^2\right)^3:\,n_1-n_2+n_3=n,\\
&|n_1|^2-|n_2|^2+|n_3|^2=|n|^2\Big\}.
\end{split}
\end{equation}
As a first step, we focus our attention to
the degree 4 truncation of it, which will give the main contribution to the
dynamics. Namely, we consider the Hamiltonian
\[
 \HH'=\DDD+\wt\GG,
\]
which has associated equations
\begin{equation}\label{eq:InftyODE:AfterNF}
 -i\dot\al_n=|n|^2\al_n+\sum_{(n_1,n_2,n_3)\in\AAA_0(n)}\al_{n_1}\ol{\al_{n_2}}
\al_{n_3}.
\end{equation}

Note that the $\ell^2$-norm of $\al$ is a first integral of this system
as well as for \eqref{eq:NLSForFourierCoefs} and \eqref{eq:InfiniteODEAfterNormalForm}.
Namely,
\[
 \|\al(t)\|_{\ell^2}=\|\al(0)\|_{\ell^2} \,\,\,\text{for all }t\in\RR.
\]
Then, to study the dynamics of $\al$ close to the origin (in
the $\ell^1$-norm) we remove its linear terms using the variation of
constants formula. Moreover, following  \cite{CollianderKSTT10}, we also
remove certain cubic terms using the gauge freedom of equation \eqref{def:NLS}. To this end, we make the change of coordinates
\begin{equation}\label{eq:InftyODE:VariationConstants}
 \al_n=\beta_n e^{i\left(G+|n|^2\right)t},
\end{equation}
where $G>0$ is a constant
to be determined. The equations for $\beta$ read
\[
 -i\dot\beta_n= G\beta_n+\sum_{(n_1,n_2,n_3)\in\AAA_0(n)
}\beta_{n_1}\ol{\beta_{n_2}}\beta_{n_3}.
\]
Choosing $G$ properly we can remove certain terms in the sum.
Indeed, we split the sum as
\[
\sum_{(n_1,n_2,n_3)\in\AAA_0(n)}=\sum_{\substack{(n_1,n_2,n_3)\in\AAA_0(n)\\n_1,
n_3\neq
n}}+\sum_{\substack{(n_1,n_2,n_3)\in\AAA_0(n)\\n_1=n}}+\sum_{\substack{(n_1,n_2,
n_3)\in\AAA_0(n)\\n_3=n}}-\sum_{\substack{(n_1,n_2,n_3)\in\AAA_0(n)\\n_1=n_3=n}}
\]
The last sum is just one term, which is given by $-\beta_n |\beta_n|^2$. The
second and third sums, are in fact single sums and each of them is given by
\[
 \beta_n \sum_{k\in\ZZ^2}|\beta_k|^2=\beta_n \|\beta\|_{\ell^2}^2.
\]
Recall that both \eqref{eq:InftyODE:AfterNF} and \eqref{eq:InftyODE:VariationConstants}
preserve the $\ell^2$-norm. Therefore,
taking $G=-2\|\al\|^2_{\ell^2}=-2\|\beta\|^2_{\ell^2}$, we can remove these two
terms. Thus, with this choice, we obtain the equation for $\beta$, which reads
\begin{equation}\label{eq:InftyODE:AfterVariation}
 -i\dot\bet_n=-\bet_n|\bet_n|^2+\sum_{n_1,n_2,n_3\in\AAA(n)}\bet_{n_1}\ol{\bet_{
n_2}}\bet_{n_3}
\end{equation}
where
\[
\begin{split}
 \AAA(n)=\Big\{&(n_1,n_2,n_3)\in\left(\ZZ^2\right)^3:\,n_1-n_2+n_3=n\\
&|n_1|^2-|n_2|^2+|n_3|^2=|n|^2, n_1\neq n, n_3\neq n\Big\}.
\end{split}
\]
We define also the set of all resonant frequencies as
\[
 \AAA=\Big\{(n_1,n_2,n_3,n_4)\in\left(\ZZ^2\right)^4:\,(n_1,n_2,
n_3)\in\AAA(n_4)\Big\}.
\]
Note that if $(n_1,n_2,n_3,n_4)\in\AAA$, then the four points  form a
rectangle in $\ZZ^2$.

We reduce this system to
a finite-dimensional one, which corresponds to an invariant
finite-dimensional plane. To this end, we  consider a set $\Lambda\subset\ZZ^2$
such that the corresponding harmonics do not interact to the harmonics outside
of $\Lambda$. Moreover, we  obtain a set $\Lambda$ such that the harmonics in
$\Lambda$ interact in a very particular way. This set was constructed in
\cite{CollianderKSTT10}. We explain now its construction and impose
 an additional condition on $\Lambda$ from \cite{CollianderKSTT10}.

Fix $N\gg 1$.
Following \cite{CollianderKSTT10} we define a set
$\La\subset \ZZ^2$ consisting of $N$ pairwise disjoint \emph{generations}:
\[
 \Lambda=\Lambda_1\cup\ldots\cup\Lambda_N.
\]
Define a \emph{nuclear family} to be a rectangle $(n_1,n_2,n_3,n_4)\in\AAA$,
such that $n_1$ and $n_3$ (known as the \emph{parents}) belong to a generation
$\Lambda_j$ and $n_2$ and $n_4$ (known as the \emph{children}) live in the next
generation $\Lambda_{j+1}$. Note that if $(n_1,n_2,n_3,n_4)$ is a nuclear
family, then so are  $(n_1,n_4,n_3,n_2)$,  $(n_3,n_2,n_1,n_4)$ and
$(n_3,n_4,n_1,n_2)$. These families are called trivial permutations of the
family $(n_1,n_2,n_3,n_4)$.

The conditions to impose to the set $\La$ are
\begin{description}
\item[$1_\Lambda$\ {\it Closure}]\
If $n_1, n_2, n_3\in\Lambda$ and $(n_1,n_2,n_3)\in\AAA(n)$,
then $n\in\Lambda$.
    In other words, if three vertices of a rectangle are in $\La$ so is
    the last fourth one.
\item[$2_\Lambda$\ {\it Existence and uniqueness of spouse and children}]\
For any $1\leq j <N$ and any $n_1\in\Lambda_j$, there exists a unique nuclear
family $(n_1,n_2,n_3,n_4)$ (up to trivial permutations) such that $n_1$ is
a parent of this family. In particular, each $n_1\in\Lambda_j$ has
a unique spouse $n_3\in\Lambda_j$ and has two unique children
$n_2,n_4\in\Lambda_{j+1}$ (up to permutation).

\item[$3_\Lambda$\ {\it Existence and uniqueness of sibling and parents}]\
For any $1\leq j <N$ and any $n_2\in\Lambda_{j+1}$, there exists a unique
nuclear family $(n_1,n_2,n_3,n_4)$ (up to trivial permutations) such that
$n_2$ is a child of this family. In particular each $n_2\in\Lambda_{j+1}$
has a unique sibling $n_4\in\Lambda_{j+1}$ and two unique parents
$n_1,n_3\in\Lambda_{j}$ (up to permutation).

\item[$4_\Lambda$\ {\it Nondegeneracy}]\
The sibling of a frequency $n$ never equal to its spouse.

\item[$5_\Lambda$\ {\it Faithfulness}]\
Apart from the nuclear families, $\Lambda$ does not contain any other rectangle.
\end{description}
These
are the conditions imposed on $\Lambda$ in \cite{CollianderKSTT10}.
We will impose an additional condition:
\begin{description}
\item[$6_\Lambda$\ {\it No spreading condition}]\
Let us consider $n\not\in \Lambda$. Then, $n$ is vertex of
at most two rectangles having two vertices in $\Lambda$ and two vertices out of
$\Lambda$.
\end{description}

\begin{proposition}\label{thm:SetLambda}
Let $\KK\gg 1$. Then, there exists
$N\gg 1$ large and a set $\Lambda\subset\ZZ^2$, with
\[
 \Lambda=\Lambda_1\cup\ldots\cup\Lambda_N,
\]
which satisfies  conditions $1_\Lambda$ -- $6_\Lambda$ and also
\begin{equation}\label{def:Growth}
 \frac{\sum_{n\in\Lambda_{N-1}}|n|^{2s}}{\sum_{n\in\Lambda_3}|n|^{2s}}\geq
 \dfrac 12 2^{(s-1)(N-4)}\ge \KK^2.
\end{equation}
Moreover, given any $R>0$
(which may depend on $\KK$), we can ensure that each generation
$\Lambda_j$ has $2^{N-1}$ disjoint frequencies $n$ satisfying $|n|\geq R$.
\end{proposition}

The  proof of Proposition 2.1 from \cite{CollianderKSTT10} applies.
We prove a quantitative version of this proposition in Appendix
\ref{app:SmallSobolev}.

We use the set $\Lambda$ to obtain a finite dimensional dynamical system
(of high dimension)
approximating \eqref{eq:InftyODE:AfterVariation}. To this end, let us first
note that, by Property $1_\Lambda$,  the manifold
\[
 M=\left\{\beta\in\CC^{\ZZ^2}: \beta_n=0\,\,\text{for all }
n\not\in\Lambda\right\}
\]
is invariant by the flow associated to \eqref{eq:InftyODE:AfterVariation} and is
finite dimensional.  Indeed, by Proposition \ref{thm:SetLambda} its dimension is
$N2^{N-1}$.  Equation  \eqref{eq:InftyODE:AfterVariation}  restricted to $M$
reads as follows. For each $n\in \La$ we have
\begin{equation}\label{eq:InftyODE:FirstFiniteReduction}
 -i\dot\bet_n=-\bet_n|\bet_n|^2+2
\beta_{n_{\mathrm{child}_1}}\beta_{n_{\mathrm{child}_2}}\ol{\beta_{n_\mathrm{spouse}}}+2
\beta_{n_{\mathrm{parent}_1}}\beta_{n_{\mathrm{parent}_2}}\ol{\beta_{n_\mathrm{sibling}}}
\end{equation}
Indeed, presence of parents, children, and the sibling are guaranteed by $2_\La$
and $3_\La$. Note, that in the first and last generations, the parents and
children are set to zero respectively. In fact, $M$ has a submanifold of
considerably lower dimension which is also invariant.

\begin{corollary}\label{coro:Invariant}(cf. \cite{CollianderKSTT10})
Consider the subspace
\[
 \wt M=\left\{\beta\in M: \beta_{n_1}=\beta_{n_2}\,\,\text{for all
}n_1,n_2\in\Lambda_j\,\,\text{for some }j\right\},
\]
where all the members of a generation take the same value. Then, $\wt M$ is invariant under the flow associated to \eqref{eq:InftyODE:FirstFiniteReduction}.
\end{corollary}

The dimension of  $\wt M$ is equal to the number of generations, namely $N$.
To define equation
\eqref{eq:InftyODE:FirstFiniteReduction} restricted to $\wt M$, let us define
\begin{equation}\label{def:ChangeToToyModel}
 b_j=\beta_n\,\,\,\text{ for any }n\in\Lambda_j.
\end{equation}
Then, \eqref{eq:InftyODE:FirstFiniteReduction}  restricted to $\wt M$ becomes
\begin{equation}\label{def:model}
 \dot  b_j=-ib_j^2\ol b_j+2i \ol
b_j\left(b_{j-1}^2+b_{j+1}^2\right),\,\,j=0,\ldots N,
\end{equation}
which is a Hamiltonian system with  respect to the Hamiltonian
\begin{equation}\label{def:Hamiltonian}
 h(b):=\frac{1}{4}\sum_j |b_j|^4-
 \frac{1}{2}\sum_j \left(\ol b_j^2b_{j-1}^2+b_j^2\ol b_{j-1}^2\right)
\end{equation}
and the  symplectic form $\Omega=\frac{i}{2}db_j\wedge d\ol b_j$.

\begin{theorem}\label{thm:ToyModelOrbit}
Fix a large $\gamma\gg 1$.  Then for any large enough $N$ and
$\de=e^{-\ga N}$, there exists an orbit of system \eqref{def:model},
$\nu>0$ and $T_0>0$ such that
\[
\begin{split}
 |b_3(0)|&>1-\de^\nu\\
|b_j(0)|&< \de^\nu\qquad\text{ for }j\neq 3
\end{split}
\qquad \text{ and }\qquad
\begin{split}
 |b_{N-1}(T_0)|&>1-\de^\nu\\
|b_j(T_0)|&<\de^\nu \qquad\text{ for }j\neq N-1.
\end{split}
\]
Moreover, there exists a constant $\KKK>0$
independent of  $N$ such that $T_0$ satisfies
\begin{equation}\label{def:Time:ToyModel}
 0<T_0<\KKK N \ln \left(\frac 1 \de \right)=\KKK\,\ga\,N^2.
\end{equation}
\end{theorem}

\begin{remark} An analog of this proposition also holds for
some smaller $\de$, e.g. $\de=C^{-\KK^2}$. This is
related to Remark \ref{remark:SlowerTime} about
time of diffusion without cancelations.
\end{remark}

Using  \eqref{def:ChangeToToyModel}, Theorem \ref{thm:ToyModelOrbit}
gives an orbit for equation \eqref{eq:InftyODE:AfterVariation}.
Moreover, both equations \eqref{eq:InftyODE:AfterVariation} and
\eqref{def:model} are invariant under certain rescaling.
Indeed if $b(t)$ is a solution of \eqref{def:model},
\begin{equation}\label{def:Rescaling}
b^\lambda(t)=\lambda^{-1}b\left(\lambda^{-2}t\right)
\end{equation}
is  a solution of the same equation.
By Theorem \ref{thm:ToyModelOrbit} duration of this solution in time is
\begin{equation}\label{def:Time:Rescaled}
 T=\la^2 T_0\le \la^2\,\KKK\,\ga\,N^2,
\end{equation}
where $T_0$ is the time obtained in Theorem \ref{thm:ToyModelOrbit},
which satisfies \eqref{def:Time:ToyModel}.

We will see that, modulo a rotation of the modes
(see \eqref{eq:InftyODE:VariationConstants}), there is a solution
of equation \eqref{eq:InfiniteODEAfterNormalForm}
which is close to the orbit $\beta^\la$ of
\eqref{eq:InftyODE:AfterVariation} defined as
\begin{equation}\label{def:RescaledApproxOrbit}
\begin{split}
\beta^\la_n(t)&=\lambda^{-1}b_j\left(\lambda^{-2}t\right)
\,\,\,\text{ for each }\ n\in\Lambda_j\\
\beta^\la_n(t)&=0\,\,\,\text{ for each }\ n\not\in\Lambda.
\end{split}
\end{equation}

To have the original system being well
approximated by the truncated system, we need that $\lambda$
is
large enough. Then the cubic terms in
\eqref{eq:InfiniteODEAfterNormalForm} dominate over
the quintic ones. Nevertheless, the bigger $\la$ is, the slower
the instability
time by
\eqref{def:Time:Rescaled}.
Thus, we  look for the smallest $\la$ (with respect to $N$) for
which the following approximation theorem
applies.

\begin{theorem}\label{thm:Approximation}
Let $\alpha(t)=\{\alpha_n(t)\}_{n\in \ZZ^2}$
be the solution of \eqref{eq:InfiniteODEAfterNormalForm}, $\bet^\la(t)=\{\bet^\la_n(t)\}_{n\in \ZZ^2}$ be the solution of
(\ref{eq:InftyODE:AfterVariation}) given
by \eqref{def:RescaledApproxOrbit}
and $T$ be the time defined in \eqref{def:Time:Rescaled}.
Suppose  $\mathrm{supp}\,\alpha(0)\subset \La$ and $\alpha(0)=\bet^\la(0)$.
Then, there exist a constant $\kk>0$ independent of $N$ and $\ga$ such that, for
\begin{equation}\label{def:LambdaOfN}
 \la=e^{\kk\ga N},
\end{equation}
and $0<t<T$ we have
\begin{equation}\label{eq:Approx:BoundDiff}
\sum_{n\in \ZZ^2}
\left|\alpha_n(t)-e^{i(G+|n|^2)t}\beta^\la_n(t)\right|
 \le \la^{-2},
\end{equation}
where $G=-2\|\al(0)\|_{\ell^2}^2$.
\end{theorem}

%\begin{theorem}\label{thm:Approximation}
%Let $\alpha(t)=\{\alpha_n(t)\}_{n\in \ZZ^2}$ be a solution to
%\[
%- i \dot \alpha_n = |n|^2 \alpha_n +
%\sum_{(n_1,n_2,n_3)\in \AAA_0 (n)} \al_{n_1} \overline { \al_{n_2}}\al_{n_3}+
%\partial_{\overline \al_n} \RRR',
%\]
%where $\RRR'=O(\|\al(0)\|^6_{\ell^1})$.
%Let $\bet(t)=\{\bet_n(t)\}_{n\in \ZZ^2}$ be a solution to
%(\ref{eq:InftyODE:AfterVariation}).
%Suppose supp $\alpha(0)\subset \La$, $\alpha(0)=\bet(0)$,
%and $\la=\|\al(0)\|_{\ell^\infty}^{-1}$. Then for $T$ satisfying
%$Te^{-\la^2 T}\la^6<1$ and $0<t<T$ we have
%\[
%\sum_{n\in \ZZ^2}
%|\alpha_n(t)-e^{i(G+|n|^2)t}\beta_n(t)|
% \le \la^{-2}.
%\]
%\end{theorem}

%From the results stated in the three key theorems, that is, Theorems \ref{thm:NormalForm}, %\ref{thm:ToyModelOrbit} and \ref{thm:Approximation},
%we are ready to finish the  proof Theorem \ref{thm:main}.

Using the three key theorems:
Theorems \ref{thm:NormalForm}, \ref{thm:ToyModelOrbit} and \ref{thm:Approximation}
we complete the proof of Theorem \ref{thm:main}.

\begin{proof}[Proof of Theorem \ref{thm:main}]
Using the change of variables $\Gamma$ obtained in Theorem \ref{thm:NormalForm}, from the solution $\al$  obtained in Theorem \ref{thm:Approximation} we define  $a=\Gamma (\al)$, which is a solution of system \eqref{eq:NLSForFourierCoefs}. We show that this orbit has the properties stated in Theorem \ref{thm:main}.

To compute the growth of Sobolev norm of this orbit $a$, we use the notation
\begin{equation}\label{def:Sums}
 S_j=\sum_{n\in\Lambda_j}|n|^{2s}
 \text{ for }j=1,\dots, N-1.
\end{equation}
To estimate the mass of our solution recall that
$2^{N-1}=\sum_{n\in\Lambda_j}1=|\Lambda_j|$. We want to prove that
\[
\frac{\|a(T)\|_{H^s}}{\|a(0)\|_{H^s}}\gtrsim \KK
\]
and estimate the mass $\|a(0)\|_{L^2}$ of the solution.
To this end, we start by bounding $\|a(T)\|_{H^s}$ in terms of
$S_{N-1}$. Since
\[
\left \|a(T)\right\|^2_{H^s}\geq \sum_{n\in\Lambda_{N-1}}|n|^{2s} \left|a_n(T)\right|^2\geq S_{N-1}\inf_{n\in\Lambda_{N-1}}\left|a_n(T)\right|^2,
\]
it is enough to obtain a lower bound for $\left|a_n(T)\right|$ with $n\in\Lambda_{N-1}$. Using the results of Theorems \ref{thm:NormalForm} and \ref{thm:Approximation}, we obtain
\begin{equation}\label{eq:ComputationFinalSobolev}
\begin{split}
 \left|a_n(T)\right|\geq & \left|\al_n(T)\right|-\left|\Gamma_n (\al)(T)-\al_n(T)\right|\\
\geq &\left|\beta^\la_n(T) e^{i\left(|n|^2+G\right)T}\right|- \left|\al_n(T)-\beta^\la_n(T) e^{i\left(|n|^2+G\right)T}\right|\\
&- \left|\Gamma_n (\al)(T)-\al_n(T)\right|.
\end{split}
\end{equation}
We need to obtain a lower bound for the first term of the right hand side and upper bounds for the second and third ones. Indeed, using the definition of $\beta^\la$ in \eqref{def:RescaledApproxOrbit} and the results in Theorem \ref{thm:ToyModelOrbit} we have that for $n\in\Lambda_{N-1}$,
\[
\left|\beta_n^\la(T) \right|^2=   \la^{-2}\left| b_{N-1}(T_0)\right|^2 \geq \frac{3}{4}\la^{-2},
\]
(the relation between $T$ and $T_0$ is established in \eqref{def:Time:Rescaled}).

For the second term in the right hand side of \eqref{eq:ComputationFinalSobolev}, it is enough to use Theorem \ref{thm:Approximation} to obtain,
\[
\left|\al_n(T)-\beta^\la_n(T) e^{i\left(|n|^2+G\right)T}\right|^2\leq \left(\sum_{n\in\ZZ^2}\left|\al_n(T)-\beta^\la_n(T) e^{i\left(|n|^2+G\right)T}\right|\right)^2 \leq \frac{\la^{-2}}{8}.
\]
For the  lower bound of the third term, we use the bound for $\Gamma-\mathrm{Id}$ given in Theorem \ref{thm:NormalForm}. Then,
\[
\left|\Gamma_n (\al)(T)-\al_n(T)\right|^2 \leq  \|\Gamma(\al)-\al\|_{\ell^1}^2\leq  \frac{\la^{-2}}{8}.
\]
Thus, we can conclude that
\begin{equation}\label{eq:GrowthSobolev:Final}
 \left \|\al(T)\right\|^2_{H^s}\geq \frac{\la^{-2}}{2}S_{N-1}.
\end{equation}
Now we prove that
\begin{equation}\label{eq:GrowthSobolev:Initial}
 \left \|a(0)\right\|^2_{H^s}\lesssim \la^{-2}S_3
 \qquad \text{ and } \qquad \left \|a(0)\right\|^2_{L^2}\lesssim \la^{-2}\,2^N.
\end{equation}
By the definition of $\la$ in (\ref{def:LambdaOfN}), the second inequality implies that the mass of $a(0)$
is small.  On the contrary, the first inequality does not imply that
the $H^s$-norm of $a(0)$ is small.
 As a matter of fact $S_3$ is large\footnote{As pointed out
to us by Terence Tao.}.

To prove the first inequality of \eqref{eq:GrowthSobolev:Initial}, let us point out that
\[
\left\|a(0)\right\|^2_{H^s}\leq \sum_{n\in\ZZ^2}|n|^{2s} \left|\al_n(0)+\left(\Gamma_n(\al(0)-\al_n(0)\right)\right|^2.
\]
We first bound $\left\|\al(0)\right\|^2_{H^s}$. To this end, let us recall that $\mathrm{supp }\ \al=\Lambda$. Then,
\[
\left\|\al(0)\right\|^2_{H^s}=\sum_{n\in\Lambda}|n|^{2s} \left|\al_n(0)\right|^2.
\]
Using Theorem \ref{thm:Approximation}, we have that
\begin{equation}\label{eq:GrowthSobolev:beta}
\left\|\al(0)\right\|^2_{H^s}\leq \sum_{n\in\Lambda}|n|^{2s} \left(\left|\beta^\la_n(0)\right|+\left|\beta^\la_n(0)-\al_n(0)\right|\right)^2.
\end{equation}
Recalling the definition of $\beta^\la$ in \eqref{def:RescaledApproxOrbit} and the results in Theorem \ref{thm:ToyModelOrbit},
\[
\begin{split}
\sum_{n\in\Lambda}|n|^{2s}\left|\beta^\la_n(0)\right|^2&\leq \left(1-\de^\nu\right)S_3+\de^\nu\sum_{j\neq3}S_j\\
&\leq S_3\left(1-\de^\nu+\de^\nu\sum_{j\neq3}\frac{S_j}{S_3}\right).
\end{split}
\]
From Proposition \ref{thm:SetLambda} we know that $j\neq 3$,
\[
\frac{S_j}{S_3}\lesssim e^{sN}
\]
Therefore, to bound these terms we use the definition of  $\de$ from Theorem \ref{thm:ToyModelOrbit} taking $\ga=\wt \ga (s-1)$.
Since $s-1>s_0-1>0$ is fixed, we can choose such $\wt \ga \gg 1$.
Then, we have that
\[
 \sum_{n\in\Lambda}|n|^{2s}\left|\beta^\la_n(0)\right|^2\lesssim \la^{-2}S_3.
\]
From this statement, \eqref{eq:GrowthSobolev:beta} and Theorem \ref{thm:Approximation}, we can conclude that
\[
 \left\|\al(0)\right\|^2_{H^s}\lesssim \la^{-2}S_3.
\]
To complete the proof of statement \eqref{eq:GrowthSobolev:Initial}
recall that the support of $\Gamma(\al)-\al$ is
\[
 \Lambda^3=\left\{n\in\ZZ^2: n=n_1-n_2+n_3, n_1,n_2,n_3\in\Lambda\right\}
\]
and apply Theorem \ref{thm:NormalForm}.

Using inequalities  \eqref{eq:GrowthSobolev:Final} and \eqref{eq:GrowthSobolev:Initial},
we have that
\[
\frac{\left\|a(T)\right\|^2_{H^s}}{\left\|a(0)\right\|^2_{H^s}}\gtrsim \frac{S_{N-1}}{S_3},
\]
and then, applying Proposition \ref{thm:SetLambda}, we obtain
\[
\frac{\|a(T)\|_{H^s}^2}{\|a(0)\|_{H^s}^2}\gtrsim \frac 12 2^{(s-1)(N-4)}
\ge \KK ^2.
\]
%Now, we need to prove that $\|a(0)\|_{H^s}\sim\mu$. To this end, we first prove that $\|\beta^\la(0)\|_{H^s}\sim \mu$. We proceed as in \cite{CollianderKSTT10}. Indeed, note that we have already seen that
%\[
%\|\beta^\la(0)\|_{H^s}\sim \la^{-1}S_{3}.
%\]
%Then, we just  need to increase the parameter $R$ given in Proposition \ref{thm:SetLambda} to ensure that $S_3\sim\la\mu$. Finally, reasoning as in the proof of statement \eqref{eq:GrowthSobolev:Initial}, we can deduce that with this choice of $R$, we have that $\|a(0)\|_{H^s}\sim\mu$.

% To finish the  proof of Theorem \ref{thm:main}, we need the estimate for the time $T$.
It is left to estimate diffusion time $T$.
% let me rewrite old version
% old:
%Recalling its definition in \eqref{def:Time:Rescaled},
%the definition of $\la$ in \eqref{def:LambdaOfN} and
%$\ga=(s-1)\wt \ga$ , we obtain
%\[
% |T|\leq \KKK\ \ga e^{2\kk \ga  N}N^2  \leq \KKK\wt\ga (s-1)^{-1}e^{2\kk\wt\ga %(s-1) N}((s-1)N)^2.
%\]
Use Proposition \ref{thm:SetLambda} to set $\KK\simeq 2^{(s-1)N/2}$ and
$c=4\kk \ga/(s-1)$, definition \eqref{def:LambdaOfN} to set
$\la=e^{\kk \ga N}\simeq \KK^{c/(2\ln2)}$.
%,formula (\ref{eq:GrowthSobolev:Initial}) to set $\mu=\la^{-1}2^N=
%\KK^{-c/(2\ln 2)}2^N\simeq \KK^{-c/(2\ln 2)+2/(s-1)}$.
For time of diffusion we obtain
\[
|T|\leq \KKK \ \ga\ \la^2 N^2  \leq \KKK\ \ga\ \KK^{c/\ln 2}
\dfrac{4\ln^2 \KK}{\ln^2 2\ (s-1)^2}\leq \KK^c
\]
for large $\KK$.
%Then, applying Proposition \ref{thm:SetLambda}, it is clear that there exists
%a constant $c>0$ for which the bound for $T$ in Theorem \ref{thm:main}
%is satisfied.
This completes the proof of Theorem \ref{thm:main}.
\end{proof}

\section{The finite dimensional model: proof of Theorem \ref{thm:ToyModelOrbit}}\label{sec:ProofToyModelThm}
We devote this section to describe the proof of Theorem \ref{thm:ToyModelOrbit}. The proofs of the partial results stated in this section are deferred to Sections \ref{sec:HypToyModel}--\ref{sec:ProofHeteroMap}.

To prove Theorem \ref{thm:ToyModelOrbit} we need to analyze certain orbits of system \eqref{def:model}
given by Hamiltonian $h(b)$ in  \eqref{def:Hamiltonian}.
Moreover,
there is another conserved quantity:
the mass
\begin{equation}\label{def:mass}
\MM(b)=\sum |b_j|^2.
\end{equation}
We obtain orbits given in Theorem \ref{thm:ToyModelOrbit} on the manifold
$\MM(b)=1$.

It can be easily seen that on $\MM(b)=1$ there are periodic orbits $\TT_j$ given
by
\begin{equation}\label{def:periodic-orbits}
 b_j(t)=e^{-it},\,\,b_k(t)=0\,\,\text{for }k\neq j,
\end{equation}
which
in normally directions have mixed type: hyperbolic in some directions and
elliptic others. Moreover, there
exist two families of heteroclinic orbits, which connect consecutive periodic orbits.
Consider the $2$-dimensional complex plane $L_j=\{\forall k\ne j,j+1:\ b_k=0\}$.
In Section \ref{sec:FeaturesModel} we show that they are invariant and dynamics
inside are integrable.
Then, the (two dimensional) unstable manifold of the periodic orbit $(b_j(t),b_{j+1}(t))=(e^{-it},0)$ coincides with the (two dimensional) stable manifold
of $(b_j(t),b_{j+1}(t))=(0,e^{-it})$ and it is foliated by heteroclinic orbits. As usual,
the stable and unstable invariant manifolds have two branches and, therefore, we have two families of heteroclinic connections. It turns out that they can be explicitly computed
\cite{CollianderKSTT10} and are given by
\begin{equation}\label{def:heteroclinic}
\ga_j^\pm(t)=(0,\ldots,0, b_j(t),  b_{j+1}^{\pm}(t),0,\ldots,0)
\end{equation}
with
\[
 b_j(t)=\frac{e^{-i(t+\vartheta)}\omega}{\sqrt{1+e^{2\sqrt{3}t}}},\qquad
b_{j+1}^{\pm}(t)=\pm\frac{e^{-i(t+\vartheta)}\omega^2}{\sqrt{1+e^{-2\sqrt{3}t}}},
\qquad\vartheta\in\TT.
\]
To prove Theorem \ref{thm:ToyModelOrbit} we  look for
an orbit which shadows the sequence of separatrices,
as follows
 \begin{itemize}
  \item it starts close to the periodic orbit $\TT_3$

  \item later it passes close to the periodic orbit $\TT_4$

  \item later it passes close to the periodic orbit $\TT_5$ and so on

  \item finally it arrives to a neighborhood the periodic orbit $\TT_{N-1}$.
  \end{itemize}
  Our main goal is to prove

 \vskip 0.1in

  \qquad \qquad {\it existence of such orbits and estimate
  the transition time in terms of $N$}.

In making these transition we have the freedom of whether to travel close
to $\ga_j^+$ or $\ga_j^-$. We will choose always $\ga_j^+$
The procedure for $\ga_j^-$ is analogous.

We believe it is helpful to the reader to have the following
information about transition of energy. We have a solution
$b(t)=\{b_j(t)\}_{j=0,\dots,N}$ to the system (\ref{def:model}).
We fix $\sigma>0$ small, but independent
of $N$, and  $\de= e^{-\ga N}$. For each $j=2,\dots,N-1$ near
the periodic orbit $\TT_j$ and later near $\TT_{j+1}$ we have
the following table of orders of magnitude of distribution of energy
%\be \beal
\begin{align} \label{energy-distribution}
 \text{near }\TT_j \qquad \qquad \qquad \qquad \qquad
 & \qquad \longrightarrow \qquad  &  \qquad \qquad \text{near }\TT_{j+1}
 \qquad \qquad \qquad \notag \\
 |b_{<j-2}| \qquad \qquad \qquad \qquad \qquad
 & \qquad \longrightarrow \qquad  & |b_{<j-2}|\ (1+O(\de^{r'})) \qquad \qquad \notag \\
 |b_{j-2}| \quad\qquad\quad \qquad \qquad \qquad
 & \qquad \longrightarrow \qquad  & K|b_{j-2}| \quad \qquad \qquad \qquad \notag \\
 |b_{j-1}|=O(\sigma) \qquad \qquad\qquad \qquad
 & \qquad \longrightarrow \qquad  & (C^{(j)}\de)^{1/2}\qquad \qquad \qquad  \\
 |b_j|=1-O(\sigma^2)\ (\textup{mass conservation}) &
 \qquad \longrightarrow \qquad & O(\sigma) \qquad \qquad \qquad\qquad \notag \\
 |b_{j+1}|=(C^{(j)}\de)^{1/2}\qquad \qquad\qquad \qquad &
 \qquad \longrightarrow \qquad & 1-O(\sigma^2)\ (\textup{mass conservation}) \notag \\
 |b_{j+2}|\qquad\,\quad\quad \qquad \qquad \qquad
 & \qquad \longrightarrow \qquad  & K|b_{j+2}| \quad \qquad \qquad \qquad \notag \\
 |b_{>j+2}| \qquad \qquad\qquad \qquad \qquad& \qquad \longrightarrow \qquad
 & |b_{>j+2}|\ (1+O(\de^{r'})).\qquad \qquad  \notag
\end{align}
%\enal \ee

 \vskip 0.1in

  We decompose
  a diffusing orbit into $N-5$ parts: near each
  periodic orbit $\TT_j,\ j=3,\dots,N-1$ we construct sections transversal to
  the flow so that they divide the orbit appropriately.
  For each transition from one section to the next one we associate
  a map $\BB^j$ which sends points close to $\TT_j$ to points close to
  $\TT_{j+1}$.
  This leads to analysis of the composition of all these maps
\[
 \BB^\ast=\BB^{N-1}\circ\ldots\circ\BB^3.
\]
To study these maps we will consider different systems of coordinates which, on
one hand, will take advantage of the fact that mass \eqref{def:mass} is a conserved
quantity, and on the other hand, will be adapted to the linear normal
behavior of the periodic orbits. These systems of coordinates are specified in Section \ref{sec:SaddleMapFirstChanges}.

\subsection{Symplectic reduction and
diagonalization}\label{sec:SaddleMapFirstChanges}
To study the different transition maps we use
a system of coordinates defined
in \cite{CollianderKSTT10}.
It consists of two steps:
\begin{itemize}
 \item A symplectic reduction
 uses that mass \eqref{def:mass} is conserved and sends
 the periodic orbit $\TT_j$ into a critical point.
\item
A linear transformation diagonalizes the linearization of dynamics near
this critical point.
\end{itemize}
We perform the change corresponding to the traveling close to the $j$ periodic
orbit $\TT_j$. We restrict ourselves to $\MM(b)=1$ and we take
\begin{equation}\label{def:SaddleAdaptedCoordinates}
 b_j=r^{(j)}e^{i\theta^{(j)}},\,\, b_k=c_k^{(j)}e^{i\theta^{(j)}} \ \text{ for all }\ k\ne j,
\end{equation}
where $\tet^{(j)}$ is a variable on $\TT_j$. From now on in this section
we omit the superscripts ${(j)}$. It can be seen that after eliminating $r$
using that $\MM(b)=1$ and omitting the equation for the variable $\theta$,
one obtain a new set of equations whose $c_k$ components form a Hamiltonian system
with the Hamiltonian
\[
\begin{split}
 H^{(j)}(c)=&\frac{1}{4}\sum_{k\neq j}|c_k|^4+\frac{1}{4}\left(1-\sum_{k\neq
j}|c_k|^2\right)^2-\frac{1}{2}\sum_{k\neq j,j+1}\ol c_k^2c_{k-1}^2+c_k^2\ol
c_{k-1}^2\\
&-\frac{1}{2}\left(1-\sum_{k\neq j}|c_k|^2\right)\left(c_{j-1}^2+\ol
c_{j-1}^2+c_{j+1}^2+\ol c_{j+1}^2\right)
\end{split}
\]
and the symplectic form $\Omega=\frac{i}{2}dc_k\wedge d\ol{ c_k}$.
The Hamiltonian $H^{(j)}(c)$ can be written as
\begin{equation}\label{def:Hamiltonian:Reduced}
 H^{(j)}(c)= H^{(j)}_2(c)+ H^{(j)}_4(c)
\end{equation}
with
 \begin{align}
H^{(j)}_2(c)=&-\frac{1}{2}\sum_{k\neq
j}|c_k|^2-\frac{1}{2}\left(c_{j-1}^2+\ol c_{j-1}^2+c_{j+1}^2+\ol
c_{j+1}^2\right)\notag\\
H^{(j)}_4(c)=&\frac{1}{4}\sum_{k\neq j}|c_k|^4+\frac{1}{4}\left(\sum_{k\neq
j}|c_k|^2\right)^2-\frac{1}{2}\sum_{k\neq j,j+1}\ol c_k^2c_{k-1}^2+c_k^2\ol
c_{k-1}^2\notag\\
&+\frac{1}{2}\sum_{k\neq j}|c_k|^2\left(c_{j-1}^2+\ol c_{j-1}^2+c_{j+1}^2+\ol
c_{j+1}^2\right).\label{def:HamC:H4}
 \end{align}
Since we are omitting the evolution of the variable $\theta$, the periodic orbit $\TT_j$ has become now a critical point for the equation associated to this Hamiltonian, which is defined as  $c=0$. For the same reason, the two families of heteroclinic connections defined in \eqref{def:heteroclinic}, now have become just two one dimensional heteroclinic connections.

The second step is to look for a change of
variables which diagonalizes the vector field around this critical point. This
change only modifies the coordinates $(c_{j-1},c_{j+1})$ and is given by
\begin{equation}\label{def:ChangeToDiagonal}
 \left(\begin{array}{c} c_{j-1}\\ c_{j+ 1}\end{array}\right)=
\left(\begin{array}{c} \omega^2 p_1+\omega q_1\\ \omega^2
p_2+\omega q_2\end{array}\right)
\end{equation}
where $\omega=e^{2\pi i/3}$ (see \cite{CollianderKSTT10}). Note that this change
is conformal and leads to the
symplectic form
\begin{equation}\label{def:SymplecticAfterDiagonal}
 \wt \Omega=\frac{i}{2}dc_k\wedge d\ol c_k+dp_1\wedge dq_1+dp_2\wedge dq_2.
\end{equation}
To study the Hamiltonian expressed in the new variables let us introduce some
notation. We define
\begin{equation}\label{def:Set:EllipticModes}
\PP_j=\{1\leq k\leq N; k \neq j-1,j,j+1\},
\end{equation}
which is the set of subindexes of the elliptic modes. From now on we will denote
by $q$ and $p$ all the stable and unstable coordinates $q=(q_1,q_2)$ and
$p=(p_1,p_2)$ respectively and by  $c$ all the elliptic modes, namely $c_k$ with
$k\in \PP_j$.

\begin{lemma}\label{lemma:Diagonalization}
 The change \eqref{def:ChangeToDiagonal} transforms the Hamiltonian
\eqref{def:Hamiltonian:Reduced} into the Hamiltonian
\begin{equation}\label{def:Ham:Diagonal}
 \wt H^{(j)}(p,q,c)=\wt H^{(j)}_2(p,q,c)+\wt H^j_4(p,q,c)
\end{equation}
with homogeneous polynomials
\[
 \wt H^{(j)}_2(p,q,c)=-\frac{1}{2}\sum_{k\in
\PP_j}|c_k|^2+\sqrt{3}\left(p_1q_1+p_2q_2\right).
\]
and
\[
 \wt H^{(j)}_4(p,q,c)= \wt H^{(j)}_{\hyp}\left(p,q\right)+\wt H^{(j)}_{\el}(c)+\wt
H^{(j)}_{\mix}\left(p,q,c\right)
\]
where
\begin{align}
\wt
H^{(j)}_\hyp(p,q)=&\sum_{k=1}^3\nu_{k}
p_1^kq_1^{4-k}+\sum_{k=1}^3\nu_{k}
p_2^kq_2^{4-k}+\sum_{k,\ell=0}^2\nu_{k\ell}
p_1^kq_1^{2-k}p_2^\ell q_2^{2-\ell}\notag\\
%&+\sum_{k,\ell=0}^2\nu_{k\ell}
%p_1^kq_1^{2-k}p_2^\ell q_2^{2-\ell}\notag\\
\wt H^{(j)}_{\el}\left(c\right)=&\frac{1}{4}\sum_{k\in\PP_j}|c_k|^4+\frac{1}{4}
\left(\sum_{k\in\PP_j}|c_k|^2\right)^2 \label{def:Ham4:Ell:Original}\\
&-\frac{1}{2}\sum_{k\neq j-1,j,j+1,j+2}c_k^2\ol{c_{k-1}}^2+\ol {c_k}^2
c_{k-1}^2\notag\\
\wt H^{(j)}_{\mix}(p,q,c)=&-\frac{\sqrt{3}}{2}\sum_{k\in
\PP_j}|c_k|^2\left(q_1p_1+q_2p_2\right)\label{def:Ham4:Mix:Original}\\
&-\frac{1}{2}\left(\omega^2 p_1+\omega q_1\right)^2\ol{ c_{j-
2}}^2-\frac{1}{2}\left(\omega^2 q_1+\omega p_1\right)^2 c_{j- 2}^2\notag\\
&-\frac{1}{2}\left(\omega^2 p_2+\omega q_2\right)^2\ol{ c_{j+
2}}^2-\frac{1}{2}\left(\omega^2 q_2+\omega p_2\right)^2 c_{j+2}^2\notag
\end{align}
for certain constants and $\nu_k,\nu_{k\ell}\in\RR$.
\end{lemma}

\begin{remark}\label{remark:FormOfVF}
Even though the proof of this lemma is a simple substitution of $(p,q)$ we do need
specifics of the form of the decomposition into Hamiltonians:
\begin{itemize}

\item $\wt H^{(j)}_2$ is the direct product of two linear saddles $(p_i,q_i),\ i=1,2$
 and $N-2$ linear elliptic points $\{c_k\}_k, \ k\in \PP_j$.

\item $\wt H^{(j)}_{\hyp}$ consists of
some only saddle terms. In particular, it does not contain terms
$p_i^4,q_i^4, \ i=1,2$ so $\{q=0\}$ and $\{p=0\}$ are invariant
manifolds of $\wt H$ if we set $c=0$. This implies that the two heteroclinic orbits which connect the critical point $(p,q,c)=(0,0,0)$ to the next periodic orbit $\TT_{j+1}$ are just defined as
\[
 \left(p_1^\pm(t), q_1^\pm(t),p_2^\pm(t), q_2^\pm(t),c^\pm (t) \right)=\left(0,0,\frac{\pm 1}{1+e^{-2\sqrt{3}t}},0,0\right).
\]
Moreover,
$\TT_{j+1}$ is now defined as $|c_{j+1}|=1$.
 Due to \eqref{def:ChangeToDiagonal} it is equivalent to $p_2^2+q_2^2-p_2q_2=1$.

\item Near $p=q=0$, which corresponds to the periodic orbit $\TT_j$
Hamiltonians $\wt H^{(j)}_{\el}$ and $\wt H^{(j)}_{\mix}$ are almost integrable.
The only source of non-integrability comes from the second line
of (\ref{def:Ham4:Ell:Original}) for $\wt H^{(j)}_{\el}$ and from
the second and third line of (\ref{def:Ham4:Mix:Original})
for $\wt H^{(j)}_{\mix}$.

\item Later we select regions with $c$'s being exponentially small in $N$.
As the result coupling between hyperbolic variables -- $(p,q)$ and
elliptic ones $c$'s is exponentially small in $N$. This decoupling
at the leading order is crucial for our analysis.
\item Among all the constants $\nu_{k}$ which appear in the definition
of Hamiltonian \eqref{def:Ham:Diagonal},  $\nu_{02}\neq 0$ is
the only one which plays a significant role in the proof of
Theorem \ref{thm:ToyModelOrbit}. Indeed, the corresponding term is
resonant and will be the leading term in studying the transition close
to the saddle. We assume, without loss of generality that $\nu_{02}>0$
since the case $\nu_{02}<0$ can be done analogously.
\end{itemize}
\end{remark}

\begin{proof}
To obtain the explicit form of $\wt H^{(j)}_4$, note that
$H^{(j)}_4(c)$ in \eqref{def:HamC:H4} can be rewritten as
\[
\begin{split}
 H^{(j)}_4(c)=&\frac{1}{4}\sum_{k\neq j}|c_k|^4+\frac{1}{4}\left(\sum_{k\neq
j}|c_k|^2+c_{j-1}^2+\ol c_{j-1}^2+c_{j+1}^2+\ol
c_{j+1}^2\right)^2\\
&-\frac{1}{2}\sum_{k\neq j,j+1}\ol c_k^2c_{k-1}^2+c_k^2\ol
c_{k-1}^2-\frac{1}{4}\left(c_{j-1}^2+\ol c_{j-1}^2+c_{j+1}^2+\ol
c_{j+1}^2\right)^2.
%\\
%&-\frac{1}{4}\left(c_{j-1}^2+\ol c_{j-1}^2+c_{j+1}^2+\ol
%c_{j+1}^2\right)^2.
\end{split}
\]
Written in this way, the second term in the first row is just a constant times
$\wt H^{(j)}_2$ squared. Then, the particular form of $\wt H^{(j)}_\hyp$, $\wt H^{(j)}_\el$,
and $\wt H^{(j)}_\mix$ can be obtained just performing the change of coordinates.
\end{proof}

Since the symplectic form is given by \eqref{def:SymplecticAfterDiagonal}, equations associated to the Hamiltonian \eqref{def:Ham:Diagonal} are
\begin{eqnarray}\label{def:VF:Full:AfterDiagonal}
\begin{aligned}
\dot p_1
& = & \sqrt{3}p_1+\ZZZ_{\hyp, p_1}+\ZZZ_{\mix, p_1}
& = & \sqrt{3}p_1+\partial_{q_1}\wt H^{(j)}_{\hyp}+\partial_{q_1}\wt H^{(j)}_{\mix}\\
\dot q_1
& = & -\sqrt{3}q_1+\ZZZ_{\hyp, q_1}+\ZZZ_{\mix, q_1}
& = & -\sqrt{3}q_1-\partial_{p_1}\wt H^{(j)}_{\hyp}-\partial_{p_1}\wt H^{(j)}_{\mix}\\
\dot p_2
& = & \sqrt{3}p_2+\ZZZ_{\hyp, p_2}+\ZZZ_{\mix, p_2}
& = & \sqrt{3}p_2+\partial_{q_2}\wt H^{(j)}_{\hyp}+\partial_{q_2}\wt H^{(j)}_{\mix}\\
\dot q_2
& = & -\sqrt{3}q_2+\ZZZ_{\hyp, q_2}+\ZZZ_{\mix, q_2}
& = & -\sqrt{3}q_2-\partial_{p_2}\wt H^{(j)}_{\hyp}-\partial_{p_2}\wt H^{(j)}_{\mix}\\
\dot c_k
& = & ic_k  + \ZZZ_{\el, c_k} + \ZZZ_{\mix, c_k}
& = & ic_k-2i\partial_{c_k}\wt H^{(j)}_{\el}-2i\partial_{c_k}\wt
H^{(j)}_{\mix}.
\end{aligned}
\end{eqnarray}
where
\begin{align}
\ZZZ_{\hyp, p_1}=
&\sum_{k=1}^3(4-k)\nu_{k}p_1^kq_1^{3-k}
+\nu_{12}p_1p_2^2+\nu_{11}p_1p_2q_2+\nu_{10}p_1q_2^2 \label{def:VF:Diag:p1:Hyp}\\
&+2\nu_{02}q_1p_2^2+2\nu_{01}q_1p_2q_2+2\nu_{00}q_1q_2^2\notag\\
\ZZZ_{\hyp,
q_1}=&-\sum_{k=1}^3 k\nu_{k}p_1^{k-1} q_1^{4-k}-2\nu_{22}p_1p_2^2-2\nu_{21}p_1p_2q_2-2\nu_{20}p_1q_2^2
\label{def:VF:Diag:q1:Hyp}\\
&-\nu_{12}q_1p_2^2-\nu_{11}q_1p_2q_2-\nu_{10}q_1q_2^2\notag\\
\ZZZ_{\hyp,
p_2}=&\sum_{k=1}^4 (4-k)\nu_{k}p_2^kq_2^{3-k}+\nu_{21}p_1^2p_2+\nu_{11}p_1q_1p_2+\nu_{01}q_1^2p_2
\label{def:VF:Diag:p2:Hyp}\\
&+2\nu_{20}p_1^2q_2+2\nu_{10}p_1q_1q_2+2\nu_{00}q_1^2q_2\notag\\
\ZZZ_{\hyp,
q_2}=&-\sum_{k=1}^4 k\nu_{k\ell}p_2^{k-1}q_2^{4-k}-2\nu_{22}p_1^2p_2-2\nu_{12}p_1q_1p_2-2\nu_{02}q_1^2p_2
\label{def:VF:Diag:q2:Hyp}\\
&-\nu_{21}p_1^2q_2-\nu_{11}p_1q_1q_2-\nu_{01}q_1^2q_2 \notag \\
 \ZZZ_{\el, c_k}=&-i|c_k|^2
c_k-i\left(\sum_{\ell\in\PP_j}|c_\ell|^2\right)c_k
+2i\ol{c_k}\left(c_{k-1}^2+c_{k+1}^2\right)\label{def:VF:Diag:cm:El}
\end{align}
%where
\begin{align}
\ZZZ_{\mix, q_1}=&\omega^2 (\omega^2 p_1+\omega q_1)\ol{c_{j-2}}^2+\omega
(\omega p_1+\omega^2 q_1){c_{j-2}}^2+\frac{\sqrt{3}}{2}\sum_{\ell\in\PP_j}|c_\ell|^2q_1
\label{def:VF:Diag:q1:Mix}\\
%&+\frac{\sqrt{3}}{2}\sum_{\ell\in\PP_j}|c_\ell|^2q_1\notag\\
\ZZZ_{\mix, p_1}=&-\omega (\omega^2 p_1+\omega q_1)\ol{c_{j-2}}^2-\omega^2
(\omega p_1+\omega^2 q_1){c_{j-2}}^2-\frac{\sqrt{3}}{2}\sum_{\ell\in\PP_j}|c_\ell|^2p_1
\label{def:VF:Diag:p1:Mix}\\
%&-\frac{\sqrt{3}}{2}\sum_{\ell\in\PP_j}|c_\ell|^2p_1\notag\\
\ZZZ_{\mix, q_2}=&\omega^2 (\omega^2 p_2+\omega q_2)\ol{c_{j+2}}^2+\omega
(\omega p_2+\omega^2 q_2){c_{j+2}}^2+\frac{\sqrt{3}}{2}\sum_{\ell\in\PP_j}|c_\ell|^2q_2
\label{def:VF:Diag:q2:Mix}\\
%&+\frac{\sqrt{3}}{2}\sum_{\ell\in\PP_j}|c_\ell|^2q_2\notag\\
\ZZZ_{\mix, p_2}=&-\omega (\omega^2 p_2+\omega q_2)\ol{c_{j+2}}^2-\omega^2
(\omega p_2+\omega^2 q_2){c_{j+2}}^2-\frac{\sqrt{3}}{2}\sum_{\ell\in\PP_j}|c_\ell|^2p_2
\label{def:VF:Diag:p2:Mix}\\
%&-\frac{\sqrt{3}}{2}\sum_{\ell\in\PP_j}|c_\ell|^2p_2\notag\\
\ZZZ_{\mix,c_k}=&i\sqrt{3}c_k (q_1p_1+q_2p_2)\,\,\,\text{ for }k\in
\PP_j\setminus\{ j\pm 2\}\label{def:VF:Diag:cm:Mix}\\
\ZZZ_{\mix,c_{j- 2}}=&i\sqrt{3}c_{j- 2} (q_1p_1+q_2p_2)
-2i (\omega^2 p_1+\omega q_1)^2\ol{c_{j- 2}}\label{def:VF:Diag:cm:Mix.Adjacent}\\
\ZZZ_{\mix,c_{j+ 2}}=&i\sqrt{3}c_{j+ 2} (q_1p_1+q_2p_2)
-2i (\omega^2 p_2+\omega q_2)^2\ol{c_{j+ 2}}\notag.
\end{align}

\subsection{The iterative Theorem}

Now that we have obtained the adapted coordinates for each saddle we are ready
to explain the strategy to prove Theorem \ref{thm:ToyModelOrbit}. To obtain the
orbit given in Theorem \ref{thm:ToyModelOrbit}, we will consider several
co-dimension one sections $\{\Sigma^{\inn}_j\}_{j=1}^N$ and transition maps
$\BB^j$ from one section $\Sigma^{\inn}_j$ to the next one $\Sigma^{\inn}_{j+1}$.
Then, we will
detect a class of open sets $\{\VV_j\}_j,\ \VV_j \subset \Sigma_j^{\inn},\ j=1,\dots,N-1$,
which have a certain {\it almost product structure} (see Definition \ref{def:ProductLikeSet})
such that $\VV_{j+1}\subset\BB^{j}\left(\VV_j\right)$ and none of them is empty.
Each set $\VV_j$ is located close to the stable manifold of the periodic orbit
$\TT_j$. Composing all these maps we will be able to find orbits claimed to exist
in Theorem \ref{thm:ToyModelOrbit}.

We start by defining these maps. The first step is to define certain transversal
sections to the flow. We use  the coordinates adapted to the saddle $j$,
$(p^{(j)},q^{(j)},c^{(j)})$, which have been introduced in Section
\ref{sec:SaddleMapFirstChanges}, to define these
sections. Indeed, in these coordinates, it can be easily seen that the
heteroclinic connections \eqref{def:heteroclinic}, which connect
$(p^{(j)},q^{(j)},c^{(j)})=(0,0,0)$ with the previous and next saddles are
defined by $(q_1^{(j)},p_2^{(j)},q_2^{(j)},c^{(j)})=(0,0,0,0)$ and
$(p_1^{(j)},q_1^{(j)},p_2^{(j)},c^{(j)})=(0,0,0,0)$ respectively. Thus,
we define the map $\BB^j$ from the section
\begin{equation}\label{def:Section1Saddle}
\Sigma_j^{\inn}=\left\{q^{(j)}_1=\sigma\right\}
\end{equation}
to the  section
\[
\Sigma_{j+1}^{\inn}=\left\{q^{(j+1)}_1=\sigma\right\}.
\]
Here $\sigma>0$ is a small parameter that will be determined later on. In fact,
we do not define the map $\BB^j$ in the whole section but in an open set
$\VV^j\subset\Sigma_j^{\inn}$, which lies close to the heteroclinic that connects the
saddle $j-1$ to the saddle $j$.  Then, we will consider maps
\[
 \BB^{j}:\VV_j\subset\Sigma_j^{\inn}\rightarrow\Sigma_{j+1}^{\inn}
\]
and we will choose the sets $\VV^j$ recursively in such a way that
\begin{equation}\label{cond:ComposeMaps}
 \VV_{j+1}\subset\BB^{j}\left(\VV_j\right).
\end{equation}
This condition will allow us to compose  all the maps $\BB^j$.
%M
Indeed, the domain of definition of the map $\BB^{j+1}$ will intersect the image
of the map $\BB^j$ in an open set.

The sets $\VV_j$ will have a product-like structure as is stated in the next definition.
Before stating it, we introduce some notation. We define the subsets of
indices $\PP_j$ in \eqref{def:Set:EllipticModes},
\begin{equation}\label{def:Elliptic:TwoSets}
\begin{split}
\PP^-_j&=\{k=1,\ldots, j-3\}\\
\PP^+_j&=\{k=j+3,\ldots, N\}.
\end{split}
\end{equation}
% I will rephrase, make your choice
%Old
%which correspond to the \emph{far} elliptic modes, in contraposition to the
%modes $k=j\pm 2$, which will be called \emph{adjacent} and which have a stronger
%interaction with the hyperbolic modes.
%New
The first set consists of preceding non-neighbor modes to $j-1$, the second ---
of foreseeing non-neighbor modes to $j+1$. The modes $k=j\pm 2$ are called
\emph{adjacent}. These modes have a stronger
interaction with the hyperbolic modes.

Note that we split the non-neighbor elliptic modes in two sets:
% I will rephrase, make your choice
%Old:
% the $+$ stands for \emph{future} since are the ones that will
% eventually become hyperbolic in the future, and $-$ stands for \emph{past},
% since they have  already been hyperbolic. Analogously, we will call future
% adjacent to the mode $c_{j+2}^{(j)}$ and past adjacent to the mode $c_{j-2}^{(j)}$.
%New
the $+$ stands for \emph{future} $-$ stands for \emph{past}.
Indeed, along orbits we study future modes will eventually become hyperbolic
in the future, past have already been hyperbolic. Analogously, we call future
adjacent --- the mode  $c_{j+2}^{(j)}$ and past adjacent --- $c_{j-2}^{(j)}$.

For a point $(p^{(j)},q^{(j)},c^{(j)})\in\Sigma_j^\inn$, we define
$c_-^{(j)}=(c_1^{(j)},\ldots,c_{j-2}^{(j)})$ and
$c_+^{(j)}=(c_{j+2}^{(j)},\ldots,c_N^{(j)})$. We define also the projections
$\pi_\pm(p^{(j)},q^{(j)},c^{(j)})=c_\pm^{(j)}$ and
$\pi_{\hyp,+}=(p^{(j)},q^{(j)},c_+^{(j)})$.

\begin{definition}\label{def:ProductLikeSet}
Fix
positive constants $r\in (0,1)$, $\de$ and $\sigma$ and
define a multi-parameter set of positive constants
\begin{equation}\label{def:ProductLikeIndex}
 \II_j=\left\{C^{(j)},m_{\el}^{(j)},  M_{\el,\pm}^{(j)}, m_{\adj}^{(j)},
M_{\adj,\pm}^{(j)}, m_{\hyp}^{(j)}, M_{\hyp}^{(j)}\right\}.
\end{equation}
Then, we say that a (non-empty) set $\UU\subset\Sigma_j^\inn$ has an
$\II_j$-product-like structure
if it satisfies the following two conditions:
\begin{description}
 \item[\textbf{C1}]
\[
 \UU\subset \DD_j^1\times\ldots\times\DD_j^{j-2}\times\NNN_{j}^+\times
\DD_j^{j+2}\times\ldots\times\DD_j^{N},
\]
where
\begin{align*}
\DD_j^k&=\left\{\left|c_k^{(j)}\right|\leq  M^{(j)}_{\el,\pm}
\de^{(1-r)/2}\right\}
\,\,\text{ for }k\in \PP^\pm_j\\
\DD_j^{j\pm 2}&\subset\left\{\left|c_{j\pm2}^{(j)}\right|\leq
M^{(j)}_{\adj,\pm} \left(C^{(j)}\de\right)^{1/2}\right\}
\end{align*}
and
\begin{equation}\label{def:Domain:Hyperbolic:sup}
\begin{split}
\NNN_{j}^+= \Big\{& \left(p_1^{(j)},q_1^{(j)},p_2^{(j)},q_2^{(j)}\right)\in
\RR^4: \\ -C^{(j)}&\de\left(\ln(1/\de)+M^{(j)}_{\hyp}\right)\leq p_1^{(j)}\leq
-C^{(j)}\de\left(\ln(1/\de)-M^{(j)}_{\hyp}\right),\\ &q_1^{(j)}=\sigma,\
%M
g_{\II_j}(p_2,q_2,\sigma,\de)=0, \ |p_2^{(j)}|,|q_2^{(j)}|\leq
M^{(j)}_{\hyp}\left(C^{(j)}\de\right)^{1/2} \Big\}.
\end{split}
\end{equation}
\item[\textbf{C2}]
\[
\NNN^-_{j}\times \DD_{j,-}^{j+2}\times\ldots\times\DD_{j,-}^{N}
\subset\pi_{\hyp,+}\UU,
\]
where
\begin{align*}
\DD_{j,-}^k&=\left\{\left|c_k^{(j)}\right|\leq  m^{(j)}_{\el}
\de^{(1-r)/2}\right\} \,\,\text{ for }k\in \PP_j^+\\
\DD_{j,-}^{j+ 2}&=\left\{\left|c_{j+2}^{(j)}\right|\leq  m^{(j)}_{\adj}
\left(C^{(j)}\de\right)^{1/2}\right\}
\end{align*}
and
\begin{equation}\label{def:Domain:Hyperbolic:inf}
\begin{split}
\NNN_j^-= \Big\{& \left(p_1^{(j)},q_1^{(j)},p_2^{(j)},q_2^{(j)}\right)\in \RR^4:
\\ -C^{(j)}&\de\left(\ln(1/\de)+m^{(j)}_{\hyp}\right)\leq p_1^{(j)}\leq
-C^{(j)}\de\left(\ln(1/\de)-m^{(j)}_{\hyp}\right),\\ &q_1^{(j)}=\sigma,\
g_{\II_j}(p_2,q_2,\sigma,\de)=0, \ |p_2^{(j)}|,|q_2^{(j)}|\leq
m^{(j)}_{\hyp}\left(C^{(j)}\de\right)^{1/2}\Big\}.
\end{split}
\end{equation}
\end{description}
The function $g_{\II_j}(p_2,q_2,\sigma,\de)$ is a smooth function
%M
defined in \eqref{def:Function_g}.
\end{definition}
\begin{remark}\label{remark:SignP1}
Note that for this product-like sets the variable $p_1^{(j)}$
is selected negative. This is related to the fact that
$\nu_{02}>0$
(see Remark \ref{remark:FormOfVF}). The reason of the choice of the sign of $p_1^{(j)}$
will be clear  in Section \ref{sec:HypToyModel}. In particular, see
Remark \ref{remark:cancellation}.
\end{remark}

The domains $\VV_j$ of the maps $\BB^j$ will have $\II_j$-product-like structure
as defined in Definition \ref{def:ProductLikeSet}. Thus, we need to obtain
the multi-parameter sets $\II_j$. They will be defined recursively.
Recall that, to prove
Theorem \ref{thm:ToyModelOrbit}, we want to obtain an orbit which starts close to
the periodic orbit $\TT_3$. Thus, the recursively defined multi-parameter sets $\II_j$
will start with a set $\II_3$.

\begin{definition}\label{def:RecursivelyDefined}
Fix any constants $r,r'\in (0,1)$ satisfying  $0<r'<1/2-2r$, $K>0$ and small $\de,
\sigma>0$.
We say that  a collection of multi-parameter sets $\{\II_j\}_{j=3,\ldots,N-1}$ defined
in \eqref{def:ProductLikeIndex} is $(\sigma,\de, K)$-recursive if for $j=3,\ldots, N-1$ the
constants
$C^{(j)}$ satisfy
\[
\begin{split}
C^{(j)}/K\leq C^{(j+1)}\leq K C^{(j)}\\
0<m_{\hyp}^{(j+1)}\leq m_{\hyp}^{(j)}
\end{split}
\]
and all the other parameters should be strictly positive and are defined
recursively as
\[
 \begin{split}
  M_{\el,\pm}^{(j+1)}&=M_{\el,\pm}^{(j)}+ K\de^{r'}\\
  m_{\el}^{(j+1)}&=m_{\el}^{(j)}- K\de^{r'}\\
  M_{\adj,+}^{(j+1)}&= 2M_{\el,+}^{(j)}+ K\de^{r'}\\
  M_{\adj,-}^{(j+1)}&= KM_{\hyp}^{(j)}\\
  m_{\adj}^{(j+1)}&=\frac{1}{2}m_{\el}^{(j)}- K\de^{r'}\\
  M_{\hyp}^{(j+1)}&=K M_{\adj,+}^{(j)} \\
 \end{split}
\]
\end{definition}

The next Theorem defines recursively the product-like sets  $\VV_j$, so that
condition \eqref{cond:ComposeMaps} is satisfied.

\begin{theorem}[Iterative Theorem]\label{theorem:iterative}
Fix large $\gamma>0$, small $\sigma>0$, and any constants $r,r'\in (0,1)$
satisfying
$0<r'<1/2-2r$. Then, if we
set $\de= e^{-\gamma N}$, there exist strictly positive constants
$K$ and $C^{(3)}$ independent of $N$ satisfying
\begin{equation}\label{cond:GrowthOfCs}
 C^{(3)}\leq \de^{-r}\,K^{-(N-2)},
\end{equation}
and a multi-parameter set $\II_3$ (as defined in \eqref{def:ProductLikeIndex})
with the following property:
there exists a
%$(\de,\sigma, K)$-recursive
% to fit previous notations
$(\sigma,\de, K)$-recursive collection of multi-parameter sets
collection of multi-parameter sets
$\{\II_j\}_{j=3,\ldots,N-1}$ and $\II_j$-product-like sets
$\VV_j\subset\Sigma_{j}^\inn$ such that for each $j=3,\ldots,N-1$ we have
\[
\VV_{j+1}\subset \BB^j(\VV_j).
\]
Moreover, the time spent to reach the section $\Sigma^{\inn}_{j+1}$ can be bounded
by
\[
 \left|T_{\BB^j}\right|\leq K\ln(1/\de)
\]
for any $(p,q,c)\in \VV_j$ and any $j=3,\ldots,N-2$.
\end{theorem}

Note that the condition
\[
 C^{(j)}/K<C^{(j+1)}<K C^{(j)}
\]
implies
\[
K^{-(j-2)}C^{(3)}\leq  C^{(j+1)}\leq K^{j+2}C^{(3)}
\]
Namely, at each saddle, the orbits we are studying may lie  further
from the heteroclinic orbit. Nevertheless, by the condition on $\de$
from Theorem \ref{thm:ToyModelOrbit} and \eqref{cond:GrowthOfCs},
these constant does not grow too much. Indeed,
\begin{equation}\label{def:UpperBoundCs}
\de^r\leq C^{(j)} \leq \de^{-r},
\end{equation}
where $r>0$  can be taken as small as desired.  We will use the bound
\eqref{def:UpperBoundCs} throughout the proof of Theorem \ref{theorem:iterative}.

Theorem \ref{thm:ToyModelOrbit} is a straightforward consequence of Theorem \ref{theorem:iterative}. In fact,  we need more precise information than the one stated in Theorem \ref{thm:ToyModelOrbit}. This more precise information will be used in the proof of Theorem \ref{thm:Approximation}. We state it in the following theorem. Theorem \ref{thm:ToyModelOrbit} is a straightforward
consequence of it.

{\bf Theorem 3--bis}\ {\it \
%\begin{lemma}\label{lemma:ToyModelBehavior}
Assume that the conditions of Theorem \ref{thm:ToyModelOrbit} hold.
Then, there exists an orbit $b(t)$
of equations \eqref{def:model}, constants $\KKK>0$ and $\nu>0$,
independent of $N$ and $\de$, and $T_0>0$ satisfying
\[
 T_0\leq \KKK N\ln(1/\de),
\]
such that
\[
\begin{split}
 |b_3(0)|&>1-\de^\nu\\
|b_j(0)|&< \de^\nu\qquad\text{ for }j\neq 3
\end{split}
\qquad \text{ and }\qquad
\begin{split}
 |b_{N-2}(T_0)|&>1-\de^\nu\\
|b_j(T_0)|&<\de^\nu \qquad\text{ for }j\neq N-2
\end{split}
\]
Moreover, call $t_j\in [0,T_0]$ the times for which
$b(t_j)\in\Sigma^\inn_j$, Then,
\[
 t_{j+1}-t_j\leq \KKK\ln (1/\de)
\]
and for any $t\in[t_j, t_{j+1}]$ and  $k\neq j-1,j,j+1$,
\[
 |b_k(t)|\leq \de^{\nu}.
\]
}

\begin{proof}[Proof of Theorem 3--bis]
It is enough to take as a initial condition $b^0$ a point in the set
$\VV_3\subset\Sigma_3^\inn$ obtained in Theorem \ref{theorem:iterative}.
Then, thanks to this theorem we know that there exists a time $T_0$
satisfying
\[
 T_0\sim N\ln(1/\de),
\]
such that the corresponding orbit satisfies that
$b(T_0)\in\VV_{N-1}\subset\Sigma_{N-1}^\inn$. Note that in this section
there are two components of $b$ with size independent of $\de$.
Nevertheless, from the proof of Theorem \ref{theorem:iterative} in
Section \ref{sec:FullSystem} it can be easily seen that if we  shift
the time interval $[0,T_0]$ to
$[\rr\ln(1/\de),\rr\ln(1/\de)+T_0]$, for any $\rr<\sqrt{3}$, there
exists $\nu>0$ such that the orbit $b(t)$ satisfies the statements given in Theorem 3--bis.
\end{proof}

\subsection{Structure of the proof of the Iterative Theorem
\ref{theorem:iterative}}
To prove Theorem \ref{theorem:iterative} we split it into two inductive lemmas. The
first part analyzes the evolution of the trajectories close to the saddle $j$
and the second one the travel along the heteroclinic orbit. Thus, we study
$\BB^j$ as a composition of two maps.

%In this section we omit the superscripts
%and we denote the variables adapted to the $j$
%and $j+1$ saddles by
%$(q,p,c)=(q^{(j)},p^{(j)},c^{(j)})$ and
% $(\wh q,\wh p,\wh
%c)=(q^{(j+1)},p^{(j+1)},c^{(j+1)})$ respectively.

We consider an intermediate section transversal to the flow
\begin{equation}\label{def:Section2Saddle}
\Sigma_j^{\out}=\left\{p_2^{(j)}=\sigma\right\},
\end{equation}
and then we consider two maps. First the local map
\begin{equation}\label{def:SaddleMap}
\BB_\loc^j:\VV_j\subset\Sigma_j^{\inn}\longrightarrow \Sigma_j^{\out},
\end{equation}
which studies the trajectories locally close to the saddle.
Then, we consider a second map,
\begin{equation}\label{def:HeteroMap}
\BB_\glob^j: \UU^j\subset\Sigma_j^{\out}\longrightarrow \Sigma_{j+1}^{\inn},
\end{equation}
which we call global map, that studies how the trajectories behave close
to the heteroclinic orbit. Then, the map $\BB^j$ considered in
Theorem \ref{theorem:iterative} is just $\BB^j=\BB_\glob^j\circ\BB_\loc^j$.

Before we go into technicalities we write
a table analogous to (\ref{energy-distribution}) of the properties of the local and
global maps. The local map $\BB_\loc^j$, projected onto hyperbolic variables,
has the form
\begin{align} \label{local-map}
 p_1^{(j)} &\sim C^{(j)}\ \de\ \ln \frac 1 \de
 \qquad  \qquad  \qquad & \longrightarrow & \qquad |p_1^{(j)}|\lesssim (C^{(j)} \de)^{1/2}
 \qquad  \notag \\
 q_1^{(j)} &= \sigma \qquad
 \qquad  \qquad  \qquad & \longrightarrow & \qquad |q_1|\lesssim (C^{(j)} \de)^{1/2}
 \qquad  \notag \\
 |p_2^{(j)}| & \lesssim (C^{(j)} \de)^{1/2}
 \qquad  \qquad  \qquad & \longrightarrow & \qquad \,\,\, p_2^{(j)}=\sigma \qquad \qquad
 \qquad   \\
 |q_2^{(j)}| &\lesssim (C^{(j)} \de)^{1/2}
 \qquad  \qquad  \qquad & \longrightarrow & \qquad |q_2^{(j)}|
 \lesssim  C^{(j)}\ \de\ \ln \frac 1 \de. \qquad  \notag
\end{align}
The global map $\BB_\glob^j$, projected onto hyperbolic variables of
the corresponding saddles, has the form
\begin{align} \label{global-map}
 |p_1^{(j)}| &\lesssim (C^{(j)} \de)^{1/2}
 \qquad  \qquad  \qquad & \longrightarrow & \qquad | p_1^{(j+1)} |
 \lesssim C^{(j)} \de \ \ln \frac 1 \de
 \qquad  \notag \\
 |q_1^{(j)}|  &\lesssim (C^{(j)} \de)^{1/2}
 \qquad  \qquad  \qquad & \longrightarrow & \qquad\,\,\,  q_1^{(j+1)}= \sigma \qquad
 \qquad  \notag \\
 p_2^{(j)} &= \sigma \qquad \qquad
 \qquad  \qquad  \qquad & \longrightarrow & \qquad |p_2^{(j+1)}|
 \lesssim  (C^{(j)} \de)^{1/2}
 \qquad   \\
 |q_2^{(j)}| &\lesssim C^{(j)}\ \de\ \ln \frac 1 \de
 \qquad  \qquad  \qquad & \longrightarrow & \qquad |q_2^{(j+1)}|
 \lesssim  (C^{(j)} \de)^{1/2}.  \qquad  \notag
\end{align}

To compose the two maps we need that  the
set $\UU^j$, introduced in \eqref{def:HeteroMap}, has
a modified product-like structure. To define its properties, we
consider the projection
\[
 \wt \pi \left(c_-^{(j)},p_1^{(j)},q_1^{(j)},p_2^{(j)},q_2^{(j)},c_+^{(j)}\right)=
 \left(p_2^{(j)},q_2^{(j)},c_+^{(j)}\right).
\]
\begin{definition}\label{definition:ModifiedProductLike}
Fix constants $r\in (0,1)$, $\de>0$ and $\sigma>0$ and define
a multi-parameter set of positive constants
\[
 \wt\II_j=\left\{\wt C^{(j)},\wt m_{\el}^{(j)},\wt M_{\el,\pm}^{(j)},
\ \wt m_{\adj}^{(j)},\ \wt M_{\adj,\pm}^{(j)},\ \wt m_{\hyp}^{(j)},\
\wt M_{\hyp}^{(j)}\right\}.
\]
%\begin{equation}\label{def:ProductLikeSet2}
% \wt\II_j=\left\{\wt C^{(j)},\wt m_{\el}^{(j)},\ M_{\el,\pm}^{(j)},
%\ \wt m_{\adj}^{(j)},\ \wt M_{\adj,\pm}^{(j)},\ \wt m_{\hyp}^{(j)},\
%\wt M_{\hyp}^{(j)}\right\}.
%\end{equation}
Then, we say that a (non-empty) set $\UU\subset\Sigma_j^\out$ has
a $\wt \II_j$-product-like structure provided it satisfies the following
two conditions:
\begin{description}
 \item[\textbf{C1}]
\[
 \UU\subset \wt\DD_j^1\times\ldots\times\wt\DD_j^{j-2}\times\wt\NNN_{j,-}\times
\wt\DD_j^{j+2}\times\ldots\times\wt\DD_j^{N}
\]
where
\begin{align*}
\wt\DD_j^k&=\left\{\left|c_k^{(j)}\right|\leq  \wt M^{(j)}_{\el,\pm}
\de^{(1-r)/2}\right\} \,\,\text{ for }k\in \PP_j^\pm\\
\wt\DD_j^{j\pm 2}&\subset\left\{\left|c_{j\pm2}^{(j)}\right|\leq
\wt M^{(j)}_{\adj,\pm} \left(\wt C^{(j)}\de\right)^{1/2}\right\},
\end{align*}
and
\[
\begin{split}
\wt\NNN_{j}^+= \Big\{& (p_1^{(j)},q_1^{(j)},p_2^{(j)},q_2^{(j)})\in \RR^4:
\left|p_1^{(j)}\right|,\left|q_1^{(j)}\right|\leq \wt M_\hyp^{(j)}\left(\wt C^{(j)}\de\right)^{1/2},\\
&p_2^{(j)}=\sigma,  -\wt C^{(j)}\,\de\,\left(\ln(1/\de)+\wt M_{\hyp}^{(j)}\right)\leq
q_2^{(j)}\leq -\wt C^{(j)}\,\de\,\left(\ln(1/\de)-\wt M_{\hyp}^{(j)}\right) \Big\},
\end{split}
\]
%\begin{equation}\label{def:Domain:Hyperbolic:Modified}
%\begin{split}
%\wt\NNN_{j}^+= \Big\{& (p_1^{(j)},q_1^{(j)},p_2^{(j)},q_2^{(j)})\in \RR^4:
%\left|p_1^{(j)}\right|,\left|q_1^{(j)}\right|\leq \wt M_\hyp^{(j)}\left(\wt C^{(j)}\de\right)^{1/2},\\
%&p_2^{(j)}=\sigma,  -\wt C^{(j)}\,\de\,(\ln(1/\de)-\wt M_{\hyp}^{(j)})\leq
%q_2^{(j)}\leq -\wt C^{(j)}\,\de\,(\ln(1/\de)+\wt M_{\hyp}^{(j)}) \Big\},
%\end{split}
%\end{equation}
\item[\textbf{C2}]
\[
 \{\sigma\}\times\left[-\wt C^{(j)}\,\de\,\left(\ln(1/\de)-\wt m_{\hyp}^{(j)}\right), -\wt
C^{(j)}\,\de\,\left(\ln(1/\de)+\wt m_{\hyp}^{(j)}\right)\right]\times
\DD_{j,-}^{j+2}\times\ldots\times\DD_{j,-}^{N} \subset\wt\pi(\UU)
\]
where
\begin{align*}
\DD_{j,-}^k&=\left\{\left|c_k^{(j)}\right|\leq  \wt m^{(j)}_{\el}
\de^{(1-r)/2}\right\} \,\,\text{ for }k\in \PP_j^+\\
\DD_{j,-}^{j+ 2}&=\left\{\left|c_{j+2}^{(j)}\right|\leq
\wt m^{(j)}_{\adj} \left(C^{(j)}\de\right)^{1/2}\right\}.
\end{align*}
\end{description}
\end{definition}
With this definition, we can state the following two lemmas.
Combining these two lemmas we deduce Theorem \ref{theorem:iterative}.

\begin{lemma}\label{lemma:iterative:saddle}
Fix any natural $j$ with $3\leq j\leq N-2$, constants $r,r'\in (0,1)$
satisfying  $0<r'<1/2-2r$ and $\sigma>0$ small enough. Take
$\de=e^{-\ga N}$, $\ga=\ga(\sigma)\gg 1$, depending on $\sigma$,
and consider a parameter set $\II_j$ with $M_\hyp^{(j)}\geq 1$ and
a $\II_j$-product-like set $\VV_j\subset\Sigma_j^\inn$. Then, f
or $N$ big enough, there exists:
\begin{itemize}
\item A constant $K>0$ independent of $N$ and $j$ but which might depend on $\sigma$.
\item A parameter set $\wt\II_j$ whose constants satisfy
\[
 \begin{split}
  C^{(j)}/2\leq \wt C^{(j)}\leq 2 C^{(j)}\\
0<\wt m_{\hyp}^{(j)}\leq m_{\hyp}^{(j)}
 \end{split}
\]
and
\[
 \begin{split}
\wt M_\hyp^{(j)}&=K\\
\wt M_{\el,\pm}^{(j)}&=M_{\el,\pm}^{(j)}+K\de^{r'}\\
\wt m_{\el}^{(j)}& = m_{\el}^{(j)}-K\de^{r'}\\
\wt M_{\adj,\pm}^{(j)}&=M_{\adj,\pm}^{(j)}(1+4\sigma)\\
\wt m_{\adj}^{(j)}&=m_{\adj}^{(j)}\ (1-4\sigma),
 \end{split}
\]
\item A  $\wt\II_j$-product-like set $\UU_j$ for which  the map $\BB_\loc^j$
satisfies
\begin{equation}\label{cond:ComposeMaps:Localmap}
\UU_j\subset\BB_\loc^j\left(\VV_j\right).
\end{equation}
%\begin{equation}\label{cond:ComposeMaps:Saddle}
%\UU_j\subset\BB_\loc^j\left(\VV_j\right).
%\end{equation}
\end{itemize}
Moreover, the time to reach the section $\Sigma^{\out}_{j}$ can be bounded
as
\[
 \left|T_{\BB_\loc^j}\right|\leq K\ln(1/\de).
\]
\end{lemma}
The proof of this lemma  is the {\it most delicate part} in the proof of
the Iterative Theorem \ref{theorem:iterative}, since we are passing close
to a hyperbolic fixed point, which implies big deviations. It is split in
several parts in the forthcoming sections to simplify the exposition.
First, in Section \ref{sec:HypToyModel}, we set the elliptic modes $c$ to zero,
and we study the saddle map  associated to the corresponding
system. We call to this system \emph{hyperbolic toy model}. It  has two degrees of
freedom. Then, in Section \ref{sec:FullSystem} we use the results obtained for
the hyperbolic toy model to deal with the full system and prove Lemma
\ref{lemma:iterative:saddle}.

Now we state the iterative lemma for the global maps $\BB^j_\glob$.
\begin{lemma}\label{lemma:iterative:hetero}
Fix any natural $j$ with $3\leq j\leq N-2$, constants $r,r'\in (0,1)$
satisfying  $0<r'<1/2-2r$ and $\sigma>0$ small enough. Take
$\de=e^{-\ga N}$, $\ga=\ga(\sigma)\gg 1$, depending on $\sigma$, and
consider a parameter set $\wt \II_j$ and a $\wt \II_j$-product-like
set $\UU_j\subset\Sigma_j^\out$. Then, for $N$ large enough, there exists:
\begin{itemize}
\item A constant $\wt K$
depending on $\sigma$, but independent of $N$ and $j$.
\item A parameter set $\II_{j+1}$ whose constants
satisfy
\[
 \begin{split}
  \wt C^{(j)}/\wt K\leq C^{(j+1)}\leq \wt K\wt C^{(j)}\\
0<m_{\hyp}^{(j+1)}\leq \wt m_\hyp^{(j)}
 \end{split}
\]
and
\[ \begin{split}
M_{\el,-}^{(j+1)}&=\max\left\{\wt M_{\el,-}^{(j)}+\wt K\de^{r'},\wt
K\wt M_{\adj,-}^{(j)}\right\}\\
M_{\el,+}^{(j+1)}&=\wt M_{\el,+}^{(j)}+\wt K\de^{r'}\\
m_{\el}^{(j+1)}&=\wt m_{\el}^{(j)}-\wt K\de^{r'}\\
M_{\adj,+}^{(j+1)}&=\wt M_{\el,+}^{(j)}+\wt K\de^{r'}\\
M_{\adj,-}^{(j+1)}&=\wt K\wt M_{\hyp}^{(j)}\\
m_{\adj}^{(j+1)}&=\wt m_{\el}^{(j)}+\wt K\de^{r'}\\
M_{\hyp}^{(j+1)}&=\max\left\{\wt K \wt M_{\adj,+}^{(j)},\wt K\right\}
 \end{split}
\]
\item A  $\II_{j+1}$-product-like set
$\VV_{j+1}\subset\Sigma_{j+1}^\inn$  for which the map $\BB_\glob^j$ satisfies
\begin{equation}\label{cond:ComposeMaps:Heteromap}
\VV_{j+1}\subset\BB_\glob^j\left(\UU_j\right).
\end{equation}
\end{itemize}
Moreover, the time spent to reach the section $\Sigma^{\inn}_{j+1}$
can be bounded as
\[
 \left|T_{\BB_\glob^j}\right|\leq \wt K.
\]
\end{lemma}
The proofs of this lemma is postponed to Section \ref{sec:ProofHeteroMap}.

Now it only remains to deduce from Lemmas \ref{lemma:iterative:saddle}
and \ref{lemma:iterative:hetero}  the Iterative Theorem \ref{theorem:iterative}.
\begin{proof}[Proof of Theorem \ref{theorem:iterative}]
We choose the multiindex $\II_3$ so that we can apply iteratively the
Lemmas \ref{lemma:iterative:saddle} and \ref{lemma:iterative:hetero}. Indeed,
from the recursive formulas in Lemma \ref{lemma:iterative:saddle} and
\ref{lemma:iterative:hetero} it is clear that it is enough to chose a parameter
set $\II_3$ satisfying
\[
 1<M_{\el,+}^{(3)}\ll M_{\adj,+}^{(3)}\ll
M_{\hyp}^{(3)}\ll M_{\adj,-}^{(3)}\ll M_{\el,-}^{(3)}
\]
and
\[
 0<m_\el^{(3)}<3 m_\adj^{(3)}.
\]
From the choice of the constants in $\II_3$ and
the recursion formulas in Lemmas \ref{lemma:iterative:saddle}
and \ref{lemma:iterative:hetero}, we have that $M_\hyp^{(j)}\geq 1$ for any $j=3,\ldots N-1$.
This fact
along with conditions \eqref{cond:ComposeMaps:Localmap} and
\eqref{cond:ComposeMaps:Heteromap}, allow us to apply
Lemmas \ref{lemma:iterative:saddle} and \ref{lemma:iterative:hetero}
iteratively so that we obtain the
$(\de,\sigma,K)$-recursive collection
of multi-parameter sets $\{\II_j\}_{j=3,\ldots,N-1}$ and
the $\II_j$-product-like sets $\VV_j\subset\Sigma_j^\inn$.
In particular, note that the recursion formulas stated in
Theorem \ref{theorem:iterative} can be easily deduced from
the recursion formulas given in Lemmas \ref{lemma:iterative:saddle}
and \ref{lemma:iterative:hetero} and the choice of $\II_3$.

Finally, we bound the time
\[
  \left|T_{\BB^j}\right|\leq  \left|T_{\BB_\loc^j}\right|+
  \left|T_{\BB_\glob^j}\right|\leq(K+\wt K)\ln(1/\de).
\]
This
completes the proof of Theorem \ref{theorem:iterative}.
\end{proof}

\section{The hyperbolic toy model}\label{sec:HypToyModel}
In this section we set the elliptic modes to zero, namely, we deal with the
system
\begin{equation}\label{def:VF:HypToyModel}
\begin{split}
\dot p_1&=\sqrt{3}p_1+\ZZZ_{\hyp, p_1}\\
\dot q_1&=-\sqrt{3}q_1+\ZZZ_{\hyp, q_1}\\
\dot p_2&=\sqrt{3}p_2+\ZZZ_{\hyp, p_2}\\
\dot q_2&=-\sqrt{3}q_2+\ZZZ_{\hyp, q_2},
\end{split}
\end{equation}
where the functions $\ZZZ_{\hyp,\ast}$ are defined in
\eqref{def:VF:Diag:p1:Hyp}, \eqref{def:VF:Diag:q1:Hyp},
\eqref{def:VF:Diag:p2:Hyp} and \eqref{def:VF:Diag:q2:Hyp}.

We start by setting some notation.
We  call
\[
 z=(x_1,y_1,x_2,y_2)
\]
the new set of coordinates, whose components are also denoted by
$z_i=(x_i,y_i)$. We also use the notation $x=(x_1,x_2)$ and $y=(y_1,y_2)$.

Moreover, we  call $K$ to any positive constant independent of $\de$, $N$,
$j$, and $\sigma$ and we call $K_\sigma$ to any positive constant
%M
%independent of $\de$, $N$ and $j$ but which might depend on $\sigma$.
depending on $\sigma$, but independent of $\de$, $N$ and $j$
Analogously, we say that $a=\OO(b)$ if $|a|\leq K|b|$ and that
$a=\OO_\sigma(b)$ if $|a|\leq K_\sigma|b|$. We will also use all these
notations in Section \ref{sec:FullSystem} and Section \ref{sec:ProofHeteroMap}.

The first step is to perform a resonant $\CCC^k$ normal form in a neighborhood of size $\sigma$ of the saddle. Note that  we do not
need much regularity for the normal form since all our study will
be done in the $\CCC^0$ norm.
%M Thus,
It turns out it is enough to consider a $\CCC^1$ normal form.
Before we state our next claim about the normal form we formulate
a well known result of Bronstein-Kopanskii \cite{BronsteinK92} about
finitely smooth normal forms of vector fields near a critical point.
We are unable to use classical results about linearizability,
because our saddle is {\it resonant}.

The main result of Bronstein-Kopanskii \cite{BronsteinK92} is that near
a saddle point a vector field can be transformed into a polynomial
one by a finitely smooth change of coordinates with only
certain (resonant) monomials present.
For convenience of the reader we use notations of this paper.

\subsection{Finitely smooth 
%M
polynomial normal forms of vector fields in
near a saddle point}

Let $\dot x=F(x)$ be a vector with the origin being a critical point,
i.e. $F(0)=0,\ x\in \RR^d$ for some $d \in \ZZ_+$. Assume that $F$ is
$C^K$ for some positive integer $K\in \ZZ_+$, i.e. $F$ has all partial
derivatives of order up to $K$ uniformly bounded. Denote
the linearization of $F$ at $0$ by $A:=DF(0)$ and $f(x)=F(x)-A(x)$.
Then, the equation becomes
\[
\dot x=Ax+f(x),\ f(0)=0,\ Df(0)=0.
\]
Let $\nu_1,\dots,\nu_d$ denote the eigenvalues of $A$ and
$\tet_1,\dots,\tet_n$ be all distinct numbers contained in the set
$\{\Re \nu_i: \ i=1,\dots,d\}$. Assume that none of $\tet_i$'s
is zero or, in other words, the rest point being hyperbolic.

% Vadim: that's the part that I don't see why is necessary
%Introduce new notations for the number $\theta_1,\dots,\theta_n$ as
%follows
%\[
%-\la_l < \dots < - \la_1 <0<\mu_1<\dots<\mu_m
%\quad (m+l=n).
%\]
The space $\RR^d$ can be represented as a direct sum of
$A$-invariant subspaces $E_1,\dots,E_n$ such that the eigenvalues
of the operator $A|_{E_i}$ satisfy the condition $\Re \nu_i=\tet_i$.

\begin{theorem} \cite{BronsteinK92}
Let $k$ be positive integer. Assume that the vector field
$\dot x=F(x)$ is of class $C^K,\ x=0$ is a hyperbolic saddle
point and $A=DF(0)$. If $\ K\ge Q(k)$ for some computable
function $Q(\cdot)$, then, for some positive integer $N$,
this vector field near the point $x=0$ can be reduced by a transformation
$y=\Phi(x), \Phi\in C^k$,
to the polynomial resonant normal form
\[
\dot y=Ay+\sum_{|\tau|=2}^N p_\tau y^\tau,
\]
where $\tau\in \ZZ^d_+$ and $p_\tau$ denotes a multi-homogeneous
polynomial 
%from 
$p_\tau(E_1,\dots,E_n;E_1\oplus \dots \oplus E_n),
\ p_\tau=(p^1_\tau,\dots,p^d_\tau)$ and $p_\tau^i \ne 0$
implies $\nu_i=\tau^1 \nu_1+\dots+\tau^d \nu_d$
(by the resonant condition).
\end{theorem}

In Theorem 3 \cite{BronsteinK92} the authors give an upper
bound on $N$. In our case $d=4,\ n=2,\ k=1$.
A direct application of this Theorem is the following

%Vadim: old version. If we do not define lambda's and mu's, we don't need to use m and l either.
%In Theorem 3 \cite{BronsteinK92} the authors give an upper
%bound on $N$. In our case $d=4,\ m=l=1,\ n=2,\ k=1$.
%A direct application of this Theorem is the following

\begin{lemma}\label{lemma:ToyModel:NormalForm}
There exists a $\CCC^1$ change of coordinates
\[
 (p_1,q_1,p_2,q_2)=\Psi_\hyp(x_1,y_1,x_2,y_2)=(x_1,y_1,x_2,y_2)+
 \wt\Psi_\hyp(x_1
,y_1,x_2,y_2)
\]
which transforms the vector field \eqref{def:VF:HypToyModel} into the vector
field
\begin{equation}\label{def:HypHam:Transformed}
\XX_\hyp(z)=D z+R_\hyp,
\end{equation}
where $D$ is the diagonal matrix
$D=\mathrm{diag}(\sqrt{3},-\sqrt{3},\sqrt{3},-\sqrt{3})$ and $R_\hyp$ is a
polynomial, which only contains resonant monomials
%M
\footnote{One can even estimate degree of this polynomial using 
\cite{BronsteinK92}}. It can be split as
\begin{equation}\label{def:HypHam:Hyp}
R_\hyp=R_\hyp^0+R_\hyp^1,
\end{equation}
where $R_\hyp^0$ is the first order, which is given by
\begin{equation}\label{def:Hyp:Transf0}
R_\hyp^0(z)=\left(\begin{array}{c}R_{\hyp,x_1}^0(z)\\R_{\hyp,y_1}^0(z)\\R_{\hyp,
x_2}^0(z)\\R_{\hyp,y_2}^0(z)\end{array}\right)=\left(\begin{array}{c}
2\nu_{2}x_1^2y_1+2\nu_{02}y_1x_2^2+\nu_{11}x_1x_2y_2 \\
-2\nu_{2}x_1y_1^2-2\nu_{20}x_1y_2^2-\nu_{11}y_1x_2y_2\\2\nu_{2}y_2x_2^2+2\nu_{20}
x_1^2y_2
+\nu_{11}x_1y_1x_2\\-2\nu_{2}x_2y_2^2-\nu_{02}y_1^2x_2-\nu_{11}
x_1y_1y_2\end{array}\right),
\end{equation}
and $R_\hyp^1$ is the remainder and satisfies
\begin{equation}\label{def:Hyp:Remainder}
 R_{\hyp,x_i}^1=\OO\left(x^3y^2\right)\,\,\,\text{ and
}\,\,\,R_{\hyp,y_i}^1=\OO\left(x^2y^3\right).
\end{equation}
Moreover, the function $\wt\Psi_\hyp=( \wt\Psi_{\hyp,x_1}, \wt\Psi_{\hyp,y_1},
\wt\Psi_{\hyp,x_2}, \wt\Psi_{\hyp,y_2})$ satisfies
\[
\begin{split}
 \wt\Psi_{\hyp,x_1}(z)&=\OO\left(x_1^3,x_1y_1,x_1(x_2^2+y_2^2),
y_1y_2(x_2+y_2)\right)\\
\wt\Psi_{\hyp,y_1}(z)&=\OO\left(y_1^3,x_1y_1,y_1(x_2^2+y_2^2),
x_1x_2(x_2+y_2)\right)\\
 \wt\Psi_{\hyp,x_2}(z)&=\OO\left(x_2^3,x_2y_2,x_2(x_1^2+y_1^2),
y_1y_2(x_1+y_1)\right)\\
\wt\Psi_{\hyp,y_2}(z)&=\OO\left(y_2^3,x_2y_2,y_2(x_1^2+y_1^2),
x_1x_2(x_1+y_1)\right).
\end{split}
\]
\end{lemma}

%\begin{proof}
%The proof can be seen in \cite{BronsteinK94}.
%\end{proof}

%M OLD
%Recall that we want to study the evolution of points which belong to the section
%$\Sigma_j^0$ with $c=0$, which is just a set $\NNN_j$ that satisfies
%$\NNN_j^-\subset\NNN_j\subset\NNN_j^+$, where $\NNN_j^\pm$ are the open sets
%defined in \eqref{def:Domain:Hyperbolic:inf} and
%\eqref{def:Domain:Hyperbolic:sup}. Since we have performed the change
%$\Psi_\hyp$ we study the flow associated to the vector field
%\eqref{def:HypHam:Transformed} in a slightly different domain $\wt\NNN_j$.
%In Section \ref{sec:FullSystem} we will choose the function $g$ involved in
% definition \eqref{def:Domain:Hyperbolic:sup} so that
%$\NNN_j\subset\Psi_\hyp^{-1}(\wt\NNN_j)$.  To define the set $\wt \NNN_j$
%we use the constant $C^{(j)}$ given in Lemma \ref{lemma:iterative:saddle} and the
%inverse of the map $\Psi$ defined in Lemma \ref{lemma:ToyModel:NormalForm},
%which we denote by $\Upsilon=\mathrm{Id}+\wt\Upsilon$. Then, we define the set

\subsection{The local map for the hyperbolic toy model in the normal form variables}
Recall that our goal in this step of the proof is to study the
evolution of points with initial conditions inside of a certain set
near the section $\Sigma_j^{\inn}$. More specifically, in formulas
\eqref{def:Domain:Hyperbolic:sup} and \eqref{def:Domain:Hyperbolic:inf}
we define sets $\NNN_j^-\subset\NNN_j^+$. We set elliptic modes $c=0$
and shall study the set $\NNN'_j$ satisfying
\[
\NNN_j^-\cap\{c=0\}\subset \NNN_j' \subset\NNN_j^+\cap\{c=0\}.
\]
Since the analysis is done in normal coordinates $\Psi_\hyp:(x,y)\to(p,q)$,
we study the a set $\wh \NNN_j$ such that $\Psi_\hyp^{-1}(\NNN_j')\subset\wh \NNN_j$. To define
this set we need to fix several parameters and define several objects.

Let $C^{(j)}$'s be the constant from Lemma \ref{lemma:iterative:saddle}.
Recall that in Definition \ref{def:RecursivelyDefined} we define a
$(\sigma,\de,K)$-recursive multiparameter set $\II_j$. Its description
includes parameters $M_{\hyp}^{(j)}$ used below.  The parameter $K$
depends on $\sigma$ and we keep this dependence in the notation: $K_\sigma$.
Denote the inverse of the map $\Psi$ from Lemma \ref{lemma:ToyModel:NormalForm},
by
\[
\Upsilon:=\mathrm{Id}+\wt\Upsilon:=\Psi_\hyp^{-1}=:
\mathrm{Id}+(\wt\Upsilon_{x_1},\wt\Upsilon_{y_1},\wt\Upsilon_{x_2},
\wt\Upsilon_{y_2}).
\]
Define
\begin{equation}\label{def:HypToyModel:Chat}
 \wh C^{(j)}:=C^{(j)}\left(1+\pa_{x_1}\wt\Upsilon_{x_1}(0,\sigma,0,0)\right).\\
\end{equation}
Notice that $\wh C^{(j)}=C^{(j)}\left(1+\OO(\sigma)\right).$
Define $f_1(\sigma)$ by
\begin{equation}\label{def:HypToyModel:f1}
f_1(\sigma)= \Upsilon_{y_1}(0,\sigma,0,0).
\end{equation}
Observe that it satisfies $f_1(\sigma)=\sigma+\OO(\sigma^3)$ and
the section $\{y_1=f_1(\sigma)\}$ approximates the image of the section
$\Upsilon(\Sigma_j^{\inn})$. Now we can define the set of points
whose evolution under the local map we shall analyze
\begin{equation}\label{def:Hyp:ModifiedDomain}
\begin{split}
\wh\NNN_j=\Big\{& |x_1+\wh
C^{(j)}\,\de\,(\ln(1/\de)|\le \wh C^{(j)}\,\de\,K_\sigma, \quad
\left|x_2-x_2^\ast\right|\leq 2\, M_{\hyp}^{(j)}\frac{\left(\wh
C^{(j)}\de\right)^{1/2}}{\ln(1/\de)},\\
 & |y_1- f_1(\sigma)|\le
K_\sigma\wh C^{(j)}\de\ln(1/\de),\quad \qquad |y_2|\leq 2\,M_{\hyp}^{(j)}\left(\wh
C^{(j)}\de\right)^{1/2}\Big\},
\end{split}
\end{equation}
where the constant $x_2^\ast$ will be defined later in this section.
It turns out a proper choice of $x_2^\ast$ leads
to a cancelation in the evolution of the $x_1$ coordinate
(described in Section \ref{sec:MainIdeasSaddle} for the simplified model).
This cancelation is crucial to obtain good estimates for the map $\BB_\loc^j$.

We also define the function $f_2(\sigma)$ as
\begin{equation}\label{def:HypToyModel:f2}
f_2(\sigma)= \Upsilon_{x_2}(0,0,\sigma,0).
\end{equation}
By analogy with $f_1(\sigma)$ notice that
the section $\{x_2=f_2(\sigma)\}$ approximates the image of the section
$\Upsilon(\Sigma_j^{\out})$ with $\Sigma_j^{\out}=\{p_2=\sigma\}$.
%M
%Notice that the image of the section $\{p_2=\sigma\}$ is
%the $\OO(\sigma)$-close to $\{x_2=f_2(\sigma)\}$. Of course the section has
%more deviation, but depending on $\de$, which we omit at this step (we will
%adjust it later on).
Later we need to compute an approximate transition time $T_j(x_2)$ from near
$\Upsilon(\Sigma_j^{\inn})$  to $\Upsilon(\Sigma_j^{\out})$.
We use $f_2$ to do that. Notice that
the $x_2$ coordinate behaves almost linearly as
\[
 x_2\sim x_2^0 e^{\sqrt{3}t}.
\]
Therefore, for an orbit to reach  $\{x_2=f_2(\sigma)\}$ it takes an
approximate time
\begin{equation}\label{def:FixedTimeSaddle}
 T_j\left(x_2^0\right)=\frac{1}{\sqrt{3}}\ln\left(\frac{f_2(\sigma)}{x_2^0}
\right).
\end{equation}
Note that this time is defined for any $x_2^0>0$. We will see that the
$x_2^0$ coordinate behaves  as $x_2^0\sim (\wh C^{(j)}\de)^{1/2}$ and,
therefore, $T_j$ behaves as
\[
 T_j\sim\ln\frac{1}{\wh C^{(j)}\de}.
\]
Even if $x_2$ behaves approximately as for a linear system,
this is not the case for the other variables, as we have explained
in Section \ref{sec:MainIdeasSaddle} with a simplified model.
Indeed, if one first considers the linear part of the vector field
\eqref{def:VF:HypToyModel}, omiting the dependence on $\wh
C^{(j)}$,  the transition map sends points
\[
\left(x_1,y_1,x_2,y_2\right)\sim \left(
\OO(\de\ln(1/\de)),\OO(\sigma),
\OO\left(\de^{1/2}\right), \OO\left(\de^{1/2}\right)\right)
\]
to
\[
\left(x_1,y_1,x_2,y_2\right)\sim
\left(\OO\left(\de^{1/2}\ln(1/\de)\right), \OO\left(\de^{1/2}\right),
\OO(\sigma),
\OO(\de)\right).
\]
However, the resonance implies a certain deviation from the heteroclinic orbits.
Indeed, one can see that tipically, the image point is of the form
\[
\left(x_1,y_1,x_2,y_2\right)\sim
\left(\OO\left(\de^{1/2}\ln(1/\de)\right), \OO\left(\de^{1/2}\right),
\OO(\sigma),
\OO(\de\ln(1/\de)\right).
\]
This apparently small deviation, after undoing the normal form, would imply
a considerably big deviation from the heteroclinic orbit and would lead
to very bad estimates. Nevertheless, if one chooses carefully $x_2$ in
terms of $x_1$ and $y_1$, one can obtain a cancelation that leads to
an image point of the form
\[
\left(x_1,y_1,x_2,y_2\right)\sim
\left(\OO\left(\de^{1/2}\right), \OO\left(\de^{1/2}\right), \OO(\sigma),
\OO(\de\ln(1/\de)\right).
\]
Since the points we are dealing with belong to the set $\wt\NNN_j$
defined in \eqref{def:Hyp:ModifiedDomain}, this cancellation boils down
 to choosing a suitable constant $x_2^\ast$. Next lemma shows that
a particular choice of $x_2^\ast$ leads to a cancellation that allow
us to obtain good estimates for the saddle map in spite of the resonance.
The choice we do is essentially the same as the one choosen in
Section \ref{sec:MainIdeasSaddle} for the simplified model that
has been considered in that section.

\begin{lemma}\label{lemma:HypToyModel:SaddleMap}
Let us consider the flow $\Phi^\hyp_t$ associated to \eqref{def:HypHam:Transformed}
and a point $z^0\in\wh\NNN_j$. Then, if we choose $x_2^\ast$ as the unique
positive solution of
\begin{equation}\label{def:ChoiceEtaHat}
\left(x_2^\ast\right)^2 T_j(x_2^\ast)=\frac{\wh
C^{(j)}\,\de\,\ln(1/\de)}{2\,\nu_{02}\,f_1(\sigma)}
\end{equation}
and we take $\de$ and $\sigma$ small enough, the point
\[
 z^f=\Phi^\hyp_{T_j}(z^0),
\]
where $T_j=T_j(x_2^0)$ is the time defined in \eqref{def:FixedTimeSaddle},
satisfies
\[
\begin{split}
 |x_1^f|\qquad \qquad \leq &\ K_\sigma\left(\wh C^{(j)}\de\right)^{1/2}\\
  |y_1^f|\qquad \qquad \leq&\ K_\sigma\left(\wh C^{(j)}\de\right)^{1/2}\\
|x_2^f-f_2(\sigma)|\qquad \leq &\
K_\sigma\left(\wh C^{(j)}\de\right)^{1/2}\ln^2(1/\de)\\
\left|y_2^f+\frac{f_1(\sigma)}{f_2(\sigma)}\wh C^{(j)}\de\ln(1/\de)
\right| \leq &\ K_\sigma\wh C^{(j)}\de.
\end{split}
\]
\end{lemma}
\begin{remark}\label{remark:cancellation}
%M
%As we have explained the particular choice of $x_2^\ast$ as
%a solution \eqref{def:ChoiceEtaHat} will imply a cancellation
%which will allow us to obtain good estimates for the local map.
The particular choice of $x_2^\ast$ being a solution \eqref{def:ChoiceEtaHat}
will ensure a cancellation. This cancellation is crucial to obtain good
estimates for the local map.

Equation \eqref{def:ChoiceEtaHat} has real solutions because
%M we have assumed
$\nu_{02}>0$ (see Remark \ref{remark:FormOfVF})
%. Indeed, in looking solutions of these equations is where we use the fact that
%we always deal with $x_1<0$ (and $p_1<0$ in the original variables,
and $x_1<0$ (and $p_1<0$ in the original variables, see Remark \ref{remark:SignP1}).
Indeed, if $x_1>0$ and $x_1\sim \wh C^{(j)}\de\ln(1/\de)$ we have
\[
\left(x_2^\ast\right)^2 T_j(x_2^\ast)=-\frac{\wh
C^{(j)}\,\de\,\ln(1/\de)}{2\,\nu_{02}\,f_1(\sigma)}.
\]
%and thus, the absence of solutions of this equation would prevent us to obtain the %desired cancellation.
If there is no solution to this equation, we cannot attain the desired cancellation.
\end{remark}

Let us point out that taking into account the estimates for the points in
$\wh\NNN^{(j)}$, the definition of $T_j$ in \eqref{def:FixedTimeSaddle} and
condition \eqref{def:UpperBoundCs}, one can deduce that condition
\eqref{def:ChoiceEtaHat}  implies
\[
 |x_2^\ast|\leq K_\sigma\left(\wh C^{(j)}\de\right)^{1/2}\leq
K_\sigma\de^{(1-r)/2}.
\]
and then,
\begin{equation}\label{def:BoundSaddleTime}
T_j(x_2^0) \leq K_\sigma\ln(1/\de).
\end{equation}
We  use this estimate throughout the proof of Lemma
\ref{lemma:HypToyModel:SaddleMap}. Note also that for the modes $(x_1^f,y_1^f)$
we just need upper bounds, since after the passage of the saddle $j$, the
associated mode will become elliptic and therefore we will not need accurate
estimates anymore.

\begin{proof}[Proof of Lemma \ref{lemma:HypToyModel:SaddleMap}]
We prove the lemma using a fixed point argument. We look for
a contractive operator using the variation of constants formula.
Namely, we perform the change of coordinates
\begin{equation}\label{def:HypVariationConstants}
 x_i=e^{\sqrt{3}t}u_i,\,\,y_i=e^{-\sqrt{3}t}v_i
\end{equation}
and then we obtain the integral equations
\begin{equation}\label{def:Hyp:FixedPoint}
\begin{split}
u_i&=x_i^0+\int_0^T e^{-\sqrt{3}t}R_{\hyp,x_i}\left(u e^{\sqrt{3}t},v
e^{-\sqrt{3}t}\right)dt\\
v_i&=y_i^0+\int_0^T e^{\sqrt{3}t}R_{\hyp,y_i}\left(u e^{\sqrt{3}t},v
e^{-\sqrt{3}t}\right)dt.
\end{split}
\end{equation}
%M
In the linear case $u_i$'s and $v_i$'s are fixed. We use these variables
%M  make
to find a fixed point argument. We define the contractive operator in
two
%M several
steps. This approach is inspired by Shilnikov \cite{Shilnikov67}.

First we define an
%M operator which is not contractive,
auxiliary (non-contractive) operator we follows
% but will help  to define a slightly modified one which will be contractive.
%we define
\[
\FF_\hyp=(\FF_{\hyp,u_1},\FF_{\hyp,v_1},\FF_{\hyp,u_2},\FF_{\hyp,v_2})
\]
 as
\begin{equation}\label{def:Hyp:Operator1}
\begin{split}
 \FF_{\hyp,u_i}(u,v)&=x_i^0+\int_0^T e^{-\sqrt{3}t}R_{\hyp,x_i}\left(u
e^{\sqrt{3}t},v e^{-\sqrt{3}t}\right)dt\\
 \FF_{\hyp,v_i}(u,v)&= y_i^0+\int_0^T e^{\sqrt{3}t}R_{\hyp,y_i}\left(u
e^{\sqrt{3}t},v e^{-\sqrt{3}t}\right)dt.
\end{split}
\end{equation}
One can easily see that in the $u_1$ and $v_2$ components the main terms
are {\it not} given by the initial condition {\it but by the integral terms}.
%M
%Indeed, as we have explained, the dynamics close to the saddle is not well
%approximated by its linear behavior due to the resonance.
This indicates that the dynamics near the saddle is {\it not } well
approximated by the linearized  dynamics and
%. This implies that
the operator is not contractive.

%Thus,
We modify slightly two of the components of $\FF_\hyp$ and obtain
a contractive operator. We define a new operator
\[
\wt\FF_\hyp=(\wt\FF_{\hyp,u_1},\wt\FF_{\hyp,v_1},\wt\FF_{\hyp,u_2},
\wt\FF_{\hyp, v_2})
\]
as
\begin{equation}\label{def:Hyp:NewOperator}
\begin{split}
\wt
\FF_{\hyp,u_1}(u_1,v_1,u_2,v_2)&=\FF_{\hyp,u_1}(u_1,\FF_{\hyp,v_1}(u_1,v_1,u_2,
v_2),\FF_{\hyp,u_2}(u_1,v_1,u_2,v_2),v_2)\\
\wt \FF_{\hyp,v_1}(u_1,v_1,u_2,v_2)&=\FF_{\hyp,v_1}(u_1,v_1,u_2,v_2)\\
\wt \FF_{\hyp,u_2}(u_1,v_1,u_2,v_2)&=\FF_{\hyp,u_2}(u_1,v_1,u_2,v_2)\\
\wt
\FF_{\hyp,v_2}(u_1,v_1,u_2,v_2)&=\FF_{\hyp,v_2}(u_1,\FF_{\hyp,v_1}(u_1,v_1,u_2,
v_2),\FF_{\hyp,u_2}(u_1,v_1,u_2,v_2),v_2)
\end{split}
\end{equation}
Note that the fixed points of these operators are exactly the same as the fixed
points of  $\FF_\hyp$. Thus, the fixed points of the operator $\wt\FF_\hyp$
are solutions of equation \eqref{def:Hyp:FixedPoint}.

%M The operator $\wt\FF_\hyp$ is contractive in a suitable Banach space.
It turns out the operator $\wt\FF_\hyp$ is contractive in a suitable
Banach space. We define the following %M
weighted norms. To fix notation, we denote by $\|\cdot\|_\infty$
the
%M classical
standard supremum norm. Then
%, taking into account the size that each mode has, we define
define
\begin{equation}\label{def:Hyp:Norms}
\begin{split}
\|h\|_{\hyp, u_1}=&\sup_{t\in [0,T_j]}\left|\left(-\wh
C^{(j)}\de\ln(1/\de)+2\nu_{02}f_1(\sigma)\left(x_2^\ast\right)^2 t+\wh
C^{(j)}\de\right)^{-1}h(t)\right| \\
\|h\|_{\hyp, v_1}=&\,f_1(\sigma)\ii\|h\|_\infty \\
\|h\|_{\hyp, u_2}=&\left(x_2^\ast\right)\ii\|h\|_\infty \\
\|h\|_{\hyp, v_2}=&\left(\left(y_1^0\right)^2x_2^0 T_j\right)\ii\|h\|_\infty
\end{split}
\end{equation}
and the norm
\begin{equation}\label{def:Hyp:FullNorm}
\|(u,v)\|_\ast=\sup_{i=1,2}\left\{\|u_i\|_{\hyp,u_i},\|v_i\|_{\hyp,v_i}\right\}.
\end{equation}
%M Then, we define the Banach space
This gives rise to the following Banach space
\[
\YY_\hyp=\left\{(u,v):[0,T]\rightarrow \RR^4;  \|(u,v)\|_\ast<\infty\right\}.
\]
%\begin{equation}\label{def:Hyp:Banach}
%\YY_\hyp=\left\{(u,v):[0,T]\rightarrow \RR^4;  \|(u,v)\|_\ast<\infty\right\}.
%\end{equation}
The contractivity of $\wt\FF_\hyp$ is a consequence of the following two
%M it is not good to use lemmas to prove a lemma
%technical lemmas.
auxiliary propositions.

\begin{proposition}\label{lemma:Hyp:FirstIteration}
Assume \eqref{def:ChoiceEtaHat}, then there exists a constant
$\kk_0>0$ independent of $\sigma$, $\de$ and $j$ such that for $\de$ and $\sigma$ small enough, the operator $\wt \FF_\hyp$ satisfies
\[
 \|\wt \FF (0)\|_\ast\leq \kk_0.
\]
\end{proposition}

\begin{proposition}\label{lemma:Hyp:Contractive}
Consider $w, w'\in B(2\kk_0)\subset\YY_\hyp$ and let us assume
\eqref{def:ChoiceEtaHat}, then taking $\de\ll\sigma$, the operator $\wt
\FF_\hyp$ satisfies
\[
 \|\wt \FF_\hyp(w)-\wt \FF_\hyp(w')\|_\ast\leq K_\sigma\left(\wh
C^{(j)}\de\right)^{1/2}\ln^2(1/\de) \|w-w'\|_\ast.
\]
\end{proposition}

These two propositions show that $\wt\FF_\hyp$ is contractive from
$B(2\kk_0)\subset\YY_\hyp$ to itself. Moreover, using them we can deduce accurate
estimates for the image point. We prove here Proposition
\ref{lemma:Hyp:FirstIteration}. The proof of Proposition \ref{lemma:Hyp:Contractive}
is deferred to the end of the section.

\begin{proof}[Proof of Proposition \ref{lemma:Hyp:FirstIteration}]
 We bound each mode separately. For $\wt \FF_{\hyp,v_1}$ and $\wt \FF_{\hyp,u_2}$, we have that
\[
\wt \FF_{\hyp,v_1}(0)=y_1^0\,\,\,\text{ and }\,\,\,\wt \FF_{\hyp,u_2}(0)=x_2^0
\]
and therefore, they satisfy the desired bounds. Now we bound the first iteration
for  $u_1$. Here we use the particular choice of $x_2^0$ in terms of
$(x_1^0,y_1^0)$ done in \eqref{def:ChoiceEtaHat} to obtain the desired
cancellations (see Remark \ref{remark:cancellation}). Indeed, taking into account the properties of $R_{\hyp,x_1}$ given in Lemma \ref{lemma:ToyModel:NormalForm}, the first iteration is just
\[
\begin{split}
\wt  \FF_{\hyp,u_1}(0)(t)=&x_1^0+\int_0^t\left(2\nu_{02}
y_1^0(x_2^0)^2+\OO((y_1^0)^2(x_2^0)^3\right)dt\\
&x_1^0+2\nu_{02}y_1^0 (x_2^0)^2t+\OO\left(y_1^0)^2(x_2^0)^3\right).
\end{split}
\]
Therefore, taking into account that $z^0\in\wh \NNN_j$ (see \eqref{def:Hyp:ModifiedDomain}) and also \eqref{def:BoundSaddleTime}, we have that
\[
\wt  \FF_{\hyp,u_1}(0)(t)=-\wh C^{(j)}\de\ln(1/\de)+2\nu_{02}f_1(\sigma) (x_2^\ast)^2t+\OO\left(\wh C^{(j)}\de\right).
\]
Thus, applying the norm given in
\eqref{def:Hyp:Norms}, we have that there exists a constant
$\kk_0>0$ such that
\[
 \left\|\wt \FF_{\hyp,u_1}(0)\right\|_{\hyp,u_1}\leq \kk_0.
\]
To bound the first iteration for $v_2$, we just have to take into account that
it is given by
\[
 \wt
\FF_{\hyp,v_2}(0)(t)=y_2^0-\int_0^t\left(2\nu_{02}
x_2^0(y_1^0)^2+\OO\left(\left(y_1^0\right)^3\left(x_2^0\right)^2\right)\right)dt.
\]
Then, recalling that $z^0\in\wh \NNN_j$,
\[
 \left|\wt \FF_{\hyp,v_2}(0)(t)\right|\leq 4\nu_{02}x_2^0(y_1^0)^2 T_j,
\]
which gives
\[
 \left\|\wt \FF_{\hyp,v_2}(0)\right\|_{\hyp,v_2}\leq 4\nu_{02}.
\]
Therefore, we can conclude that
\[
 \left\|\wt\FF(0)\right\|_\ast\leq \kk_0
\]
for certain constant $\kk_0>0$ independent of $\de$, $\sigma$ and $j$.
\end{proof}

The previous two Propositions show that $\wt\FF_\hyp$ is contractive
from $B(2\kk_0)\subset\YY_\hyp$ to itself. Therefore, it has a unique
fixed point in $B(2\kk_0)\subset\YY_\hyp$ which we denote by $w^*$.
Now it only remains to deduce the bounds for $z^f$ stated in Lemma
\ref{lemma:HypToyModel:SaddleMap}. To this end, we use the contractivity
of the operator $\wt\FF_\hyp$ and we undo the change
\eqref{def:HypVariationConstants}. Using the definition of $T_j$
in \eqref{def:FixedTimeSaddle}, we obtain
\[
\begin{split}
x_2^f=&e^{\sqrt{3}T_j}v_2(T_j)\\
=&\frac{f_2(\sigma)}{x_2^0}\left(x_2^0+\wt\FF_{\hyp,v_2}(w^*)(T_j)-\wt\FF_{\hyp,
v_2}(0)(T_j)\right)\\
=&f_2(\sigma)\left(1+\OO\left(\left(\sigma \wh
C^{(j)}\de\right)^{1/2}\ln^2(1/\de)\right)\right)
\end{split}
\]
Analgously, one can see that
\[
|y_1^f|\leq K_\sigma\left(\wh C^{(j)}\de\right)^{1/2}.
\]
To obtain the estimates for $x_1^f$, note that the particular choice that we
have done for $x_2^\ast$ in \eqref{def:ChoiceEtaHat} implies that
\[
\begin{split}
 |u_1(T_j)|\leq
&\left|\wt\FF_{\hyp,u_1}(0)(T_j)\right|+\left|\wt\FF_{\hyp,u_1}
(w^*)(T_j)-\wt\FF_{\hyp,u_1}(0)(T_j)\right|\\
\leq&K_\sigma\wh C^{(j)}\de \left(1+\OO_\sigma\left(\left(\wh
C^{(j)}\de\right)^{1/2}\ln^2(1/\de)\right)\right).
\end{split}
\]
Then, undoing the change of coordinates \eqref{def:HypVariationConstants} and using the definition of  $T_j$ in \eqref{def:FixedTimeSaddle}, one obtains
\[
|x_1^f|\leq K_\sigma\left(\wh C^{(j)}\de\right)^{1/2}.
\]
Finally, proceeding analogously, and taking into account
\eqref{def:ChoiceEtaHat} again, one can see that
\[
 y_2^f=-\frac{f_1(\sigma)}{f_2(\sigma)}\wh
C^{(j)}\de\ln(1/\de)\left(1+\OO_\sigma\left(\frac{1}{\ln(1/\de)}\right)\right)
\]
which completes the proof of Proposition \ref{lemma:HypToyModel:SaddleMap}.
\end{proof}

Now, it only remains to
%M
prove Proposition \ref{lemma:Hyp:Contractive}.
\begin{proof}[Proof of Proposition \ref{lemma:Hyp:Contractive}]
To compute the Lipschitz constant we need first upper bounds for $w\in
B(2\kk_0)\subset\YY_\hyp$ in the classical supremmum norm $\|\cdot\|_\infty$. They
can be deduced from the definition of the norms $\|\cdot\|_{\hyp, \ast}$ in
\eqref{def:Hyp:Norms} and the fact that $z^0\in\wt\NNN^{(j)}$ (see
\eqref{def:Hyp:ModifiedDomain}). Then, we have that
\begin{equation}\label{eq:Lip:SupBounds}
 \begin{split}
 |u_1|&\leq K_\sigma\wh C^{(j)}\de\ln(1/\de)\\
 |v_1|&\leq K\sigma\\
|u_2|&\leq K_\sigma \left(\wh C^{(j)}\de\right)^{1/2}\\
|v_2|&\leq K_\sigma\left(\wh C^{(j)}\de\right)^{1/2}\ln(1/\de).
\end{split}
\end{equation}
where $K>0$ is a constant independent of $\sigma$.

We use these bounds to obtain the Lipschitz constant. We start by computing the
Lipschitz constant of $\wt \FF_{\hyp,v_1}=\FF_{\hyp,v_1}$ and $\wt
\FF_{\hyp,u_2}=\FF_{\hyp,u_2}$ and then we will compute the other two.

Using the properties of $R_{\hyp, y_1}$ given in Lemma \ref{lemma:ToyModel:NormalForm}, \eqref{def:BoundSaddleTime} and the just obtained bounds, one can easily see that
\[
\begin{split}
\left|\FF_{\hyp,v_1}(u,v)-\FF_{\hyp,v_1}(u',v')\right|\leq
&\int_0^{T_j}\OO\left(uv\right)\sum_{i=1,2}|v_i-v_i'|dt+
\int_0^{T_j}\OO\left(v^2\right)\sum_{i=1,2}|u_i-u_i'|dt\\
\leq &K_\sigma\left(\wh
C^{(j)}\de\right)^{1/2}\ln(1/\de)\sum_{i=1,2}\|v_i-v_i'\|_\infty \\
&+ K_\sigma\ln(1/\de)\sum_{i=1,2}\|u_i-u_i'\|_\infty\\
\leq &K_\sigma\left(\wh
C^{(j)}\de\right)^{1/2}\ln(1/\de)\sum_{i=1,2}\|v_i-v_i'\|_{\hyp,v_i} \\
&+ K_\sigma\left(\wh
C^{(j)}\de\right)^{1/2}\ln(1/\de)\sum_{i=1,2}\|u_i-u_i'\|_{\hyp,u_i}.
\end{split}
\]
Note that we are abusing notation since inside the $\OO(\cdot)$ the dependence
of the size on $(u,v)$ means both dependence on $(u,v)$ and $(u',v')$. We do not
write the full dependence since both terms have the same size. Applying
the norms defined in \eqref{def:Hyp:Norms}, we get
\[
 \left\|\FF_{\hyp,v_1}(u,v)-\FF_{\hyp,v_1}(u',v')\right\|_{\hyp, v_1}\leq K_\sigma
\left(\wh C^{(j)}\de\right)^{1/2}\ln(1/\de)\|(u,v)-(u',v')\|_\ast.
\]
Now we bound the Lipschitz constant of $\FF_{\hyp,u_2}$. Proceeding as in the
previous case one obtains
\[
\begin{split}
\left|\FF_{\hyp,u_2}(u,v)-\FF_{\hyp,u_2}(u',v')\right|\leq
&\int_0^{T_j}\OO\left(uv\right)\sum_{i=1,2}|u_i-u_i'|dt+
\int_0^{T_j}\OO\left(u^2\right)\sum_{i=1,2}|v_i-v_i'|dt\\
\leq &K_\sigma\left(\wh
C^{(j)}\de\right)^{1/2}\ln(1/\de)\sum_{i=1,2}\|u_i-u_i'\|_\infty\\
&+ K_\sigma\wh C^{(j)}\de\ln(1/\de)\sum_{i=1,2}\|v_i-v_i'\|_\infty.\\
\leq &K_\sigma\wh C^{(j)}\de\ln(1/\de)\sum_{i=1,2}\|u_i-u_i'\|_{\hyp,u_i}\\
&+ K_\sigma\wh C^{(j)}\de\ln(1/\de)\sum_{i=1,2}\|v_i-v_i'\|_{\hyp,v_i}
\end{split}
\]
and thus
\[
 \left\|\FF_{\hyp,u_2}(u,v)-\FF_{\hyp,u_2}(u',v')\right\|_{\hyp, u_2}\leq
K_\sigma\left(\wh C^{(j)}\de\right)^{1/2}\ln(1/\de) \|(u,v)-(u',v')\|_\ast.
\]
To bound the Lipschitz constant of $\wt \FF_{\hyp, u_1}$ we use its definition
in \eqref{def:Hyp:NewOperator}. First we study $\FF_{\hyp, u_1}(w)- \FF_{\hyp,
u_1}(w')$. We proceed as for $\FF_{\hyp, u_2}$ but we have to be more accurate.
We obtain
\[
\begin{split}
\left|\FF_{\hyp,u_1}(u,v)-\FF_{\hyp,u_1}(u',v')\right|\leq
&\int_0^{T_j}\OO\left(uv\right)\sum_{i=1,2}|u_i-u_i'|dt+
\int_0^{T_j}\OO\left(u^2\right)\sum_{i=1,2}|v_i-v_i'|dt\\
\leq &K_\sigma\left(\wh
C^{(j)}\de\right)^{1/2}\ln(1/\de)\sum_{i=1,2}\|u_i-u_i'\|_\infty\\
&+ K_\sigma\wh C^{(j)}\de\ln(1/\de)\sum_{i=1,2}\|v_i-v_i'\|_\infty\\
\leq &K_\sigma\left( \wh C^{(j)}\de\right)^{1/2}\wh
C^{(j)}\de\ln^2(1/\de)\|u_1-u_1\|_{\hyp,u_1}\\
&+K_\sigma\wh C^{(j)}\de\ln(1/\de)\|u_2-u_2'\|_{\hyp,u_2}\\
&+ K_\sigma\wh C^{(j)}\de\ln(1/\de)\|v_1-v_1'\|_{\hyp,v_1}\\
&+ K_\sigma\left(\wh C^{(j)}\de\right)^{1/2}\wh
C^{(j)}\de\ln^2(1/\de)\|v_2-v_2'\|_{\hyp,v_2}.
\end{split}
\]
Thus, taking into account that for $\de$ small enough,
\[
 \sup_{t\in[0, T_j(x_2^0)]}\left|\frac{1}{-\wh
C^{(j)}\de\ln(1/\de)+2\nu_{02}f_1(\sigma)\left(x_2^\ast\right)^2 t+\wh
C^{(j)}\de}\right|\leq\frac{2}{\wh C^{(j)}\de},
\]
one can deduce that
\[
\begin{split}
 \left\|\FF_{\hyp,u_1}(u,v)-\FF_{\hyp,u_1}(u',v')\right\|_{\hyp,
u_1}\leq&K_\sigma\left(\wh
C^{(j)}\de\right)^{1/2}\ln^2(1/\de)\|u_1-u_1\|_{\hyp,u_1}
\\&+K_\sigma\ln(1/\de)\|u_2-u_2'\|_{\hyp,u_2}\\
&+ K_\sigma\ln(1/\de)\|v_1-v_1'\|_{\hyp,v_1}\\&+ K_\sigma\left(\wh
C^{(j)}\de\right)^{1/2}\ln^2(1/\de)\|v_2-v_2'\|_{\hyp,v_2}.
\end{split}
\]
Therefore, to obtain the Lipschitz constant for $\wt \FF_{\hyp,u_1}$, it only
remains to use its definition in \eqref{def:Hyp:NewOperator} and the Lipschitz
constants already obtained for $\FF_{\hyp,v_1}$ and $\FF_{\hyp,u_2}$ to obtain
\[
 \left\|\wt\FF_{\hyp,u_1}(u,v)-\wt\FF_{\hyp,u_1}(u',v')\right\|_{\hyp, u_1}\leq
K_\sigma\left(\wh C^{(j)}\de\right)^{1/2}\ln^2(1/\de) \|(u,v)-(u',v')\|_\ast.
\]
Proceeding analogously, one can see also that
\[
 \left\|\wt\FF_{\hyp,v_2}(u,v)-\wt\FF_{\hyp,v_2}(u',v')\right\|_{\hyp, v_2}\leq
K_\sigma\left(\wh C^{(j)}\de\right)^{1/2}\ln(1/\de) \|(u,v)-(u',v')\|_\ast.
\]
This
%M finishes is not usually used
completes the proof.
\end{proof}

\section{The local map: proof of Lemma \ref{lemma:iterative:saddle}}\label{sec:FullSystem}

%M Once we have dealt with
Analysis of Section \ref{sec:HypToyModel} describes
dynamics of the hyperbolic toy model \eqref{def:VF:HypToyModel}.
Now we add
%M again
the elliptic modes and
%M  therefore we
consider the whole  vector field \eqref{def:VF:Full:AfterDiagonal}.
%Then, we
Our goal is to study the map $\BB_\loc^j$. The key point of this study
is that  the elliptic modes  remain almost constant through
the saddle map and %M that they
do not make much influence on the hyperbolic ones.
%M
In other words, there is {\it an almost product structure.}
%M This fact allows us apply
This allows us to extend the results obtained for the hyperbolic toy
model  \eqref{def:VF:HypToyModel} in Section \ref{sec:HypToyModel}
to the general system.

As a first step we perform the change obtained  in Lemma
\ref{lemma:ToyModel:NormalForm} by means of a normal form procedure
for the hyperbolic toy model \eqref{def:VF:HypToyModel}. The proof
of this lemma is straightforward taking into account the form of
the vector field \eqref{def:VF:Full:AfterDiagonal} and the properties
of $\Psi_\hyp$ given in Lemma \ref{lemma:ToyModel:NormalForm}.

\begin{lemma}\label{lemma:FullModel:NormalForm}
%M If we apply
Let $\Psi_\hyp$ be the map defined in Lemma
\ref{lemma:ToyModel:NormalForm}. Then an application of
the change of coordinates
\begin{equation}\label{def:Full:NFChange}
 (p_1,q_1,p_2,q_2,c)=\left(\Psi_\hyp(x_1,y_1,x_2,y_2),c\right),
\end{equation}
to the vector field \eqref{def:VF:Full:AfterDiagonal}
leads to a vector field of the form
\[
\begin{split}
\dot z&= Dz+R_{\hyp}(z)+R_{\mix,z}(z,c)\\
\dot c_k&= ic_k+\ZZZ_{\el,c_k}(c)+R_{\mix,c}(z,c),
\end{split}
\]
%\begin{equation}\label{def:VF:Full:AfterNF}
%\begin{split}
%\dot z&= Dz+R_{\hyp}(z)+R_{\mix,z}(z,c)\\
%\dot c_k&= ic_k+\ZZZ_{\el,c_k}(c)+R_{\mix,c}(z,c),
%\end{split}
%\end{equation}
where $z$ denotes $z=(x_1,y_1,x_2,y_2)$,
$D=\mathrm{diag}(\sqrt{3},-\sqrt{3},\sqrt{3},-\sqrt{3})$,  $R_{\hyp}$ has been
given in Lemma \ref{lemma:ToyModel:NormalForm}, $\ZZZ_{\el,c_k}$ is defined in
\eqref{def:VF:Diag:cm:El},
and $R_{\mix,z}$ and $R_{\mix,c_k}$ are defined as
\begin{align}
R_{\mix,x_1}&=A_{x_1}(z)\ol{c_{j-2}}^2+\ol{A_{x_1}(z)}{c_{j-2}}^2+\frac{\sqrt{3}
}{2}\sum_{k\in\PP}|c_k|^2\Psi_{x_1}(z)\notag\\
R_{\mix,y_1}&=A_{y_1}(z)\ol{c_{j-2}}^2+\ol{A_{y_1}(z)}{c_{j-2}}^2+\frac{\sqrt{3}
}{2}\sum_{k\in\PP}|c_k|^2\Psi_{y_1}(z)\notag\\
R_{\mix,x_2}&=A_{x_2}(z)\ol{c_{j+2}}^2+\ol{A_{x_2}(z)}{c_{j+2}}^2+\frac{\sqrt{3}
}{2}\sum_{k\in\PP}|c_k|^2\Psi_{x_2}(z)\notag\\
R_{\mix,y_2}&=A_{y_2}(z)\ol{c_{j+2}}^2+\ol{A_{y_2}(z)}{c_{j+2}}^2+\frac{\sqrt{3}
}{2}\sum_{k\in\PP}|c_k|^2\Psi_{y_2}(z)\notag\\
 R_{\mix,c_k}&=i\sqrt{3}c_k P(z)\,\,\,\text{ for }m\neq j\pm
2\notag\\
R_{\mix,c_{j\pm 2}}&=i\sqrt{3}c_{j\pm 2} P(z)-i \ol c_{j\pm
2}Q_\pm(z)\notag
\end{align}
%\begin{align}
%R_{\mix,x_1}&=A_{x_1}(z)\ol{c_{j-2}}^2+\ol{A_{x_1}(z)}{c_{j-2}}^2+\frac{\sqrt{3}
%}{2}\sum_{k\in\PP}|c_k|^2\Psi_{x_1}(z)\label{def:VF:Full:Hyp:Mix:x1}\\
%R_{\mix,y_1}&=A_{y_1}(z)\ol{c_{j-2}}^2+\ol{A_{y_1}(z)}{c_{j-2}}^2+\frac{\sqrt{3}
%}{2}\sum_{k\in\PP}|c_k|^2\Psi_{y_1}(z)\label{def:VF:Full:Hyp:Mix:y1}\\
%R_{\mix,x_2}&=A_{x_2}(z)\ol{c_{j+2}}^2+\ol{A_{x_2}(z)}{c_{j+2}}^2+\frac{\sqrt{3}
%}{2}\sum_{k\in\PP}|c_k|^2\Psi_{x_2}(z)\label{def:VF:Full:Hyp:Mix:x2}\\
%R_{\mix,y_2}&=A_{y_2}(z)\ol{c_{j+2}}^2+\ol{A_{y_2}(z)}{c_{j+2}}^2+\frac{\sqrt{3}
%}{2}\sum_{k\in\PP}|c_k|^2\Psi_{y_2}(z)\label{def:VF:Full:Hyp:Mix:y2}\\
% R_{\mix,c_k}&=i\sqrt{3}c_k P(z)\,\,\,\text{ for }m\neq j\pm
%2\label{def:VF:Full:Ell:Mix:AfterNF}\\
%R_{\mix,c_{j\pm 2}}&=i\sqrt{3}c_{j\pm 2} P(z)-i \ol c_{j\pm
%2}Q_\pm(z)\label{def:VF:Full:Ell:Mix:Adjacent:AfterNF}
%\end{align}
where $\Psi_{\hyp,z}$ are the functions defined in  Lemma
\ref{lemma:ToyModel:NormalForm},  $A_{z}$ satisfy
\[
 A_{x_i}=\OO(x_i, y_i)\,\,\,\text{ and }\,\,\,A_{y_i}=\OO(x_i, y_i)
\]
and $P$ and $Q_\pm$ satisfy
\[
P(z)=\OO\left(x_1y_1,x_2y_2,z_1^2z_2^2\right),\,\,\,Q_-(z)=\OO\left(x_1,
y_1\right)  \,\,\,\text{ and }\,\,\,Q_+(z)=\OO\left(x_2,y_2\right).
\]
\end{lemma}

One can easily see that for this system there is a rather strong interaction between the hyperbolic and the elliptic modes due to the terms $
R_{\mix,x_i}$ and $ R_{\mix,y_i}$. The importance of these terms can be seen as follows. The manifold $\{x=0,y=0\}$ is normally hyperbolic \cite{Fenichel74, Fenichel77, HirschPS77}  for the linear truncation of the vector field obtained in Lemma \ref{lemma:FullModel:NormalForm} and its stable and unstable manifolds  are defined as $\{x=0\}$ and $\{y=0\}$. For the full vector field, the manifold $\{x=0,y=0\}$ is persistent. Moreover it is still normally hyperbolic thanks to \cite{Fenichel74, Fenichel77, HirschPS77}. Nevertheless, the associated invariant manifolds deviate from $\{x=0\}$ and $\{y=0\}$ due to the terms $R_{\mix,x_i}$ and $ R_{\mix,y_i}$. To overcome this problem, we slightly modify  the change \eqref{def:Full:NFChange} to straighten these invariant manifolds completely.

\begin{lemma}\label{lemma:StraightenInvManifolds}
There exist a change of coordinates of the form
\begin{equation}\label{def:Change:FullSystem:Hyp}
(p_1,q_1,p_2,q_2,c)=\left(\Psi(x_1,y_1,x_2,y_2,c),c\right)=\left(x_1,y_1,x_2,y_2
,c\right)+\left(\wt\Psi (x_1,y_1,x_2,y_2,c),0\right)
\end{equation}
which transforms the vector field \eqref{def:VF:Full:AfterDiagonal} into a
vector field of the form
\begin{equation}\label{def:VF:Straight}
\begin{split}
\dot z&= Dz+R_{\hyp}(z)+\wt R_{\mix,z}(z,c)\\
\dot c_k&= ic_k+\ZZZ_{\el,c_k}(c)+ \wt R_{\mix,c_k}(z,c),
\end{split}
\end{equation}
where  $R_\hyp$ and $\ZZZ_\el$ are the functions defined in
\eqref{def:HypHam:Hyp} and \eqref{def:VF:Diag:cm:El} respectively, and
\begin{align}
\wt
R_{\mix,x_1}&=B_{x_1}(z,c)\ol{c_{j-2}}^2+\ol{B_{x_1}(z,c)}{c_{j-2}}^2+\frac{
\sqrt{3}}{2}\sum_{k\in\PP}|c_k|^2C_{x_1}(z,c)\notag
\\
\wt
R_{\mix,y_1}&=B_{y_1}(z,c)\ol{c_{j-2}}^2+\ol{B_{y_1}(z,c)}{c_{j-2}}^2+\frac{
\sqrt{3}}{2}\sum_{k\in\PP}|c_k|^2C_{y_1}(z,c)\notag
\\
\wt
R_{\mix,x_2}&=B_{x_2}(z,c)\ol{c_{j+2}}^2+\ol{B_{x_2}(z,c)}{c_{j+2}}^2+\frac{
\sqrt{3}}{2}\sum_{k\in\PP}|c_k|^2C_{x_2}(z,c)\notag
\\
\wt
R_{\mix,y_2}&=B_{y_2}(z,c)\ol{c_{j+2}}^2+\ol{B_{y_2}(z,c)}{c_{j+2}}^2+\frac{
\sqrt{3}}{2}\sum_{k\in\PP}|c_k|^2C_{y_2}(z,c)\notag
\\
\wt R_{\mix,c_k}&=i\sqrt{3}c_k \wt P(z,c)\,\,\,\text{ for }k\neq j\pm
2\notag\\
\wt R_{\mix,c_{j\pm 2}}&=i\sqrt{3}c_{j\pm 2} \wt P(z,c)-i \ol c_{j\pm 2}\wt
Q_\pm(z,c),\notag
\end{align}
%\begin{align}
%\wt
%R_{\mix,x_1}&=B_{x_1}(z,c)\ol{c_{j-2}}^2+\ol{B_{x_1}(z,c)}{c_{j-2}}^2+\frac{
%\sqrt{3}}{2}\sum_{k\in\PP}|c_k|^2C_{x_1}(z,c)\label{def:VF:Full:Hyp:Mix:Str:x1}
%\\
%\wt
%R_{\mix,y_1}&=B_{y_1}(z,c)\ol{c_{j-2}}^2+\ol{B_{y_1}(z,c)}{c_{j-2}}^2+\frac{
%\sqrt{3}}{2}\sum_{k\in\PP}|c_k|^2C_{y_1}(z,c)\label{def:VF:Full:Hyp:Mix:Str:y1}
%\\
%\wt
%R_{\mix,x_2}&=B_{x_2}(z,c)\ol{c_{j+2}}^2+\ol{B_{x_2}(z,c)}{c_{j+2}}^2+\frac{
%\sqrt{3}}{2}\sum_{k\in\PP}|c_k|^2C_{x_2}(z,c)\label{def:VF:Full:Hyp:Mix:Str:x2}
%\\
%\wt
%R_{\mix,y_2}&=B_{y_2}(z,c)\ol{c_{j+2}}^2+\ol{B_{y_2}(z,c)}{c_{j+2}}^2+\frac{
%\sqrt{3}}{2}\sum_{k\in\PP}|c_k|^2C_{y_2}(z,c)\label{def:VF:Full:Hyp:Mix:Str:y2}
%\\
%\wt R_{\mix,c_k}&=i\sqrt{3}c_k \wt P(z,c)\,\,\,\text{ for }k\neq j\pm
%2\label{def:VF:Full:Ell:Mix:Straight}\\
%\wt R_{\mix,c_{j\pm 2}}&=i\sqrt{3}c_{j\pm 2} \wt P(z,c)-i \ol c_{j\pm 2}\wt
%Q_\pm(z,c)\label{def:VF:Full:Ell:Mix:Adjacent:Straight}
%\end{align}
where the functions $B_z$ and $C_z$ satisfy
\begin{align*}
B_{x_1}(z,c)=\OO\left(x_1+y_1x_2z_2\right)&&B_{x_2}(z,
c)=\OO\left(x_2+y_2x_1z_1\right)\\
B_{y_1}(z,c)=\OO\left(y_1+x_1y_2z_2\right)&&B_{y_2}(z,
c)=\OO\left(y_2+x_2y_1z_1\right)\\
C_{x_1}(z,c)=\OO\left(x_1+y_1x_2z_2\right)&&C_{x_2}(z,
c)=\OO\left(x_2+y_2x_1z_1\right)\\
C_{y_1}(z,c)=\OO\left(y_1+x_1y_2z_2\right)&&C_{y_2}(z,
c)=\OO\left(y_2+x_2y_1z_1\right)
\end{align*}
and $\wt P$ and $\wt Q_\pm$ satisfy
\[
\wt P(z,c)=\OO\left(x_1y_1,x_2y_2,z_1^2z_2^2\right),\,\,\,\wt
Q_-(z,c)=\OO\left(x_1,y_1\right)  \,\,\,\text{ and }\,\,\,\wt
Q_+(z)=\OO\left(x_2,y_2\right).
\]
Moreover, the function $\wt\Psi$ satisfies
\[
\begin{split}
\wt\Psi_{x_1}&=\OO\left(x_1^3,x_1y_1,x_1(x_2^2+y_2^2),
y_1y_2(x_2+y_2),c_{j-2}^2y_1,\sum_{k\in\PP}|c_k|^2y_1y_2^2\right)\\
\wt\Psi_{y_1}&=\OO\left(y_1^3,x_1y_1,y_1(x_2^2+y_2^2),
x_1x_2(x_2+y_2),c_{j-2}^2x_1,\sum_{k\in\PP}|c_k|^2x_1x_2^2\right)\\
\wt\Psi_{x_2}&=\OO\left(x_2^3,x_2y_2,x_2(x_1^2+y_1^2),
y_1y_2(x_1+y_1),c_{j+2}^2y_1,\sum_{k\in\PP}|c_k|^2y_2y_1^2\right)\\
\wt\Psi_{y_2}&=\OO\left(y_2^3,x_2y_2,y_2(x_1^2+y_1^2),
x_1x_2(x_1+y_1),c_{j+2}^2x_1,\sum_{k\in\PP}|c_k|^2x_2x_1^2\right).
\end{split}
\]
%\begin{equation} \label{def:ChangeStraight:Bounds}
%\end{equation}
\end{lemma}
\begin{proof}
It is enough to compose two change of coordinates. The first change is the
change \eqref{def:Change:FullSystem:Hyp} considered in Lemma
\ref{lemma:FullModel:NormalForm}. The second one is the one which straightens
the invariant manifolds of a normally hyperbolic invariant manifold
\cite{Fenichel74, Fenichel77, HirschPS77}. Then, to obtain
%M all
the required estimates,
%M is enough to consider the ones obtained in
it suffices to combine Lemmas \ref{lemma:ToyModel:NormalForm}
and \ref{lemma:FullModel:NormalForm}
%jointly
with
%the ones given by
the
%classical
standard results about normally hyperbolic invariant  manifolds.
\end{proof}

After performing this change of coordinates,  the stable and unstable invariant
manifolds of $\{x=0,y=0\}$ are straightened. This will facilitate the study of the
transition map close to the saddle.

As we have done  in Section
\ref{sec:HypToyModel}, we define a set $\wh \VV_j$ such that
\begin{equation}\label{Cond:SaddleMap:Inclusion}
 \Upsilon\left(\VV_j\right)\subset \wh \VV_j,
\end{equation}
where  $\VV_j$ is the set defined in Lemma \ref{lemma:iterative:saddle} and
$\Upsilon$ is the inverse of the %M
coordinate change $\Psi$ given in Lemma
\ref{lemma:StraightenInvManifolds}. Then, we will apply the flow $\wh \Phi^t$
associated to the vector field \eqref{def:VF:Straight} to points in $\wh\VV_j$.
To obtain the inclusion \eqref{Cond:SaddleMap:Inclusion} we define the function
$g_{\II_j}(p_2,q_2,\sigma,\de)$ involved in the definition of $\VV_j$.

%M it is not good to put definitions inside statements

Define the set
\begin{equation*}
 \wh\VV_j=\DD_1^{1}\times\ldots\times \DD_j^{j-2}\times\wh\NNN_j\times
\DD_j^{j+2}\times\ldots\times \DD_j^{N},
\end{equation*}
where  $\wh\NNN_j$ is the set defined in \eqref{def:Hyp:ModifiedDomain} and
$\DD_j^{k}$ are defined as
\[
 \begin{split}
  \DD_j^{k}&=\left\{\left|c_{k}\right|\leq
M_{\el,\pm}\de^{(1-r)/2}\right\}\,\,\,\text{ for }k\in\PP_j^\pm\\
 \DD_j^{j\pm 2}&=\left\{\left|c_{j\pm 2}\right|\leq M_{\adj,\pm}\left(\wh
C^{(j)}\de\right)^{1/2}\right\}.
 \end{split}
\]
Define the function
%M
$g_{\II_j}(p_2,q_2,\sigma,\de)$ involved in the definition of
the set $\VV_j$ as
\begin{equation}\label{def:Function_g}
 g_{\II_j}(p_2,q_2,\sigma,\delta)=p_2+a_p(\sigma) p_2+a_q(\sigma)q_2-x^\ast_2
\end{equation}
where $x_2^\ast$ is the constant defined in \eqref{def:ChoiceEtaHat} and
\begin{equation*}
\begin{split}
a_p(\sigma)&=\pa_{p_2}\wt \Upsilon_{p_2}(0,\sigma,0,0,0)\\
a_q(\sigma)&= \pa_{q_2}\wt \Upsilon_{p_2}(0,\sigma,0,0,0),
\end{split}
\end{equation*}
where  $\Upsilon=\mathrm{Id}+\wt\Upsilon$ is the inverse of
the change $\Psi$ given in Lemma \ref{lemma:StraightenInvManifolds}.
%M
\begin{lemma}\label{lemma:FullModel:TiltedSection} With the above
notations
%M Then, taking
for $\de$ small enough condition \eqref{Cond:SaddleMap:Inclusion} is satisfied.
\end{lemma}
\begin{proof}
 It is a straightforward consequence of Lemmas \ref{lemma:ToyModel:NormalForm}
and \ref{lemma:StraightenInvManifolds}.
\end{proof}

After straightening the invariant manifold, next lemma studies the saddle map in
the transformed variables for points belonging to $\VV_j$.

\begin{lemma}\label{lemma:FullModel:SaddleMap}
Let us consider the flow $\wh\Phi_t$ associated to \eqref{def:VF:Straight} and a
point $(z^0,c^0)\in\wh\VV_j$. Then
%M if  we take $\de$ and $\sigma$ small enough,
for $\de$ and $\sigma$ small enough, the point
\[
 \left(z^f,c^f\right)=\wh \Phi_{T_j}\left(z^0,c^0\right),
\]
where $T_j=T_j(x_2^0)$ is the time defined in \eqref{def:FixedTimeSaddle},
satisfies
\[
\begin{split}
 |x_1^f| \qquad  \qquad \leq &\ K_\sigma\left(\wh C^{(j)}\de\right)^{1/2}\\
  |y_1^f| \qquad \qquad \leq &\ K_\sigma\left(\wh C^{(j)}\de\right)^{1/2}\\
 |x_2^f-f_2(\sigma)| \qquad \leq &\ K_\sigma\de^{r'}\\
\left| y_2^f + \frac{f_1(\sigma)}{f_2(\sigma)}
\wh C^{(j)} \de \ln(1/\de) \right|
\leq & \  \frac{f_1(\sigma)}{f_2(\sigma)}\de.
\end{split}
\]
and
\[
\begin{split}
\left| c_k^f-c_k^0 e^{iT_j}\right|\leq & K_\sigma\de^{(1-r)/2+r'}\,\,\,\text{
for }k\in\PP_j^\pm\\
\left| c_{j\pm 2}^f-c_{j\pm 2}^0 e^{iT_j}\right|\leq & 2M_{\adj,\pm} \sigma
\left(\wh C^{(j)}\de\right)^{1/2}.
\end{split}
\]
\end{lemma}
We postpone the proof of this lemma to Section \ref{sec:ProofLemmaFixedTimeMap}.

Now, to
%M finish
complete the proof of Lemma \ref{lemma:iterative:saddle}
%, it only remain
we need two steps.

The first is to undo the change of coordinates
performed in Lemma \ref{lemma:StraightenInvManifolds} to express
the estimates of the saddle map in the original variables.

The second step is to adjust the time so that the image belongs to
the section $\Sigma_{j}^\out$. These two final steps are done in
the next two following lemmas.

Concerning the first step, recall that the change of variables $\Psi$
defined in Lemma \ref{lemma:StraightenInvManifolds} does not change the elliptic
variables, and therefore it only affects the hyperbolic ones.

\begin{lemma}\label{lemma:FullModel:SaddleMap:Original:FixedTime}
Let us consider the flow $\Phi_t$ associated to
\eqref{def:VF:Full:AfterDiagonal} and a point $(p^0,q^0,c^0)\in\wh\VV_j$.
Then
%M , if
for $\de$ and $\sigma$ small enough, the point
\[
 \left(p^f,q^f,c^f\right)=\Phi_{T_j}\left(p^0,q^0,c^0\right),
\]
where $T_j$ is the time defined in \eqref{def:FixedTimeSaddle},  satisfies
\[
\begin{split}
 |p_1^f|\qquad \qquad\leq &\ K_\sigma\left(\wh C^{(j)}\de\right)^{1/2}\\
  |q_1^f| \qquad \qquad\leq&\ K_\sigma\left(\wh C^{(j)}\de\right)^{1/2}\\
|p_2^f- \sigma| \qquad\leq &\ K_\sigma\de^{r'}\\
| q_2^f+ \wt
C^{(j)}\de \ln(1/\de)| \leq & \ \wt C^{(j)}\ \de\ K_\sigma.
\end{split}
\]
for certain constant $\wt C^{(j)}$ satisfying $C^{(j)}/2\leq \wt C^{(j)}\leq
2C^{(j)}$  and
\[
\begin{split}
\left| c_k^f-c_k^0 e^{iT_j}\right|\leq & K_\sigma\de^{(1-r)/2+r'}\,\,\,\text{ for
}m\in\PP^{\pm}\\
\left| c_{j\pm 2}^f-c_{j\pm 2}^0 e^{iT_j}\right|\leq & 2M_{\adj,\pm} \sigma
\left(\wh C^{(j)}\de\right)^{1/2}.
\end{split}
\]
\end{lemma}
\begin{proof}
In Lemma \ref{lemma:StraightenInvManifolds} we have defined the change  $\Psi$ which relates the two sets of coordinates by
\[
 \left(p_1^f, q_1^f,p_2^f,q_2^f, c^f\right)=\left(\Psi\left(x_1^f, y_1^f,x_2^f,y_2^f, c^f\right), c^f\right).
\]
Then, taking  into account the properties of the change $\Psi$ stated in this
lemma, one can easily see that from the estimates obtained in Lemma
\ref{lemma:FullModel:SaddleMap}, one can deduce the estimates stated in Lemma
\ref{lemma:FullModel:SaddleMap:Original:FixedTime}. First recall that the change
$\Psi$ does not modify the elliptic modes and therefore we only need to deal
with the hyperbolic ones.

Using the properties of $\Psi$ and modifying slightly $K_\sigma$, it is easy to
see that for $\de$ small enough,
\[
\begin{split}
 |p_1^f|\leq &K_\sigma\left(\wh C^{(j)}\de\right)^{1/2}\\
  |q_1^f|\leq&K_\sigma\left(\wh C^{(j)}\de\right)^{1/2}.
\end{split}
\]
To obtain the estimates for $p_2$ it is enough to recall the definition of
$f_2(\sigma)$ in \eqref{def:HypToyModel:f2}. For the estimates for $q_2$, it is
enough to see that from the properties of $\Psi$ and the estimates for $z^f$ one
can deduce that
\[
 q_2=\pa_{x_2}\Psi_{x_2}(0,0,\sigma,0)x_2+\OO_\sigma\left(\wh C^{(j)}\de\right).
\]
Therefore, we can define a constant $\wt C^{j}$ such that the estimate for $q_2$
is satisfied.
\end{proof}

Once we have obtained good estimates for the  approximate time map in the
original variables, we adjust it to obtain image points belonging to the section
$\Sigma_j^\out$.
\begin{lemma}\label{lemma:Saddle:AdjustingSection}
 Let us consider a point $ \left(p^f,q^f,c^f\right)\in \Phi^{T_j}(\VV_j)$,
 where $\Phi^t$ is the flow of \eqref{def:VF:Full:AfterDiagonal}, $T_j$
 is the time defined in \eqref{def:FixedTimeSaddle} and $\VV_j$ is
 the set considered in Theorem \ref{theorem:iterative}.

Then, there exists a time $T'$, which depends on the point
$\left(p^f,q^f,c^f\right)$, such that
\[
\left(p^\ast,q^\ast,c^\ast\right)=
\Phi^{T'}\left(p^f,q^f,c^f\right)\in\Sigma_{j}^\out.
\]
Moreover, there exists a constant $K_\sigma$ such that
\begin{equation}\label{def:AdjustingTime}
|T'|\leq K_\sigma\de^r
\end{equation}
and
\[
\begin{split}
\left|c_k^\ast-c_k^f\right|&\leq K_\sigma \de^{1-r}\,\,\,\text{ for }m\in\PP\\
\left|p_1^\ast-p_1^f\right|&\leq K_\sigma
\left(C^{(j)}\de\right)^{1/2}\de^{1-r}\\
\left|q_1^\ast-q_1^f\right|&\leq K_\sigma
\left(C^{(j)}\de\right)^{1/2}\de^{1-r}\\
p_2&=\sigma\\
%M below I change q_1 to q_2
\left|q_2^\ast-q_2^f\right|&\leq K_\sigma C^{(j)}\de^{2-r}\ln(1/\de).
\end{split}
\]
\end{lemma}

\begin{proof}
 The proof of this Lemma follows the same lines as the proof of Proposition
\ref{prop:HeteroMap}.
%M That is,
Namely, first we obtain a priori bounds for each
variable, which then allow us to obtain more refined estimates.
\end{proof}

To finish the proof of Lemma \ref{lemma:iterative:saddle}, we define
$\UU_j=\BB_\loc^j(\VV_j)$ and we check that this set
%M is
has a  $\wt \II_j$-product-like
%set
structure for a multiindex $\wt\II_j$   satisfying the properties
stated in Lemma \ref{lemma:iterative:saddle} (see Definition
\ref{definition:ModifiedProductLike}). Indeed, from the results obtained in
Lemmas \ref{lemma:FullModel:SaddleMap:Original:FixedTime} and
\ref{lemma:Saddle:AdjustingSection} and recalling that by the hypotheses of Lemma \ref{lemma:iterative:saddle} we have that $M_\hyp^{(j)}\geq 1$, it is easy to see that one can define a
constant $K_\sigma$ so that if we consider the constants $\wt
M_{\el,\pm}^{(j)}$, $\wt M_{\adj,\pm}^{(j)}$ and $\wt M_\hyp^{(j)}$  defined in Lemma \ref{lemma:iterative:saddle}  and the constant
$\wt C^{(j)}$ given in Lemma \ref{lemma:FullModel:SaddleMap:Original:FixedTime},
the set $\UU_j=\BB_\loc^j(\VV_j)$ satisfies condition \textbf{C1} stated in
Definition \ref{definition:ModifiedProductLike}.

Thus, it only remains to check  that the set $\UU_j$ also satisfies condition
\textbf{C2} of Definition \ref{definition:ModifiedProductLike}.  First we check
the part of the condition \textbf{C2} concerning the elliptic modes. Indeed,
from the estimates for the  non-neighbor and adjacent elliptic modes given in Lemma
\ref{lemma:FullModel:SaddleMap:Original:FixedTime} and
\ref{lemma:Saddle:AdjustingSection}, one can easily see that for any fixed
values for the hyperbolic modes, if one takes the constants
$\wt m_{\el}^{(j)}$, $\wt m_{\adj}^{(j)}$ given in Lemma
\ref{lemma:iterative:saddle}, the image of the elliptic modes contains  disks as
stated in  Definition \ref{definition:ModifiedProductLike}. Then, it only
remains to check that the inclusion condition is also satisfied for the variable
$q_2$. From the proof of Lemma \ref{lemma:FullModel:SaddleMap} given in Section
\ref{sec:ProofLemmaFixedTimeMap}, one can easily deduce that the image in the  $y_2$
variable contains an interval of length $\OO(\wh C^{(j)}\de)$ and whose points
are of size smaller than $2\wh C^{(j)}\de\ln(1/\de)$. Then, when we undo the
normal form change of coordinates (Lemma
\ref{lemma:FullModel:SaddleMap:Original:FixedTime}),
this interval is only modified slightly but keeping still  a length of order
$\OO(\wh C^{(j)}\de)$. Thus taking into account the constant $\wt C^{(j)}$ given
Lemma \ref{lemma:FullModel:SaddleMap:Original:FixedTime} and the results of
Lemma \ref{lemma:Saddle:AdjustingSection}, we can obtain a constant $\wt
m_\hyp^{(j)}$ so that condition \textbf{C2} is satisfied.

Finally, it only remains to obtain upper bounds for the time spent by the map
$\BB_\loc^j$. To this end it is enough to recall that the time spent is the sum
of the time $T_j$ defined in \eqref{def:FixedTimeSaddle}, which has been bounded
in \eqref{def:BoundSaddleTime},  and the time $T'$ given in Lemma
\ref{lemma:Saddle:AdjustingSection}, which has been bounded in
\eqref{def:AdjustingTime}. Thus, taking into accounts these two bounds we obtain
the bound for the time spent by $\BB_\loc^j$ given in Lemma
\ref{lemma:iterative:saddle}. This finishes the proof of Lemma
\ref{lemma:iterative:saddle}.

\subsection{Proof of Lemma
\ref{lemma:FullModel:SaddleMap}}\label{sec:ProofLemmaFixedTimeMap}
As we have done in the Section \ref{sec:HypToyModel}, we make variation of
constants to set up a fixed point argument. Namely, we consider
\[
 x_i=e^{\sqrt{3}t}u_i,\,\,y_i=e^{-\sqrt{3}t}v_i,\,\,c_k=e^{it}d_k
\]
and then we obtain the integral equation
\begin{equation}\label{def:Full:ConstantVariation}
\begin{split}
u_i&=x_i^0+\int_0^{T_j} e^{-\sqrt{3}t}\left( R_{\hyp,x_i}\left(u e^{\sqrt{3}t},v
e^{-\sqrt{3}t}\right)+\wt R_{\mix,x_i}\left(u e^{\sqrt{3}t},v
e^{-\sqrt{3}t},de^{it}\right)\right)dt\\
v_i&=y_i^0+\int_0^{T_j} e^{\sqrt{3}t}\left( R_{\hyp,y_i}\left(u e^{\sqrt{3}t},v
e^{-\sqrt{3}t}\right)+\wt R_{\mix,y_i}\left(u e^{\sqrt{3}t},v
e^{-\sqrt{3}t},de^{it}\right)\right)dt\\
d_k&=c_k^0+\int_0^{T_j} e^{-it}\left( \ZZZ_{\el,c_k}\left(de^{it}\right)+\wt
R_{\mix,c_k}\left(u e^{\sqrt{3}t},v e^{-\sqrt{3}t},de^{it}\right)\right)dt.
\end{split}
\end{equation}
Note that the terms $ R_{\hyp,z}$ are the ones considered in Section
\ref{sec:HypToyModel}, and, therefore, we will use the properties of these
functions obtained in that section. We use the same integration time $T_j$ in
\eqref{def:FixedTimeSaddle}.

As before, we use \eqref{def:Full:ConstantVariation} to set up a fixed point
argument in two steps. First we define $\GG=(\GG_\hyp,\GG_\el)$ as
\[
\begin{split}
 \GG_{\hyp,u_i}(u,v,d)&=x_i^0+\int_0^{T_j}
e^{-\sqrt{3}t}\left(R_{\hyp,x_i}\left(u e^{\sqrt{3}t},v
e^{-\sqrt{3}t}\right)+\wt R_{\mix,x_i}\left(u e^{\sqrt{3}t},v
e^{-\sqrt{3}t},de^{it}\right)\right)dt\\
&=\FF_{\hyp,u_i}(u,v)+\int_0^{T_j} e^{-\sqrt{3}t}\wt R_{\mix,x_i}\left(u
e^{\sqrt{3}t},v e^{-\sqrt{3}t},de^{it}\right)dt\\
 \GG_{\hyp,v_i}(u,v,d)&=
y_i^0-\int_0^{T_j}e^{\sqrt{3}t}\left(R_{\hyp,y_i}\left(u e^{\sqrt{3}t},v
e^{-\sqrt{3}t}\right)+\wt  R_{\mix,x_i}\left(u e^{\sqrt{3}t},v
e^{-\sqrt{3}t},de^{it}\right)\right)dt\\
&= \FF_{\hyp,v_i}(u,v)+\int_0^{T_j} e^{\sqrt{3}t}\wt R_{\mix,x_i}\left(u
e^{\sqrt{3}t},v e^{-\sqrt{3}t},de^{it}\right)dt,\\
\end{split}
\]
where  $\FF_\hyp$ is the operator defined in \eqref{def:Hyp:Operator1}, and
\[
 \GG_{\el,c_k}(u,v,d)=c_k^0+\int_0^{T_j} e^{-it}\left(
\ZZZ_{\el,c_k}\left(de^{it}\right)+\wt R_{\mix,c_k}\left(u e^{\sqrt{3}t},v
e^{-\sqrt{3}t},de^{it}\right)\right)dt.
\]
We modify this operator slightly as we have done for  $\FF_\hyp$ in Section
\ref{sec:HypToyModel} to make it contractive.  We define
\[
\begin{split}
\wt
\GG_{\hyp,u_1}(u,v,d)&=\GG_{\hyp,u_1}(u_1,\GG_{\hyp,v_1}(u,v,d),\GG_{\hyp,u_2}(u
,v,d),v_2,d)\\
\wt
\GG_{\hyp,v_2}(u,v,d)&=\GG_{\hyp,v_2}(u_1,\GG_{\hyp,v_1}(u,v,d),\GG_{\hyp,u_2}(u
,v,d),v_2,d).
\end{split}
\]
We denote the new operator by
\begin{equation}\label{def:FullSystem:OperatorModified}
 \wt \GG=\left(\wt \GG_{\hyp,u_1}, \GG_{\hyp,u_2},\GG_{\hyp,v_1},\wt
\GG_{\hyp,v_2},\GG_\el\right),
\end{equation}
whose fixed points coincide with those of $\GG$.

We extend the norm defined in \eqref{def:Hyp:Norms} to incorporate the elliptic
modes. To this end, we define
\[
\begin{split}
\|h\|_{\el,\pm}&=\left(M_{\el,\pm}\de^{(1-r)/2}\right)^{-1}\|h\|_\infty \\
\|h\|_{\adj,\pm}&=M_{\adj,\pm}\ii\left(\wh
C^{(j)}\de\right)^{-1/2}\|h\|_\infty
\end{split}
\]
and
\[
\|(u,v,d)\|_\ast=\sup_{\substack{k\in
\PP_j^\pm\\i=1,2}}\Big\{\|u_i\|_{\hyp,u_i},\|v_i\|_{\hyp,v_i},\|d_k\|_{\el,\pm},
\|d_{j\pm 2}\|_{\adj,\pm}\Big\}
\]
which, abusing notation, is denoted as the norm in \eqref{def:Hyp:FullNorm}. We
also define the Banach space
\[
\YY=\left\{(u,v,d):[0,T]\rightarrow \CC^{N-3}\times\RR^4;
\|(u,v,d)\|_\ast<\infty\right\}.
\]
Proceeding as in Section \ref{sec:HypToyModel}, we state the two following
propositions, from which one can easily deduce
the contractivity of $\wt \GG$. The proof of the first one is straightforward
taking into account the definition of $\wt\GG$ and Lemma
\ref{lemma:Hyp:FirstIteration} and the proof of the second one is deferred to
end of the section.

%M it is not good to use several lemmas to prove a lemma
% so I change lemmas to propositions
\begin{proposition}\label{lemma:Full:FirstIteration}
Let us consider the operator $\wt \GG$ defined in
\eqref{def:FullSystem:OperatorModified}. Then, the components of $\wt\GG(0)$ are
given by
\[
 \begin{split}
  \wt\GG_{\hyp,u_1}(0)&=\wt\FF_{\hyp,u_i}(0)\\
 \wt\GG_{\hyp,v_1}(0)&=y_1^0\\
 \wt\GG_{\hyp,u_2}(0)&=x_2^0\\
 \wt\GG_{\hyp,v_2}(0)&=\wt\FF_{\hyp,v_2}(0)\\
 \wt\GG_{\el,c_k}(0)&=c_k^0.
\end{split}
\]
Thus, there exists a constant $\kk_1>0$  independent of
$\sigma$, $\de$ and $j$ such that the operator $\wt \GG$ satisfies
\[
\left \|\wt \GG (0)\right\|_\ast\leq \kk_1.
\]
\end{proposition}

\begin{proposition}\label{lemma:Full:Contractive}
Let us consider $w_1, w_2\in B(2\kk_1)\subset\YY$, a constant $r'$
satisfying $0<r'<1/2-2r$ and  $\delta$ as defined in
Theorem \ref{thm:ToyModelOrbit}. Then taking  $\sigma$ small enough
and $N$ big enough such that $0<\de=e^{-\ga N}\ll 1$,   there exist
a constant $K_\sigma>0$ which is independent of $j$ and $N$, but
might depend on $\sigma$, and a constant $K$ independent of $j$, $N$
and $\sigma$, such that the operator $\wt \GG$ satisfies
\[
 \begin{split}
&\left\|\wt\GG_{\hyp,u_i}(u,v,d)-\wt\GG_{\hyp,u_i}(u',v',d')\right\|_{\hyp,u_i,
v_i}\leq \\
&\qquad\qquad\qquad\leq K_\sigma\de^{r'}\left\|(u,v,d)-(u',v',d')\right\|_\ast\\
&\left\|\wt\GG_{\hyp,v_i}(u,v,d)-\wt\GG_{\hyp,v_i}(u',v',d')\right\|_{\hyp,u_i,
v_i}\leq \\
&\qquad\qquad\qquad\leq K_\sigma\de^{r'}\left\|(u,v,d)-(u',v',d')\right\|_\ast\\
&\left\|\wt\GG_{\el,c_k}(u,v,d)-\wt\GG_{\el,c_k}(u',v',d')\right\|_{\el,\pm}\leq \\
&\qquad\qquad\qquad\leq K_\sigma\de^{r'}
\left\|(u,v,d)-(u',v',d')\right\|_\ast,\,\,\,\,\,\text{ for }m\in\PP^\pm\\
&\left\|\wt\GG_{\adj,\pm}(u,v,d)-\wt\GG_{\adj,
\pm}(u',v',d')\right\|_{\adj,\pm}\leq \\
&\qquad\qquad\qquad\leq  K\sigma\left\|(u,v,d)-(u',v',d')\right\|_\ast.
\end{split}
\]
Thus, since $0<\de\ll \sigma$,
\[
\left \|\wt \GG(w_2)-\wt \GG(w_1)\right\|_\ast\leq 2K\sigma\|w_2-w_1\|_\ast
\]
and therefore, for   $\sigma$ small enough, it is contractive.
\end{proposition}

The previous two
%M
propositions show that the operator $\wt\GG$ is contractive.
Let us denote by $(u^*,v^*,d^*)$ its unique fixed point in the ball
$B(2\kk_1)\subset\YY$. Now,  it only remains to obtain
the estimates stated in Lemma \ref{lemma:FullModel:SaddleMap}.
The estimates for the hyperbolic variables are obtained as in the proof of
Lemma \ref{lemma:HypToyModel:SaddleMap}. For the elliptic ones it is enough
to take into account that
\[
\begin{split}
 c_k^f&=c_k(T_j)=d_k(T_j) e^{iT_j}\\
&=\GG_{\el,c_k}(0)(T_j)e^{iT_j}+\left(\GG_{\el,c_k}(u^*,v^*,d^*)(T_j)-\GG_{\el,c_k}
(0)(T_j)\right)e^{iT_j}\\
&=c_k^0e^{iT_j}+\left(\GG_{\el,c_k}(u^*,v^*,d^*)(T_j)-\GG_{\el,c_k}(0)(T_j)\right)e^{iT_j}
\end{split}
\]
and bound the second term using the Lipschitz constant obtained in
% M
Proposition
\ref{lemma:Full:Contractive}.

We finish the section by proving
%M
Proposition \ref{lemma:Full:Contractive}, which
%M finishes
completes the proof of Lemma \ref{lemma:FullModel:SaddleMap}.

\begin{proof}[Proof of Proposition \ref{lemma:Full:Contractive}]
As we have done in the proof of Proposition \ref{lemma:Hyp:Contractive},
first, we stablish bounds for any $(u,v,d)\in B(2\kk_1)\subset\YY$ in
the supremmum norm, which will be used to bound the Lipschitz constant of
each component of $\wt\GG$. Indeed, if $(u,v,d)\in B(2\kk_1)\subset\YY$,
it satisfies \eqref{eq:Lip:SupBounds} and
\[
 \begin{split}
  |d_k|&\leq K_\sigma\de^{(1-r)/2}\,\,\,\,\text{ for }\,\,\,k\in\PP_j^\pm\\
  |d_{j\pm 2}|&\leq K_\sigma\left(\wh C^{(j)}\de\right)^{1/2}\leq K_\sigma\de^{(1-r)/2}.
 \end{split}
\]
We bound the Lipschitz constant for each component of $\wt\GG_\el$. We split
each component of the operator between the elliptic, hyperbolic and mixed part.
We deal first with the elliptic part. It can be seen that for $k\in\PP_j^{\pm}$,
\[
\begin{split}
\left|\ZZZ_{\el,c_k} \left(d'e^{it}\right)- \ZZZ_{\el,c_k} \left(de^{it}\right)
\right|\leq& K_\sigma\de^{1-r} N (d_k-d_k')\\
&+K_\sigma\de\sum_{\ell\in\PP_j\setminus\{ k\}}(d_\ell-d_\ell').
\end{split}
\]
Therefore,
\[
\left\|\int_0^{T_j} e^{-it} \left(\ZZZ_{\el,c_k}
\left(de^{it}\right)-\ZZZ_{\el,c_k}  \left(d'e^{it}\right)\right)
dt\right\|_{\el,\pm}\leq K_\sigma\de^{1-r} N T_j\|(u,v,d)-(u',v',d')\|_\ast.
\]
Proceeding analogously, one can see also that
\[
\left\|\int_0^{T_j} e^{-it} \left(\ZZZ_{\el,c_{j\pm 2}}
\left(de^{it}\right)-\ZZZ_{\el,c_{j\pm 2}}  \left(d'e^{it}\right)\right)
dt\right\|_{\adj,\pm}\leq K_\sigma\de^{1-r} N T_j\|(u,v,d)-(u',v',d')\|_\ast.
\]
Now we bound the mixed terms. Proceeding analogously and considering the
properties of $\wt R_{\mix,c_k}$ stated in Lemma
\ref{lemma:StraightenInvManifolds}, we can see that for $m\neq j\pm 2$,
\[
 \begin{split}
  \left\|\wt R_{\mix,c_k} \right.&\left.
\left(ue^{\sqrt{3}t},ve^{-\sqrt{3}t},de^{it}\right)-\wt R_{\mix,c_k}
\left(u'e^{\sqrt{3}t},v'e^{-\sqrt{3}t},d'e^{it}\right)\right\|_{\el,\pm}\\
\leq&K_\sigma \wh
C^{(j)}\de\ln^2(1/\de)\sum_{i=1,2}\left(\|u_i-u_i'\|_{\hyp,u_i}+\|v_i-v_i'\|_{
\hyp,v_i}\right)\\
&+K_\sigma\wh
C^{(j)}\de\ln^2(1/\de)\left(\left\|d_k-d_k'\right\|_{\el,\pm}+K_\sigma\de^{(1-r)/2}
\sum_{\ell\in\PP_j^\pm}\left\|d_\ell-d_\ell'\right\|_{\el,\pm}\right)\\
&+K_\sigma \wh
C^{(j)}\de^{1+(1-r)/2}\ln^2(1/\de)\left(\left\|d_{j-2}-d_{j-2}'\right\|_{\adj,-}
+\left\|d_{j+2}-d_{j+2}'\right\|_{\adj,+}\right)\\
&\leq K_\sigma \wh
C^{(j)}\de\ln^2(1/\de)\left(1+K_\sigma N\de^{(1-r)/2}\right)\left\|(u,v,d)-(u',v',
d')\right\|_\ast.
 \end{split}
\]
Therefore, using that $\de=e^{-\ga N}$ and \eqref{def:BoundSaddleTime},
\[
\begin{split}
\left\|\int_0^{T_j} e^{-it}\left(\wt R_{\mix,c_k}
(ue^{\sqrt{3}t},ve^{-\sqrt{3}t},de^{it})-\wt R_{\mix,c_k}
(u'e^{\sqrt{3}t},v'e^{-\sqrt{3}t},d'e^{it})\right)dt\right\|_{\el,\pm}\\
\quad\quad\quad\quad\quad\quad\quad\quad\quad\leq K_\sigma \wh
C^{(j)}\de\ln^3(1/\de)\|(u,v,d)-(u',v',d')\|_\ast.
\end{split}
\]
So, we can conclude that for $m\in\PP^\pm$,
\[
\left\|\GG_{\el,c_k}(u,v,d)-\GG_{\el,c_k}(u',v',d')\right\|_{\el,\pm}\leq
K_\sigma\de^{1-r}\ln^3(1/\de)\|(u,v,d)-(u',v',d')\|_\ast.
\]
Proceeding analogously we can bound the Lipschitz constant for $\GG_{\el,c_{j\pm
2}}$. We bound it for $m=j-2$, the other case can be done analogously. Here $K$ denotes a generic constant independent of $\sigma$.  Note that
now there is an additional term in $\wt R_{\mix,c_{j-2}}$. This implies that
\[
 \begin{split}
  \left|\wt R_{\mix,c_{j-2}}
(ue^{\sqrt{3}t}\right.&\left.,ve^{-\sqrt{3}t},de^{it})-\wt R_{\mix,c_{j-2}}
(u'e^{\sqrt{3}t},v'e^{-\sqrt{3}t},d'e^{it})\right|\\
&\leq K\sigma M_{\adj,-}\left(\wh C^{(j)}\de\right)^{1/2}
e^{-\sqrt{3}t}\sum_{i=1,2}\left(\|u_i-u_i'\|_{\hyp,u_i}+\|v_i-v_i'\|_{\hyp,v_i}
\right)\\
&+K\sigma M_{\adj,-}\left(\wh C^{(j)}\de\right)^{1/2}
e^{-\sqrt{3}t}\left\|d_{j-2}-d_{j-2}'\right\|_{\adj,-}\\
&+K_\sigma M_{\adj,-}\left(\wh C^{(j)}\de\right)^{1/2}\de^{(1-r)/2}
e^{-\sqrt{3}t}\left(\left\|d_{j+2}-d_{j+2}'\right\|_{\adj,+}+\sum_{\ell\in\PP_j^\pm}
\left\|d_\ell-d_\ell'\right\|_{\el,\pm}\right)\\
&\leq  K\sigma M_{\adj,-}\left(\wh C^{(j)}\de\right)^{1/2}
e^{-\sqrt{3}t}\left\|(u,v,d)-(u',v',d')\right\|_\ast.
 \end{split}
\]
Therefore, integrating and applying norms, we obtain
\[
\begin{split}
\left\|\int_0^{T_j} e^{-it}\left(\wt R_{\mix,c_{j-2}}
\left(ue^{\sqrt{3}t},ve^{-\sqrt{3}t},de^{it}\right)-\wt R_{\mix,c_{j-2}}
\left(u'e^{\sqrt{3}t},v'e^{-\sqrt{3}t},d'e^{it}\right)\right)dt\right\|_{\adj,-}
\\
\quad\quad\quad\quad\quad\quad\quad\quad\quad\leq
K\sigma\|(u,v,d)-(u',v',d')\|_\ast,
\end{split}
\]
which leads to
\[
\left\|\GG_{\el,c_{j-2}}(u,v,d)-\GG_{\el,c_{j-2}}(u',v',d')\right\|_{\adj,-}\leq
K\sigma\|(u,v,d)-(u',v',d')\|_\ast.
\]
Now we bound the Lipschitz constant for the hyperbolic components of the
operator. Note that we only need to bound the terms involving $\wt R_{\mix,z}$
since the other terms of the operator have been bounded in
%M
Proposition
\ref{lemma:Hyp:Contractive}. We start with the Lipschitz constants of
$\GG_{\hyp,v_i}$. To this end we bound
\[
\begin{split}
& \left|\int_0^{T_j} e^{\sqrt{3}t}\left(\wt R_{\mix, y_i}
(ue^{\sqrt{3}t},ve^{-\sqrt{3}t},de^{it})-\wt R_{\mix, y_i}
(ue^{\sqrt{3}t},ve^{-\sqrt{3}t},de^{it})\right)dt\right|\\
&\qquad\qquad\leq \int_0^{T_j} \left(\OO\left(\sum_{k\in\PP_j}
|d_k|^2(v_1+v_2)\right)
e^{\sqrt{3}t}\left|u_i-u_i'\right|+\OO\left(\sum_{k\in\PP_j} |d_k|^2\right) \sum
|v_i-v_i'|\right)dt\\
&\qquad\qquad\qquad+\int_0^{T_j} \sum_{k\in \PP_j} \OO(d_k( v_1+v_2))
|d_k-d_k'|dt,
\end{split}
\]
where we abuse notation concerning the $\OO(\cdot)$ as before. Thus, integrating
the exponentials and applying norms, one can easily see that
\[
\begin{split}
& \left|\int_0^{T_j} e^{\sqrt{3}t}\left(\wt R_{\mix, y_i}
(ue^{\sqrt{3}t},ve^{-\sqrt{3}t},de^{it})-\wt R_{\mix, y_i}
(ue^{\sqrt{3}t},ve^{-\sqrt{3}t},de^{it})\right)dt\right|\\
&\qquad\qquad\leq K_\sigma N \de^{1-r} \ln(1/\de) \|(u,v,d)-(u',v',d')\|_\ast.
\end{split}
\]
Therefore, applying norms and using condition on $\de$ from
Theorem \ref{thm:ToyModelOrbit}, we obtain
\[
\begin{split}
& \left\|\int_0^{T_j} e^{\sqrt{3}t}\left(\wt R_{\mix, y_i}
\left(ue^{\sqrt{3}t},ve^{-\sqrt{3}t},de^{it}\right)-\wt R_{\mix, y_i}
\left(ue^{\sqrt{3}t},ve^{-\sqrt{3}t},de^{it}\right)\right)dt\right\|_{\hyp, v_1}\\
&\qquad\qquad\leq K_\sigma \de^{1-r} \ln^2(1/\de) \|(u,v,d)-(u'v',d')\|_\ast\\
& \left\|\int_0^{T_j} e^{\sqrt{3}t}\left(\wt R_{\mix, y_i}
\left(ue^{\sqrt{3}t},ve^{-\sqrt{3}t},de^{it}\right)-\wt R_{\mix, y_i}
\left(ue^{\sqrt{3}t},ve^{-\sqrt{3}t},de^{it}\right)\right)dt\right\|_{\hyp, v_2}\\
&\qquad\qquad\leq K_\sigma \de^{1/2-2r} \ln(1/\de) \|(u,v,d)-(u'v',d')\|_\ast.
\end{split}
\]
Then,  taking into account the results of Lemma \ref{lemma:Hyp:Contractive}, one
can conclude that
\[
\begin{split}
&\left\|\wt\GG_{\hyp,v_1}(u,v,d)-\wt\GG_{\hyp,v_1}(u',v',d')\right\|_{\hyp,v_1}
\leq\\&\qquad\qquad\qquad\leq K_\sigma\left(\left(\wh C^{(j)}
\de\right)^{1/2}\ln(1/\de)+\de^{1-r}
\ln^2(1/\de)\right)\left\|(u,v,d)-(u',v',d')\right\|_\ast\\
&\left\|\wt\GG_{\hyp,v_2}(u,v,d)-\wt\GG_{\hyp,v_2}(u',v',d')\right\|_{\hyp,v_2}
\leq\\&\qquad\qquad\qquad\leq K_\sigma\left(\left(\wh C^{(j)}
\de\right)^{1/2}\ln(1/\de)+ \de^{1/2-2r} \ln(1/\de)
\right)\left\|(u,v,d)-(u',v',d')\right\|_\ast.
\end{split}
\]
Proceeding in the same way, one can obtain that
\[
\begin{split}
&\left\|\wt\GG_{\hyp,u_1}(u,v,d)-\wt\GG_{\hyp,u_1}(u',v',d')\right\|_{\hyp,u_1}
\leq\\
&\qquad\qquad\qquad\leq  K_\sigma\left(\left(\wh C^{(j)}
\de\right)^{1/2}\ln(1/\de)+\de^{1-r}
\ln^2(1/\de)\right)\left\|(u,v,d)-(u',v',d')\right\|_\ast\\
&\left\|\wt\GG_{\hyp,u_2}(u,v,d)-\wt\GG_{\hyp,u_2}(u',v',d')\right\|_{\hyp,u_2}
\leq \\
&\qquad\qquad\qquad\leq K_\sigma\left(\left(\wh C^{(j)}
\de\right)^{1/2}\ln(1/\de)+\de^{1/2-2r}
\ln^2(1/\de)\right)\left\|(u,v,d)-(u',v',d')\right\|_\ast.
\end{split}
\]
%M
This completes
%which finishes
the proof.
\end{proof}

\section{The global map: proof of Lemma
\ref{lemma:iterative:hetero}}\label{sec:ProofHeteroMap}
We devote this section to prove Lemma \ref{lemma:iterative:hetero}.
The continuous dependence with respect to initial conditions of ordinary
differential equations gives for free that the map $\BB_\glob^j$,
defined in \eqref{def:HeteroMap}, is well defined for points
close enough to the heteroclinic connection defined in \eqref{def:heteroclinic}.
Nevertheless, to prove Lemma \ref{lemma:iterative:hetero}, we need
more accurate estimates.

Recall that the map $\BB_\glob^j$ is defined in $\Sigma_j^\out$, which
is contained in  $\MM(b)=1$ (see \eqref{def:mass}). So, as we have
done for $\BB_\loc^j$, we use the system of coordinates defined in
Section \ref{sec:SaddleMapFirstChanges}.  Recall that the initial
section $\Sigma^{\out}_{j}$, defined in \eqref{def:Section2Saddle},
and the final section $\Sigma^{\inn}_{j+1}$,
defined in \eqref{def:Section1Saddle}, are expressed in the variables
adapted to the
%M
$j^{th}$ and $(j+1)^{st}$ saddles respectively. Namely, in the coordinates
$(p_1^{(j)}, q_1^{(j)},  p_2^{(j)}, q_2^{(j)},
c^{(j)})$ and  $(p_1^{(j+1)}, q_1^{(j+1)},  p_2^{(j+1)}, q_2^{(j+1)},
c^{(j+1)})$ (see Section \ref{sec:ProofHeteroMap}). To simplify
the exposition, first we will study the map $\BB_\glob^j$ expressing
both the domain and the image in the variables
$( p_1^{(j)}, q_1^{(j)},  p_2^{(j)}, q_2^{(j)}, c^{(j)})$. Then we will
express the image of $\BB_\glob^j$ in the new variables. To simplify
notation we denote the variables adapted to the
$j^{th}$ and $(j+1)^{st}$ saddles by
\[
( p_1, q_1, p_2, q_2, c)=
\left(p_1^{(j)},q_1^{(j)}, p_2^{(j)}, q_2^{(j)}, c^{(j)}\right)
\]  and
\[
\left(\wt p_1,\wt q_1, \wt p_2, \wt q_2, \wt c\right)=
\left(p_1^{(j+1)},q_1^{(j+1)}, p_2^{(j+1)},  q_2^{(j+1)},
c^{(j+1)}\right)
\]
and we
denote by $\Theta^j$ the change of coordinates that relates them, namely
\[
 (\wt p_1,\wt q_1, \wt p_2, \wt q_2, \wt c)=\Theta^j( p_1, q_1, p_2, q_2, c).
\]
%\begin{equation}\label{def:ChangeReducedCoordinates}
% (\wt p_1,\wt q_1, \wt p_2, \wt q_2, \wt c)=\Theta^j( p_1, q_1, p_2, q_2, c).
%\end{equation}

\begin{lemma}\label{lemma:ChangeOfSaddle}
The change of coordinates $\Theta^j$ is given by
\begin{align*}
\Theta_{\wt c_k}^j( p_1, q_1, p_2, q_2, c)&=\frac{\omega p_2+\omega^2 q_2}{\wt r}c_k &\text{ for }k\in \PP_{j+1}^\pm \cup \{j+3\}\\
\Theta_{\wt c_{j-1}}^j( p_1, q_1, p_2, q_2, c)&=\frac{\omega p_2+\omega^2 q_2}{\wt r}\left(\omega^2 p_1+\omega q_1\right)\\
\Theta_{\wt p_1}^j( p_1, q_1, p_2, q_2, c)&= \frac{r}{\wt r}q_2\\
\Theta_{\wt q_1}^j( p_1, q_1, p_2, q_2, c)&= \frac{r}{\wt r}p_2\\
\Theta_{\wt p_2}^j( p_1, q_1, p_2, q_2, c)&= \Re z+\frac{\sqrt{3}}{3}\Im z\\
\Theta_{\wt q_2}^j( p_1, q_1, p_2, q_2, c)&= \Re z-\frac{\sqrt{3}}{3}\Im z,
\end{align*}
where $\omega=e^{2\pi i/3}$ and
\begin{align}
r^2=&1-\sum_{k\neq j-1,j,j+1} |c_k|^2-(p_1^2+q_1^2-p_1q_1)-(p_2^2+q_2^2-p_2q_2)\label{def:RadiusInSaddle}\\
\wt r^2=&p_2^2+q_2^2- p_2q_2\notag\\
z=&\frac{c_{j+2}}{\wt r}\left(\omega p_2+\omega^2 q_2\right)\notag.
\end{align}
\end{lemma}
\begin{proof}
We consider a point $(p,q,c)$ and we express it in the new variables.
We have to undo the changes \eqref{def:ChangeToDiagonal} and
\eqref{def:SaddleAdaptedCoordinates} referred to the saddle $j$ and
then apply them again but referred to the saddle $j+1$.
The point $(p,q,c)$ has associated variables $r$
(as defined in \eqref{def:RadiusInSaddle}) and $\theta$.
We do not need to know the value of $\theta$ to deduce the form of
the change $\Theta^j$. Indeed, note that if we consider the changes
%M changed order
 \eqref{def:SaddleAdaptedCoordinates} and \eqref{def:ChangeToDiagonal}
for the mode $b_{j+1}$ we have
\[
\wt  re^{i\wt \theta}=b_{j+1}=c_{j+1}e^{i\theta}=
\left(\omega^2 p_2+\omega q_2\right)e^{i\theta},
\]
which implies
\begin{equation}\label{def:Change:AngleRelation}
 e^{i(\theta-\wt\theta)}=\frac{\omega p_2+\omega^2 q_2}{\wt r}.
\end{equation}
Using this formula and recalling that
$\wt c_k e^{i\wt\theta}=b_k=c_k e^{i\theta}$, it is straightforward
to deduce the form of $\Theta_{\wt c_k}^j$ for
$k\in \PP_{j+1}^\pm \cup \{j+3\}$. To deduce the form of
$\Theta_{\wt p_1}^j$ and $\Theta_{\wt q_1}^j$ it is enough
to consider the changes
 \eqref{def:SaddleAdaptedCoordinates} and \eqref{def:ChangeToDiagonal}
for the mode $b_j$ to obtain
\[
 re^{i\theta}=b_j=\wt c_je^{i\wt \theta}=
\left(\omega^2 \wt p_1+\omega \wt q_1\right)e^{i\wt \theta}
\]
Then, it is enough to use formula \eqref{def:Change:AngleRelation}
to obtain $\Theta_{\wt p_1}^j$ and $\Theta_{\wt q_1}^j$. The others
components can be obtained proceeding %M analogously
in the same way.
\end{proof}

The next step %M
of the proof of
% prove
Lemma \ref{lemma:iterative:hetero} is to express the
section $\Sigma_{j+1}^{\inn}$ in the variables $(p_1,q_1, p_2,  q_2, c)$
using the change $\Theta^j$ obtained in Lemma \ref{lemma:ChangeOfSaddle}.
This is done in the next corollary, which is a straightforward
consequence of Lemma  \ref{lemma:ChangeOfSaddle}.

\begin{corollary}\label{coro:SectionOldVariables}
Fix $\sigma>0$ and define the set
\[
 \wt \Sigma_{j+1}^{\inn}=\left(\Theta^j\right)^{-1}
\left(\Sigma_{j+1}^{\inn}\cap\WW_{j+1}\right),
\]
%\begin{equation}\label{def:TiltedSectionGlobal}
% \wt \Sigma_{j+1}^{\inn}=\left(\Theta^j\right)^{-1}\left(\Sigma_{j+1}^{\inn}\cap\WW_{j+1}\right),
%\end{equation}
where $\Sigma^{\inn}_{j+1}$ is the section defined in
\eqref{def:Section1Saddle} and
\[
 \WW_{j+1}=\left\{|p_1|\leq \eta, |q_1|\leq \eta,
|q_2|\leq \eta, |c_k|\leq \eta\,\,\text{for }k\in\PP_j^\pm
\,\text{ and }k=j\pm 2\right\},
\]
Then, for $\eta>0$ small enough, $ \WW_{j+1}$
can be expressed  as a graph as
\[
 p_2=w(p_1,q_1,  q_2, c).
\]
Moreover, there exist constants $\kk',\kk''$
independent of $\eta$   satisfying
\[
 0<\kk'<\sqrt{1-\sigma^2}<\kk''<1
\]
such that, for any  $(p_1,q_1,q_2,c)\in\WW_{j+1}$,
the function $w$ satisfies
\[
\kk'< w(p_1,q_1,  q_2, c)<\kk''.
\]
\end{corollary}

Once we have defined the section $\wt\Sigma_{j+1}^{\inn}$,
we can define the map
\[
\begin{array}{cccc}
  \wt\BB_\glob^j:&\UU_j\subset\Sigma_j^{\out}&\longrightarrow
&\wt\Sigma_{j+1}^{\inn}\\
&(p_1,q_1,q_2,c)&\mapsto&\wt\BB_\glob^j(p_1,q_1,q_2,c)
\end{array}
 \]
as
\[
 \wt\BB_\glob^j=\Theta_j^{-1}\circ\BB_\glob^j.
\]
%\begin{equation}\label{def:GlobalMapTransformed}
% \wt\BB_\glob^j=\Theta_j^{-1}\circ\BB_\glob^j.
%\end{equation}
We want upper bounds independent of $\de$ and $j$ for
%M the time spent by this map.
the transition time of the correponding orbits for
this map. In the variables $(p_1,q_1, p_2,  q_2, c)$
the heteroclinic connection \eqref{def:heteroclinic} is
simply given by
\begin{equation}\label{def:heteroclinic:straightened}
\left(p_1^h(t),q_1^h(t),p_2^h(t),q_2^h(t),c^h(t)\right)=
\left(0,0,\frac{1}{1+e^{2\sqrt{3}(t-t_0)}},0,0\right)
\end{equation}
(see \cite{CollianderKSTT10}). Taking $t_0$ such that
\[
\frac{1}{1+e^{2\sqrt{3}t_0}}=\sigma,
\]
one can easily see that $p_2^h(2t_0)=\sqrt{1-\sigma^2}$ and $2t_0\sim\ln(1/\sigma)$.
In the new coordinates this point is  $(\wt p_1,\wt q_1, \wt p_2, \wt q_2, \wt
c)=(0,\sigma, 0,0,0)$ and thus belongs to the section $\wt q_1=\sigma$. Then, thanks to Corollary \ref{coro:SectionOldVariables}, one can easily deduce
that the time $T_{\wt \BB_\glob^j}=T_{\wt \BB_\glob^j}(q_1,p_1,p_2,c)$ spent by
the map $\wt\BB_\glob^j$ for any point $(q_1,p_1,p_2,c)\in\UU_j\subset\Sigma_j^\out$ is also independent of $\de$ and $j$. Recall that the
difference between $\wt\BB_\glob^j$ and $\BB_\glob^j$ is just a change of
coordinates and therefore the time $T_{\BB_\glob^j}$ spent  by $\BB_\glob^j$ is the same
as $T_{\wt \BB_\glob^j}$. Thus, from know on we will only refer to $T_{\BB_\glob^j}$.

Next step is to study the behavior of the map $\wt \BB_\glob^j$. In particular, we want to know the properties of the image set $\wt \BB_\glob^j(\UU_j)$.
\begin{proposition}\label{prop:HeteroMap}
Let us consider a parameter set $\wt \II_j$ (as defined in Definition \ref{definition:ModifiedProductLike}) and a $\wt \II_j$-product-like set $\UU_j$. Then, there exists a constant $\wt K_\sigma$ independent of $j$,  $N$ and $\de$ and a constant $D^{(j)}$ satisfying
\[
\wt C^{(j)} /\wt K_\sigma\leq D^{(j)}\leq \wt K_\sigma \wt C^{(j)},
\]
such that the set $\wt \BB_\glob^j(\UU_j)\subset \wt\Sigma_j^\inn$ satisfies the following conditions:
\begin{description}
 \item[\textbf{C1}]
\[
\wt \BB_\glob^j(\UU_j)\subset \wh\DD_j^1\times\ldots\times\wh\DD_j^{j-2}\times \SSS_j\times
\wh\DD_j^{j+2}\times\ldots\times\wh\DD_j^{N}
\]
where
\begin{align*}
\wh\DD_j^k&=\left\{\left|c_k\right|\leq  \left(\wt M^{(j)}_{\el,\pm}+\wt K_\sigma\de^{r'}\right)
\de^{(1-r)/2}\right\} \,\,\text{ for }k\in \PP_j^\pm\\
\wh\DD_j^{j\pm 2}&\subset\left\{\left|c_{j\pm2}\right|\leq
\wt K_\sigma\wt M^{(j)}_{\adj,\pm} \left(\wt C^{(j)}\de\right)^{1/2}\right\},
\end{align*}
and
\[
\begin{split}
\SSS_{j}= \Big\{& (p_1,q_1,p_2,q_2)\in \RR^4:
|p_1|,|q_1|\leq \wt K_\sigma\wt M_\hyp^{(j)}\left(\wt C^{(j)}\de\right)^{1/2},\\
&p_2=\sigma,  - D^{(j)}\,\de\,\left(\ln(1/\de)-\wt K_\sigma\right)\leq
q_2^{(j)}\leq -D^{(j)}\,\de\,\left(\ln(1/\de)+\wt K_\sigma\right) \Big\},
\end{split}
\]
\item[\textbf{C2}] Let us define the projection $\wt \pi(p,q,c)=(p_2,q_2,c_{j-2},\ldots,c_N)$. Then,
\[
\left[- D^{(j)}\,\de\,(\ln(1/\de)-1/\wt K_\sigma), -
D^{(j)}\,\de\,(\ln(1/\de)+1/\wt K_\sigma)\right]\times  \{\sigma\}\times
\DD_{j,-}^{j+2}\times\ldots\times\DD_{j,-}^{N} \subset\wt\pi\left(\wt \BB_\glob^j(\UU_j)\right)
\]
where
\begin{align*}
\DD_{j,-}^k&=\left\{\left|c_k^{(j)}\right|\leq  \left(\wt m^{(j)}_{\el}-\wt K_\sigma \de^{r'}\right)
\de^{(1-r)/2}\right\} \,\,\text{ for }k\in \PP_j^+\\
\DD_{j,-}^{j+ 2}&=\left\{\left|c_{j+2}^{(j)}\right|\leq
\wt m^{(j)}_{\adj} \left(C^{(j)}\de\right)^{1/2}/\wt K_\sigma\right\}.
\end{align*}
\end{description}

\end{proposition}
The proof of this proposition is postponed to Section
\ref{sec:ProofPropHeteroMap}.

Once we know the properties of the set $\wt \BB_\glob^j(\UU_j)$, there only remain two final steps. First to deduce  analogous properties for the set $\BB_\glob^j(\UU_j)\subset \Sigma_{j+1}^\inn$. Second, to obtain a parameter set $\II_{j+1}$ and $\II_{j+1}$-product like   set $\VV_j\subset \Sigma_{j+1}^\inn$ which satisfies condition \eqref{cond:ComposeMaps:Heteromap}. These two last steps are summarized in the next lemma. Lemma \ref{lemma:iterative:hetero} follows easily from it.

\begin{lemma}\label{lemma:HeteroMapOriginalVars}
Let us consider a parameter set  $\II_{j+1}$ whose constants
satisfy
\[
 \begin{split}
  D^{(j)}/2\leq C^{(j+1)}\leq 2 D^{(j)}\\
0<m_{\hyp}^{(j+1)}\leq \wt m_\hyp^{(j)}
 \end{split}
\]
and
\[
 \begin{split}
M_{\el,-}^{(j+1)}&=\max\left\{\wt M_{\el,-}^{(j)}+\wt K_\sigma\de^{r'},\wt
K_\sigma\wt M_{\adj,-}^{(j)}\right\}\\
M_{\el,+}^{(j+1)}&=\wt M_{\el,+}^{(j)}+\wt K_\sigma\de^{r'}\\
m_{\el}^{(j+1)}&=\wt m_{\el}^{(j)}-\wt K_\sigma\de^{r'}\\
M_{\adj,+}^{(j+1)}&=\wt m_{\el,+}^{(j)}+\wt K_\sigma\de^{r'}\\
M_{\adj,-}^{(j+1)}&=\wt K_\sigma\wt M_{\hyp}^{(j)}\\
m_{\adj}^{(j+1)}&=\wt m_{\el}^{(j)}+\wt K_\sigma\de^{r'}\\
M_{\hyp}^{(j+1)}&=\max\left\{\wt K_\sigma \wt M_{\adj,+}^{(j)},\wt K_\sigma\right\}.
 \end{split}
\]
Then, the set
\[
 \VV_{j+1}=\BB_\glob^j(\UU_j)\cap\left\{g_{\II_{j+1}}(p_2,q_2,\sigma,\delta)=0\right\},
\]
where $g_{\II_{j+1}}$ is the function defined in \eqref{def:Function_g}, is a $\II_{j+1}$-product-like set and satisfies condition \eqref{cond:ComposeMaps:Heteromap}
\end{lemma}

\begin{proof}
It is enough to apply the change of coordinates $\Theta^j$ given in Lemma \ref{lemma:ChangeOfSaddle}.
\end{proof}

\subsection{Proof of Proposition
\ref{prop:HeteroMap}}\label{sec:ProofPropHeteroMap}
We split the proof of Proposition \ref{prop:HeteroMap} in several lemmas, which will give the needed estimates
for the different modes. First, let us obtain rough bounds for all the
variables, which will be used in the proofs of the forthcoming lemmas. Indeed,
since we are restricted to $\MM(b)=1$ (see \eqref{def:mass}) we know that
\begin{equation}\label{eq:HeteroMap:RoughBoundsEl}
|c_m|<1.
\end{equation}
Analogously, using the change \eqref{def:ChangeToDiagonal}, one can see that
\begin{equation}\label{eq:HeteroMap:RoughBoundsHyp}
|p_i|<2, \,\,|q_i|<2\,\,\,\text{ for }i=1,2.
\end{equation}
Now, we start by obtaining more accurate upper bounds for each mode.
\begin{lemma}\label{lemma:HeteroMap:UpperBounds}
Consider the flow $\Phi^t$ associated to the vector field in
\eqref{def:VF:Full:AfterDiagonal} and a point $(p_1,q_1,q_2,\sigma,c)\in\UU_j\subset\Sigma_j^{\out}$.
Then, there exists a constant $\wt K_\sigma>0$ such that for $t\in [0,T_{\BB_\glob^j}]$, $\Phi^t(p_1,q_1,\sigma,q_2,c)$ satisfies
\[
 \begin{split}
  \left|\Phi_{c_k}^t(p_1,q_1,\sigma,q_2,c)\right|&\leq \wt K_\sigma \wt M^{(j)}_{\el,\pm}\de^{(1-r)/2} \,\,\,\text{ for }m\in\PP_j^\pm\\
\left|\Phi_{c_{j\pm 2}}^t(p_1,q_1,\sigma,q_2,c)\right|&\leq  \wt K_\sigma \wt M^{(j)}_{\adj,\pm}\left(\wt C^{(j)}\de\right)^{1/2}\\
 \end{split}
\]
and
\[
\begin{split}
\left|\Phi_{p_1}^t(p_1,q_1,\sigma,q_2,c)\right|&\leq  \wt K_\sigma \wt M^{(j)}_{\hyp}\left(\wt C^{(j)}\de\right)^{1/2}\\
\left|\Phi_{q_1}^t(p_1,q_1,\sigma,q_2,c)\right|&\leq  \wt  K_\sigma\wt M^{(j)}_{\hyp} \left(\wt C^{(j)}\de\right)^{1/2}\\
\left|\Phi_{p_2}^t(p_1,q_1,\sigma,q_2,c)-p_2^h(t)\right|&\leq  \wt  K_\sigma \de^{r'}\\
\left|\Phi_{q_2}^t(p_1,q_1,\sigma,q_2,c)\right|&\leq  \wt K_\sigma\wt
C^{(j)}\de\ln(1/\de).
\end{split}
\]
\end{lemma}
We defer the proof of this lemma to the end of the section.

The bounds obtained in Lemma \ref{lemma:HeteroMap:UpperBounds} are not enough to prove Proposition \ref{prop:HeteroMap} since  we need more accurate estimates for the elliptic modes, the future adjacent modes and $q_2$. We obtain them in the following three lemmas.

\begin{lemma}\label{lemma;HeteroMap:EllipticModes}
Consider the flow $\Phi^t$ associated to the vector field in
\eqref{def:VF:Full:AfterDiagonal} and a point $(p_1,q_1,\sigma,q_2,c)\in\Sigma_j^{\out}$.
Then, there exists a constant $\wt K_\sigma>0$ such that for $t\in [0,T_{\BB_\glob^j}]$ and $k\in\PP_j^\pm$,
\[
 \left|\Phi_{c_k}^t(p_1,q_1,\sigma,q_2,c)-c_k e^{iT_{\BB_\glob^j}}\right|\leq \wt
K_\sigma\de^{(1-r)/2+r'}.
\]
\end{lemma}

\begin{proof}
It is enough to point out that, using the bounds obtained in
Lemma \ref{lemma:HeteroMap:UpperBounds}, the equation for $c_k$ in
\eqref{def:VF:Full:AfterDiagonal} can be written as
\[
\dot c_k=i c_k+\gamma_k(t)
\]
where $\gamma$ satisfies $\|\gamma\|_\infty\leq \wt K_\sigma \de^{1-r+r'}$.
Then, to finish the proof of the lemma it is enough to apply the variation
of constants formula and take into account that the time $T_{\BB_\glob^j}$ has an upper
bound independent of $\de$.
\end{proof}

\begin{lemma}\label{lemma;HeteroMap:AdjacentModes}
Fix values $p_1,q_1,q_2, c_{j-2}$ and $c_k$ for $k\in\PP_j^\pm$ such that the set
\[
\DDD=\left\{c_1,\ldots,c_{j-2},p_1,q_1,\sigma,q_2\right\}\times \wt \DD^{j+2}_{j,-}\times \left\{c_{j+3},\ldots,c_{j_N}\right\},
\]
where
\[
\wt \DD^{j+2}_{j,-}=\left\{\left|c_{j+2}\right|\leq\wt m_\adj^{(j)}\left(\wt C^{(j)}\de\right)^{1/2}\right\},
\]
satisfies
\[
\DDD\subset\UU_j.
\]
Consider the flow $\Phi^t$ associated to the vector field in
\eqref{def:VF:Full:AfterDiagonal} and define the following map for points in $\DDD$
\[
F_{\adj}(p_1,q_1,\sigma,q_2,c)=\Phi_{c_{j+ 2}}^{T_{\BB_\glob^j}}(p_1,q_1,\sigma,q_2,c)
\]
Then, there exists $\wt K_\sigma>0$ such that
\[
\left\{\left|c_{j+2}\right|\leq \wt m_\adj^{(j)}\left(\wt C^{(j)}\de\right)^{1/2}/ \wt K_\sigma\right\}\subset F_{\adj}(\DDD).
\]
\end{lemma}

\begin{proof}
Taking into account the estimates obtained in Lemma \ref{lemma:HeteroMap:UpperBounds},  the equation for $c_{j+2}$ in \eqref{def:VF:Full:AfterDiagonal} can be written as
\[
\frac{d}{dt}\left(\begin{array}{c}
\dps c_{j+2}\\ \dps\ol{c_{j+2}}\end{array}\right) =
\left(\begin{array}{c}
i c_{j+2}- i\omega \left(p_2^h(t)\right)^2 \ol{c_{j+2}}+\gamma_{j+2}(t)\\
-i\ol{c_{j+2}}+ i\omega^2 \left(p_2^h(t)\right)^2 {c_{j+2}}+\ol\gamma_{j+2}(t)
\end{array}\right),
\]
where  $p_2^h$ has been defined in \eqref{def:heteroclinic:straightened}  and  $\gamma$ satisfies $\|\gamma\|_\infty\leq K_\sigma  (\wt C^{(j)}\de)^{1/2}\de^{r'}$. Then, to finish the proof it is enough to apply the variation of constants formula.
\end{proof}

Now we obtain the refined estimates for  $q_2$.
\begin{lemma}
Fix values $p_1,q_1, c_{j\pm 2}$ and $c_k$ for $k\in\PP_j^\pm$ such that
\[
\QQQ=\left\{c_1,\ldots,c_{j-2},p_1,q_1,\sigma\right\}\times \left[-\wt C^{(j)}\de\left(\ln(1/\de)-\wt m_\hyp^{(j)}\right),-\wt C^{(j)}\de\left(\ln(1/\de)+\wt m_\hyp^{(j)}\right)\right]\times \left\{c_{j+2},\ldots,c_{j_N}\right\}
\]
satisfies
\[
\QQQ\subset\UU_j.
\]
Consider the flow $\Phi^t$ associated to the vector field in
\eqref{def:VF:Full:AfterDiagonal} and define the following map for points in $\QQQ$
\[
F_{\hyp}(q_2)=\Phi_{q_2}^{T_{\BB_\glob^j}}(p_1,q_1,\sigma,q_2,c)
\]
Then, there exists $\wt K_\sigma>0$ and $D^{(j)}$ satisfying
\[
\wt C^{(j)}/\wt K_\sigma\leq D^{(j)}\leq \wt K_\sigma \wt C^{(j)}
\]
such that
\[
\left[-D^{(j)}\de\left(\ln(1/\de)-1/\wt K_\sigma\right),-D^{(j)}\de\left(\ln(1/\de)+1/\wt K_\sigma\right)\right]\subset F_\hyp(\QQQ).
\]
\end{lemma}

\begin{proof}
Taking into account the estimates obtained in Lemma \ref{lemma:HeteroMap:UpperBounds}, we write the equation for $q_2$ in  \eqref{def:VF:Full:AfterDiagonal} as
\[
 \dot q_2=\zeta_0(t)q_2+\zeta_1(t),
\]
where $\zeta_0$ only depends on $p_2^h$ in \eqref{def:heteroclinic:straightened} and $\zeta_1$ satisfies
\[
\|\zeta_1\|_\infty\leq \wt K_\sigma \wt C^{(j)}\de.
\]
Thus, the proof of the lemma follows from the variation of constants formula.
\end{proof}

We devote the rest of the section to prove Lemma \ref{lemma:HeteroMap:UpperBounds}.
\begin{proof}[Proof of Lemma \ref{lemma:HeteroMap:UpperBounds}]
During the proof of this lemma the time $t$ will always satisfy $t\in[0,T_{\BB_\glob^j}]$ and the norm $\|\cdot\|_\infty$ will always refer to the supremmum taken over this time interval .

We start by obtaining the bounds for the  non-neighbor elliptic modes. By \eqref{def:VF:Full:AfterDiagonal}, one can easily see that for $k\in\PP_j^\pm$
\[
 \frac{d}{dt}|c_k|^2=\frac{1}{2}\left(c_{k-1}^2+c_{k+1}^2\right)\ol
{c_k}^2-\frac{1}{2}\left(\ol{c_{k-1}}^2+\ol{c_{k+1}}^2\right){c_k}^2.
\]
Then, using \eqref{eq:HeteroMap:RoughBoundsEl}, we have that
\[
 \frac{d}{dt}|c_k|^2\leq\left|c_k\right|^2
\]
and therefore, applying Gronwall estimates we obtain that for $t\in
[0,T_\glob^j]$,
\[
\left|\Phi_{c_k}^t(p_1,q_1,\sigma,q_2,c)\right|^2\leq
e^{T_{\BB_\glob^j}}\left|c_k\right|^2\leq  \wt  K_\sigma \wt M^{(j)}_{\el,\pm}\de^{(1-r)}.
\]
Proceeding analogously we deal with the adjacent elliptic mode  $c_{j- 2}$.
Its associated equation is
\[
\begin{split}
 \frac{d}{dt}|c_{j- 2}|^2=&\frac{1}{2}c_{j-3}^2\ol{c_{j- 2}}^2+\frac{1}{2}\ol{c_{j-3}}^2 {c_{j- 2}}^2\\
&-\frac{1}{2}\left(\omega^2
p_1+\omega q_1\right)^2\ol{c_{j-2}}^2-\frac{1}{2}\left(\omega
p_1+\omega^2 q_1\right)^2{c_{j-2}}^2.
\end{split}
\]
Taking into account the bounds in
\eqref{eq:HeteroMap:RoughBoundsEl} and also \eqref{eq:HeteroMap:RoughBoundsHyp}, to  obtain
\[
  \frac{d}{dt}|c_{j- 2}|^2\leq 5|c_{j- 2}|^2
\]
which, applying Gronwall lemma, gives
\[
 \left|\Phi_{c_{j-2}}^t(p_1,q_1,\sigma,q_2,c)\right|^2\leq e^{5T_{\BB_\glob^j}} |c_{j-2}|^2\leq  \wt K_\sigma  \wt M^{(j)}_{\adj,-}\wt C^{(j)}\de.
\]
Analogously, one can obtain
\[
 \left|\Phi_{c_{j+2}}^t(p_1,q_1,\sigma,q_2,c)\right|^2\leq e^{5T_\glob^j} |c_{j+2}|^2\leq \wt  K_\sigma \wt M^{(j)}_{\adj,+}\wt C^{(j)}\de.
\]
Now we obtain the bounds for the hyperbolic modes. We define
\[
\rr_1(t)=(\Phi_{p_1}^t(p_1,q_1,\sigma,q_2,c),\Phi_{q_1}^t(p_1,q_1,\sigma,q_2,c)).
\]
From \eqref{def:Ham:Diagonal},
one can see that $\rr_1$ satisfies an equation of the form $\dot \rr_1= A_1(t) \rr_1$
where $A_1(t)$ is a time dependent matrix (which of course depends on
$\Phi_{p_1}^t(p_1,q_1,\sigma,q_2,c)$ itself). Using \eqref{eq:HeteroMap:RoughBoundsEl} and
\eqref{eq:HeteroMap:RoughBoundsHyp}, one can deduce that
\[
 \|A_1\|_\infty\leq \wt  K_\sigma.
\]
Then, the fundamental matrix
$\Psi$ satisfying $\Psi(0)=\mathrm{Id}$ associated to this system  satisfies
$\|\Psi\|_\infty\leq  \wt K_\sigma$. Since $\rr_1$ can be just written
\[
 \rr_1(t)=\Psi(t)\rr_1(0),
\]
using that by hypothesis $|p_1(0)|,|q_1(0)|\leq \wt M_\hyp^{(j)}\left(\wt C^{(j)}\de\right)^{1/2}$, we have that for $t\in[0,T_{\BB_\glob^j}]$,
\[
\left|\rr_1(t)\right|\leq   \wt K_\sigma \wt M_\hyp^{(j)}\left(\wt C^{(j)}\de\right)^{1/2}.
\]
We finish the proof of the lemma obtaining the estimates for
the $(p_2,q_2)$ components. To this end, let us point out that
the equation for $q_2$ can be written as
\[
\dot q_2=a_1(t)q_2+b_1(t)
\]
where $a_1(t)$ and $b_1(t)$ are functions which depend on
$\Phi_{p_1}^t(p_1,q_1,\sigma,q_2,c)$. Using \eqref{eq:HeteroMap:RoughBoundsHyp}
and the just obtained bounds for the  non-neighbor and adjacent elliptic modes and
for $(p_1,q_1)$ components, one can easily see that
\[
\|a_1\|_\infty\leq\wt K_\sigma \,\,\,\text{ and } \,\,\,\|b_1\|_\infty\leq \wt K_\sigma\left(\wt C^{(j)}\de\right)^{1/2}.
\]
Therefore, applying Gronwall lemma, we can deduce that
\[
 \left|\Phi_{q_2}^t(p_1,q_1,\sigma,q_2,c)\right|\leq \wt  K_\sigma \wt C^{(j)}\de\ln(1/\de).
\]
To obtain the bounds for $p_2$ we define $\xi=p_2-p_2^h$, where $p_2^h$ is the function defined in \eqref{def:heteroclinic:straightened}. Using \eqref{eq:HeteroMap:RoughBoundsHyp} and \eqref{def:heteroclinic:straightened} we have the a priori bound $\|\xi\|_\infty\leq 3$. Therefore, from \eqref{def:VF:Full:AfterDiagonal} we can deduce an equation for $\xi$ of the form
\[
\dot \xi=a_2(t)\xi+b_2(t),
\]
where the functions $a_2$ and $b_2$ satisfy
\[
\|a_2\|_\infty\leq K_\sigma \,\,\,\text{ and }
\,\,\,\|b_2\|_\infty\leq  \wt  K_\sigma \de^{r'}.
\]
Then, applying Gronwall's lemma, we obtain
\[
\|\xi\|_\infty\leq  \wt K_\sigma \de^{r'}
\]
which implies the estimate for $\Phi_{p_2}^t(p_1,q_1,\sigma,q_2,c)-p_2^h$.
This finishes the proof of the lemma.
\end{proof}

\appendix

\section{Proof of Normal Form Theorem \ref{thm:NormalForm}}\label{app:NormalForm}
To proof of Theorem \ref{thm:NormalForm}, we consider as a change of variables
$\Gamma$  the time one map of a Hamiltonian vector field $X_F$, where
$F$ is the Hamiltonian
\[
 F=\frac{1}{4}\sum_{n_1,n_2,n_3,n_4\in\ZZ^2}F_{n_1n_2n_3n_4}\al_{n_1}\ol{\al_{n_2}},
\al_{n_3}\ol{\al_{n_4}}
\]
with
\begin{align*}
F_{n_1n_2n_3n_4}=&
\frac{-i}{|n_1|^2-|n_2|^2+|n_3|^2-|n_4|^2}
\qquad\qquad\text{ if } n_1-n_2+n_3-n_4=0, \\
&\qquad\qquad\qquad\qquad\qquad\qquad\qquad\qquad|n_1|^2-|n_2|^2+|n_3|^2-|n_4|^2\neq 0\\
F_{n_1n_2n_3n_4}=&0 \qquad\qquad\qquad\qquad\qquad\qquad\qquad\quad\,\,\ \text{otherwise.}&
\end{align*}
If we define $\Phi_F^t$ the flow of the vector field associated to the
Hamiltonian $F$, we have that
\[
 \begin{split}
\HH\circ\Gamma=&\left.H\circ   \Phi_F^t\right|_{t=1}\\
=&\HH+\{\HH,F\}+\int_0^1(1-t)\left\{\left\{\HH,F\right\},F\right\}
\circ\Phi_F^tdt\\
=&\DDD +\GG+\{\DDD, F\}\\
&+\{\GG,F\}+\int_0^1(1-t)\left\{\left\{\HH,F\right\},F\right\}\circ\Phi_F^tdt,
 \end{split}
\]
where $\{\cdot,\cdot\}$ denotes the Poisson bracket with respect to the symplectic form $\Omega=\frac{i}{2}\sum_{n\in\ZZ^2}\al_n\wedge \ol{\al_n}$. We define
\[
 \RRR=\{\GG,F\}+\int_0^1(1-t)\left\{\left\{\HH,F\right\},F\right\}\circ\Phi_F^tdt.
\]
Then, it only remains to obtain the desired bounds for $X_\RRR$ and $\Gamma$ and
to see that
\[
 \GG+\{\DDD, F\}=\wt \GG.
\]
To obtain, this last equality, it is enough to use the definition for $F$ to see
that
\[
\begin{split}
  \GG+\{\DDD,
F\}=&\frac{1}{4}\sum_{n_1-n_2+n_3=n_4}\left(1-i(|n_1|^2-|n_2|^2+|n_3|^2-|n_4|^2)F_{n_1n_2n_3n_4}\right)\al_{
n_1}\ol{\al_{n_2}}\al_{n_3}\ol{\al_{n_4}}\\
=&\frac{1}{4}\sum_{\substack{n_1-n_2+n_3=n_4\\|n_1|^2-|n_2|^2+|n_3|^2=|n_4|^2}}
\al_{n_1}\ol{\al_{n_2}}\al_{n_3}\ol{
\al_{n_4}}\\
=& \wt\GG.
\end{split}
\]
Now we obtain the  bounds for $X_\RRR$. We start by bounding $X_{\{\GG,F\}}$, the
vector field associated to the Hamiltonian $\{\GG,F\}$. We have to bound
\[
\left\|X_{\{\GG,F\}}\right\|_{\ell^1}=2 \sum_{n\in \ZZ^2}
\left|\pa_{\ol{\al_n}}\{\GG,F\}\right|.
\]
Then,
\[
 \begin{split}
  \left\|X_{\{\GG,F\}}\right\|_{\ell^1}\leq& 2 \sum_{n,m\in \ZZ^2}
\left|\pa_{\ol{\al_n}}\left(\pa_{\ol{\al_m}}\GG\pa_{\al_m} F\right)\right|+ 2 \sum_{n,m\in
\ZZ^2} \left|\pa_{\ol{\al_n}}\left(\pa_{\al_m}\GG\pa_{\ol{\al_m}} F\right)\right|\\
\leq &2\sum_{n,m\in \ZZ^2} \left|\pa_{\ol{\al_n}\ol{\al_m}}\GG\right|\left|\pa_{\al_m}
F\right|+ 2 \sum_{n,m\in \ZZ^2} \left|\pa_{\ol{\al_m}}\GG\right|\left|\pa_{\ol{\al_n}\al_m}
F\right|\\
&+2\sum_{n,m\in \ZZ^2} \left|\pa_{\ol{\al_n}\al_m}\GG\right|\left|\pa_{\ol{\al_m}} F\right|+
2 \sum_{n,m\in \ZZ^2} \left|\pa_{\al_m}\GG\right|\left|\pa_{\ol{\al_n}\ol{\al_m}} F\right|.
 \end{split}
\]
All the terms can be bounded analogously. As an example, we bound the first one,
\[
\begin{split}
\sum_{n,m\in \ZZ^2} \left|\pa_{\ol{\al_n}\ol{\al_m}}\GG\right|\left|\pa_{\al_m}
F\right|\leq &4\sum_{n,m\in \ZZ^2} \left|\sum_{n_1+n_2=m+n}\al_{n_1}
{\al_{n_2}}\right|\left|\sum_{n_1-n_2+n_3=m}\ol{\al_{n_1}}\al_{n_2}\ol{\al_{n_3}}
\right|\\
\leq &4\sum_{n\in \ZZ^2} \sum_{n_1+n_2=n}|\al_{n_1}||\al_{n_2}|\sum_{m\in
\ZZ^2}\sum_{n_1-n_2+n_3=m}|\al_{n_1}||\al_{n_2}||\al_{n_3}|\\
\leq &\OO\left(\|\al\|_{\ell^1}^5\right),
\end{split}
\]
where, in the first line we have taken into account that
$|F_{n_1n_2n_3n_4}|\leq 1$, and  to obtain the last line we have used that
each sum in the previous line is a convolution product. The other term in the
reminder can be bounded analogously taking into account that
\[
\begin{split}
\int_0^1(1-t)\left\{\left\{\HH,F\right\},F\right\}
\circ\Phi_F^tdt=&\int_0^1(1-t)\left\{\wt\GG-\GG,F\right\}\circ\Phi_F^tdt\\
&+\int_0^1(1-t)\left\{\left\{\GG,F\right\},F\right\}\circ\Phi_F^tdt.
\end{split}
\]
Analogously, one can obtain bounds for $\Gamma-\mathrm{Id}$ recalling that
\[
 \Gamma=\Id+\int_0^1 X_F\circ \Phi_F^t dt.
\]

\section{Proof of Approximation Theorem \ref{thm:Approximation}}\label{app:Approx}
We devote this section to proof the Approximation Theorem \ref{thm:Approximation}.
Throughout this section $C$ denotes any positive constant independent of $N$ and $\la$.

The solution $\beta^\la$ is expressed in rotating coordinates (see change \eqref{eq:InftyODE:VariationConstants}) and $\al$ is not. To compare them in
a simpler way, we consider the equation \eqref{eq:InftyODE:AfterNF} in rotating
coordinates. To this end, we use that equation \eqref{eq:InfiniteODEAfterNormalForm}
also preserves the $\ell^2$ norm and therefore we perform the change of coordinates
\begin{equation}\label{def:Change:RotatinAfterNF}
 \al_n=g_n e^{i\left(G+|n|^2\right)t},
\end{equation}
with $G=-2\|\al\|_{\ell^2}^2$.  Then, the equation for $g=\{g_n\}_{n\in\ZZ^2}$ reads
\begin{equation}\label{eq:AfterNFInRotating}
- i \dot g_n = \EE_n(g)+ \JJ_n(g),
\end{equation}
where $\EE:\ell^1\rightarrow \ell^1$ is the function defined as
\begin{equation}\label{def:Approx:CubicTerm}
 \EE_n(g)=-|g_n|^2g_n+
\sum_{(n_1,n_2,n_3)\in \AAA(n)} g_{n_1} \overline { g_{n_2}}g_{n_3}
\end{equation}
with $ \AAA(n)\subset (\ZZ^2)^3$  defined in (\ref{eq:InftyODE:AfterVariation}),
and $\JJ:\ell^1\rightarrow \ell^1$ is the vector field associated to the Hamiltonian
\[
 \RRR'\left(g\right)=\RRR\left(\left\{g_n e^{i\left(G+|n|^2\right)t}\right\}_{n\in\ZZ^2}\right),
\]
where $\RRR$ is the Hamiltonian introduced in Theorem \ref{thm:NormalForm}.
Therefore, $\JJ$ satisfies
\begin{equation}\label{eq:Approx:BoundQuinticVF}
\|\JJ(g)\|_{\ell^1}=\OO\left(\|g\|^5_{\ell^1}\right).
\end{equation}
Note that  equation \eqref{eq:AfterNFInRotating} and equation
\eqref{eq:InftyODE:AfterVariation} only differ by $\JJ$, that is,
in the fifth degree terms of the equation. Moreover, note that
$g(0)=\al(0)$ and therefore, by the hypotheses of Theorem \ref{thm:Approximation},
\begin{equation}\label{eq:Approx:InitialCond}
g(0)=\beta^\la(0).
\end{equation}
To prove that $g$ and $\beta$ are close  we define the function $\xi$ as
\begin{equation}\label{def:Approx:Error}
 \xi_n=g_n-\beta_n
\end{equation}
and we apply refined Gronwall-like estimates to bound its $\ell^1$ norm. Thanks to \eqref{eq:Approx:InitialCond}, we have that $\xi(0)=0$. Moreover, from equations
\eqref{eq:InftyODE:AfterVariation} and \eqref{eq:AfterNFInRotating}, one can deduce
the equation for $\xi$. It can be written as
\begin{equation}\label{eq:Approx:InfiniteODE}
\dot \xi =\ZZZ^0(t)+\ZZZ^1(t)\xi+\ZZZ^2(\xi,t)
\end{equation}
where
\begin{align}
 \ZZZ^0(t)=&\JJ\left(\beta^\la\right)\label{def:Approx:FirstIteration}\\
\ZZZ^1(t)=&D\EE\left(\beta^\la\right)\label{def:Approx:Linear}\\
\ZZZ^2(t)=&\EE\left(\beta^\la+\xi\right)-\EE\left(\beta^\la\right)-
D\EE\left(\beta^\la\right)\xi+
\JJ\left(\beta^\la+\xi\right)-\JJ\left(\beta^\la\right).
\label{def:Approx:Higher}
\end{align}
Applying the $\ell^1$ norm to equation \eqref{eq:Approx:InfiniteODE}, we obtain
\begin{equation}\label{eq:Approx:Difference}
 \frac{d}{dt}\|\xi\|_{\ell^1}\leq \left\|\ZZZ^0(t)\right\|_{\ell^1}+\left\| \ZZZ^1(t)\xi\right\|_{\ell^1}+\left\| \ZZZ^2(\xi,t)\right\|_{\ell^1}.
\end{equation}
The next three lemmas give estimates for each term in the right hand side of this equation. Their proofs are deferred to the end of this appendix.
\begin{lemma}\label{lemma:Approx:BoundFirstIteration}
The function $\ZZZ^0$
 defined in \eqref{def:Approx:FirstIteration}
 satisfies
$ \left\|\ZZZ^0\right\|_{\ell^1}\leq C\la^{-5} 2^{5N}.$
\end{lemma}

\begin{lemma}\label{lemma:Approx:BoundLinear}
The linear operator $\ZZZ^1(t)$
satisfies
$\left\| \ZZZ^1(t)\xi\right\|_{\ell^1}\leq \sum_{n\in\ZZ^2}f_n(t)|\xi_n|,$
where $f_n(t)$ are positive functions satisfying
\begin{equation}\label{eq:Approx:BoundLinearFs}
 \int_0^Tf_n(t)dt\leq C\ga N,
\end{equation}
where $T$ is the time given in \eqref{def:Time:Rescaled} and $\ga$ is the constant given in Theorem \ref{thm:ToyModelOrbit}.
\end{lemma}

To obtain estimates for  $\ZZZ^2(\xi,t)$ defined in  \eqref{def:Approx:Higher},
we apply bootstrap.

Assume that for $0<t<T^*$ we have
\begin{equation}\label{cond:Bootstrap}
 \|\xi(t)\|_{\ell^1}\leq C\la^{-3/2}2^{-N}.
\end{equation}
\emph{A posteriori} we will show that the time \eqref{def:Time:Rescaled}
satisfies $0<T<T^*$ and therefore
the bootstrap assumption holds.
\begin{lemma}\label{lemma:Approx:BoundHigher}
Assume that condition \eqref{cond:Bootstrap} is satisfied. Then
the  operator $\ZZZ^2(\xi,t)$
satisfies
\[
 \left\| \ZZZ^2(\xi,t)\right\|_{\ell^1}\leq  C\la^{-5/2}\|\xi(t)\|_{\ell^1}.
\]
\end{lemma}

Combining Lemmas \ref{lemma:Approx:BoundFirstIteration},
\ref{lemma:Approx:BoundLinear}, \ref{lemma:Approx:BoundHigher},
equation \eqref{eq:Approx:Difference} implies
\[ \frac{d}{dt}\|\xi\|_{\ell^1}\leq \sum_{n\in\ZZ^2}\left(f_n(t)+
 C\lambda^{-5/2}\right)|\xi_n|+C\la^{-5} 2^{5N}
\]
To obtain bounds for $\|\xi\|_{\ell^1}$ we write this equation as
\[
\sum_{n\in\ZZ^2} \frac{d}{dt}|\xi_n|\leq
\sum_{n\in\ZZ^2}\left(f_n(t)+C\lambda^{-5/2}\right)|\xi_n|+C\la^{-5} 2^{5N}
\]
and we apply a Gronwall-like argument for each harmonic of $\xi$. Namely, we consider the following change of coordinates,
\begin{equation}\label{def:Approx:GronwallLikeChange}
 \xi_n=\zeta_n e^{\int_0^t\left(f_n(s)+C\lambda^{-5/2}\right)ds}.
\end{equation}
Then, we obtain
\[
\sum_{n\in\ZZ^2} e^{\int_0^t\left(f_n(s)+C\lambda^{-5/2}\right)ds} \frac{d}{dt}|\zeta_n|
\leq C\la^{-5} 2^{5N}
\]
From this equation and taking into account that
\[
 f_n(t)+C\lambda^{-5/2}\geq 0,
\]
we obtain that
\[
\frac{d}{dt}\|\zeta\|_{\ell^1}=\sum_{n\in\ZZ^2} \frac{d}{dt}|\zeta_n|\leq C\la^{-5} 2^{5N}.
\]
Therefore, integrating this equation, taking into account that $\zeta(0)=\xi(0)=0$  and using the bound for
$T$ in \eqref{def:Time:Rescaled} we obtain that
\[
 \|\zeta\|_{\ell^1}\leq C \la^{-3} 2^{5N}\ga N^2
\]
To deduce from this bound, the corresponding bound for $\|\xi\|_{\ell^1}$ it is enough to use the change \eqref{def:Approx:GronwallLikeChange}, the estimate \eqref{eq:Approx:BoundLinearFs} from \ref{lemma:Approx:BoundLinear}
 and the definition of $T$ in \eqref{def:Time:Rescaled}. Then, we obtain
\[
 |\xi_n|\leq e^{C\ga N}e^{\la^{-5/2}T}|\zeta_n|\leq 2e^{C\ga N}|\zeta_n|
\]
which implies
\[
 \|\xi\|_{\ell^1}\leq 2e^{C\ga N} \|\zeta\|_{\ell^1}\leq 2e^{C\ga N}\la^{-3} 2^{5N}\ga N^2.
\]
Therefore, using the condition on $\la$ from Theorem \ref{thm:Approximation} with any $\kk>C$ and taking $N$ big enough,
we obtain that for $t\in [0,T]$
\[
 \|\xi\|_{\ell^1}\leq \la^{-2}
\]
and therefore we can drop the bootstrap assumption \eqref{cond:Bootstrap}.

Finally, taking into account \eqref{def:Approx:Error} and \eqref{def:Change:RotatinAfterNF}
we obtain
\[
 \sum_{n\in\ZZ^2}\left|\al_n e^{-i\left(G+|n|^2\right)t}-\beta_n\right|\leq C\la^{-3/2},
\]
which is equivalent to statement \eqref{eq:Approx:BoundDiff} in Theorem \ref{thm:Approximation}.

It only remains to prove Lemmas \ref{lemma:Approx:BoundFirstIteration},
\ref{lemma:Approx:BoundLinear} and \ref{lemma:Approx:BoundHigher}.
\begin{proof}[Proof of Lemma \ref{lemma:Approx:BoundFirstIteration}]
Taking into account \eqref{eq:Approx:BoundQuinticVF}, we have that
\[
 \left\|\ZZZ^0\right\|_{\ell^1}\leq C\left\|\beta^\la\right\|_{\ell^1}^5.
\]
Therefore it only remains to obtain an upper bound for $\left\|\beta^\la\right\|_{\ell^1}$. Taking into account  that $\mathrm{supp}\{\beta^\la\}\subset\Lambda$, the definition of $\beta^\la$ in \eqref{def:RescaledApproxOrbit} and Theorem \ref{thm:ToyModelOrbit}, we have that
\[
 \left\|\beta^\la(t)\right\|_{\ell^1}\leq \sum_{n\in\Lambda}\left|\beta_n^\la(t)\right|\leq 2^N\la^{-1}\sum_{j=1}^N \left|b_j\left(\la^{-2}t\right)\right|.
\]
Now it only remains to point out that from the results obtained in Theorem 3--bis,
we know that at each time all but three components of $b$ are of size $|b_j|\lesssim \de^\nu$ for certain $\nu>0$ whereas the other two satisfy $|b_j|\leq 1$. Then, using the definition of $\de$ in Theorem \ref{thm:ToyModelOrbit}, we obtain that
\[
\sum_{j=1}^N \left|b_j\left(\la^{-2}t\right)\right|\leq C(1+N\de^\nu)\leq C,
\]
which implies
\begin{equation}\label{eq:Beta:BoundL1}
 \left\|\beta^\la(t)\right\|_{\ell^1}\leq C2^N\la^{-1}.
\end{equation}
This  finishes the proof of the lemma.
\end{proof}

\begin{proof}[Proof of Lemma \ref{lemma:Approx:BoundLinear}]
To proof Lemma \ref{lemma:Approx:BoundLinear} we start by analyzing
each component of $\ZZZ^1(t)\xi$. To this end, we use the function
$\EE$ defined in \eqref{def:Approx:CubicTerm} to obtain
\[
 \left(\ZZZ^1(t)\xi\right)_n=\sum_{k\in\ZZ^2} \pa_{\xi_k}\EE_n\left(\beta^\la\right)\xi_k+\sum_{k\in\ZZ^2} \pa_{\ol \xi_k}\EE_n\left(\beta^\la\right)\ol\xi_k.
\]
We define the functions $f_n$ as
\begin{equation}\label{def:Approx:Defi:fn}
 f_n(t)=\sum_{k\in\ZZ^2} \left|\pa_{\xi_n}\EE_k\left(\beta^\la\right)\right|+\sum_{k\in\ZZ^2}\left| \pa_{\ol \xi_k}\EE_n\left(\beta^\la\right)\right|.
\end{equation}
We analyze them differently whether $n\in\Lambda$ or $n\not\in\Lambda$. We start with the first case.

We fix $n\in\Lambda$ and we want to study which terms in the right hand side of
\eqref{def:Approx:Defi:fn} are non zero. Indeed, each of the terms  $\left|\pa_{\xi_n}\EE_k\left(\beta^\la\right)\right|$ is of the form
$\beta_{n_1}^\la\beta_{n_2}^\la$ with $(n_1,n_2,n)\in\AAA(k)$, $(n,n_2,n_1)\in\AAA(k)$ or
$n_1=n_2=n=k$ (the last case arising due to the term $-|g_n|^2g_n$ in
\eqref{def:Approx:CubicTerm}). Then, these terms are non-zero provided
$\beta_{n_1}^\la\neq 0$ and $\beta_{n_2}^\la\neq 0$. This condition is satisfied provided
$n_1,n_2\in\Lambda$ (see \eqref{def:RescaledApproxOrbit}). Thus, we have that
$n,n_1,n_2\in\Lambda$. Then,  property $1_\La$ of the set $\La$ guarantees that $k\in\Lambda$.
Properties $2_\Lambda$ and $3_\Lambda$ imply that  $n$ only belongs to two nuclear families.
Therefore, it only interacts with seven vertices (recall that
it can interact with itself through the term $-|g_n|^2g_n$ in
\eqref{def:Approx:CubicTerm}). This implies that for a fixed $n$,
\[
\pa_{\xi_n}\EE_k\left(\beta^\la\right)=0
\]
except for seven values of $k$, which correspond to the parents,
children, spouse and sibling of $n$ and $n$ itself. Moreover,
for the same reason, each term $\pa_{\xi_n}\EE_k\left(\beta^\la\right)$
which is non-zero, only contains a finite and independent of $N$ and
$n$ number of summands of the form $\beta_{n_1}\ol\beta_{n_2}$ with
$(n_1,n_2,n)\in\AAA(k)$, $(n,n_2,n_1)\in\AAA(k)$ or $n_1=n_2=n=k$.

Reasoning in the same way, we can obtain analogous results for
the terms $| \pa_{\ol \xi_k}\EE_n(\beta^\la)|$.

From these facts, we can deduce formula
\eqref{eq:Approx:BoundLinearFs} for $n\in\Lambda$. Indeed, we have seen that $f_n$ only involves seven harmonics of $\beta^\la$ and that it is quadratic in them. Then, recalling the definition of $\beta^\la$ in \eqref{def:RescaledApproxOrbit}, Theorem 3-bis ensures  that $f_n(t)$ has size $f_n\sim \la^{-2}$ for a time interval of order $\la^2\ln(1/\de)\sim \la^2 \ga N$ (recall that $\de=e^{- \ga N}$) and has size $f_n\sim\la^{-2}\de^{\nu}\sim \la^{-2}e^{-\ga\nu N}$ for the rest of the time, that is, for a time interval of order $\la^2N\ln(1/\de)\sim \la^2 \ga N^2$. Therefore,
\[
 \int_0^T f_n(t)dt\leq C\left(N+N^2e^{-\ga\nu N}\right)\leq C\ga N.
\]
This finishes the proof for $n\in\Lambda$.

Now we need analgous results for $n\not\in\Lambda$. We need to see which terms of $\left|\pa_{\xi_n}\EE_k\left(\beta^\la\right)\right|$, that are of the form $\beta_{n_1}^\la\beta_{n_2}^\la$,  are non-zero. We know that they are non-zero provided
$(n_1,n_2,n)\in\AAA(k)$ or  $(n,n_2,n_1)\in\AAA(k)$ and  $n_1,n_2\in\Lambda$. Note that know the case $n_1=n_2=n=k$ is excluded since $n\not\in\Lambda$ and $n_1,n_2\in\Lambda$. Since $n\not\in\Lambda$ and $n_1,n_2\in\Lambda$, property $1_\Lambda$ implies that $k\not\in\Lambda$. Then, property $6_\Lambda$ guarantees that there are at most two rectangles with two vertices in $\Lambda$ and two out of $\Lambda$. Therefore, we have that
\[
\pa_{\xi_n}\EE_k\left(\beta^\la\right)=0
\]
except for three values of $k$, which correspond to $n$ itself and the other vertex not belonging to $\Lambda$ of each of these two rectangles. Reasoning as before each term $\pa_{\xi_n}\EE_k\left(\beta^\la\right)$ which is non-zero, only contains a finite and independent of $N$ and $n$ number of summands of the form $\beta_{n_1}\ol\beta_{n_2}$ with $n_1,n_2\in\Lambda$. Then, reasoning as in the previous case, we obtain
\[
 \int_0^T f_n(t)dt\leq C\ga N.
\]
This finishes the proof of the lemma.
\end{proof}

\begin{proof}[Proof of Lemma \ref{lemma:Approx:BoundHigher}]
To prove Lemma \ref{lemma:Approx:BoundHigher}, we split $\ZZZ^2$ in \eqref{def:Approx:Higher} as $\ZZZ^2=\ZZZ^2_1+\ZZZ^2_2$ with
\[
\begin{split}
\ZZZ_1^2(t)=&\EE\left(\beta^\la+\xi\right)-\EE\left(\beta^\la\right)-D\EE\left(\beta^\la\right)\xi\\
\ZZZ_2^2(t)=&\JJ\left(\beta^\la+\xi\right)-\JJ\left(\beta^\la\right).
\end{split}
\]
Using the definition of $\EE$ in \eqref{def:Approx:CubicTerm}, it can be easily seen that
\[
 \left\|\ZZZ_1^2\right\|_{\ell^1}\leq C \left(\|\beta^\la\|_{\ell^1}\left\|\xi\right\|_{\ell^1}^2+\left\|\xi\right\|_{\ell^1}^3\right).
\]
Then, using the bound for $\|\beta^\la\|_{\ell^1}$
obtained in \eqref{eq:Beta:BoundL1} and the bootstrap
assumption \eqref{cond:Bootstrap}, we obtain
\[
 \left\|\ZZZ_1^2\right\|_{\ell^1}\leq C  \la^{-5/2}\left\|\xi\right\|_{\ell^1}.
\]
We proceed analogously for $\ZZZ^2_2$. Indeed, it satisfies
\[
 \left\|\ZZZ_2^2\right\|_{\ell^1}\leq C \sum_{k=1}^5\|\beta^\la\|_{\ell^1}^{5-k}\left\|\xi\right\|_{\ell^1}^k
\]
and applying  \eqref{eq:Beta:BoundL1} and \eqref{cond:Bootstrap} again, we obtain
\[
 \left\|\ZZZ_2^2\right\|_{\ell^1}\leq C  \la^{-5/2}\left\|\xi\right\|_{\ell^1}.
\]
Thus, we can conclude that
\[
 \left\|\ZZZ^2\right\|_{\ell^1}\leq C  \la^{-5/2}\left\|\xi\right\|_{\ell^1}.
\]
\end{proof}

\begin{figure}[t]
  \centering
  \includegraphics[width=4in]{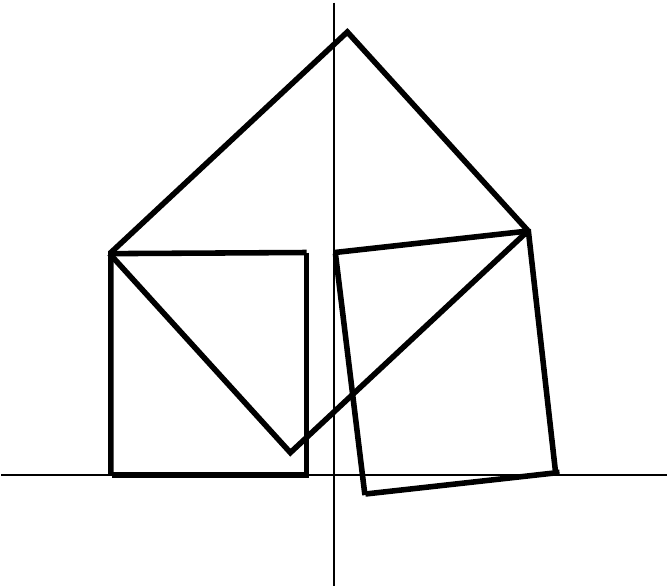}
  %,scale=0.5in
  %width=4in
  \caption{Rectangles}
  \label{fig:rectangles}
\end{figure}

\section{A result for small initial Sobolev norm}\label{app:SmallSobolev}
In Theorem \ref{thm:main} we cannot ensure that the  initial Sobolev norm $\|u(0)\|_{H^s}$
is arbitrarily small as is done in \cite{CollianderKSTT10}. One could impose this condition
at the expense of obtaining a worse estimate for the time $T$.
%M
%Old:
%We devote this appendix to explain how the results stated in Theorem \ref{thm:main}
%change if we assume that $\|u(0)\|_{H^s}$ is arbitrarily small.
In this appendix we state an analog of Theorem \ref{thm:main} under assuming
that $\|u(0)\|_{H^s}$ is arbitrarily small.

\begin{theorem}\label{thm:main:slow}
Let $s>1$. Then there exists $c>0$ with the following
property: for any small $\mu\ll1$ and large $\AAA\gg 1$ there exists a
a global solution $u(t,x)$ of (\ref{def:NLS}) and a time $T$ satisfying
\[
0<T\leq \left(\frac{\AAA}{\mu}\right)^{c\ln (\AAA/\mu)}
\]
such that
\[
\|u(T)\|_{H^s}\ge \AAA\,\,\,\text{ and }\,\,\,\|u(0)\|_{H^s}\leq \mu.
\]
%Moreover, for any $t$ with $1< t < T$, we have
%\[
% \|u(t)\|_{H^s} \ge t^{\frac{1}{c}}  \|u(0)\|_{H^s}.
%\]
\end{theorem}

%M add
\begin{remark}
Combination of Theorems \ref{thm:main} and \ref{thm:main:slow}
covers all regimes studied in \cite{CollianderKSTT10}.
\end{remark}

The proof of this theorem follows along the same lines as the proof of Theorem \ref{thm:main}
explained in Section \ref{sec:SketchProof} taking $\KK=\AAA/\mu$. The only difference
is the choice of the parameter $\la$
to ensure
\[
\|u(0)\|_{H^s}\leq \mu.
\]
Indeed, as it is explained in Section \ref{sec:SketchProof}, we have that
\[
 \left \|u(0)\right\|^2_{H^s}\lesssim \la^{-2}S_3
\]
and, therefore, one needs to choose $\la$ such that $\la^{-2}S_3\sim \mu$. By
Proposition \ref{thm:SetLambda}, the constant $S_3$, defined in \eqref{def:Sums},
depends on $N$. Nevertheless, in that theorem
there is no quantitative estimate of this dependence.
We will compute it here and show how it affects the estimates for the diffusion time $T$.

We will show that there is a choice of the set $\La$ with $S_3$ from
\eqref{def:Sums} satisfying
\begin{equation}\label{def:BoundS3}
S_3
\lesssim B^{N^2},
\end{equation}
for certain $B>0$ independent of $N$,
e.g. $B=60^4$ applies.

First, using this estimate we derive the time estimate  in Theorem \ref{thm:main:slow}
from \eqref{def:BoundS3}. Later we  prove \eqref{def:BoundS3}.
We choose
\[
 \la\sim\frac{1}{\mu}B^{N^2}
\]
so that $\la^{-2}S_3\sim \mu$. Then, by Proposition \ref{thm:SetLambda}
we have $N\sim \ln \KK$. Taking $\KK=\AAA/\mu$, we know that
there exists a constant $c>0$ such that
\[
 \la\lesssim  \left(\frac{\AAA}{\mu}\right)^{c\ln (\AAA/\mu)},
\]
and therefore, using formula \eqref{def:Time:Rescaled} we obtain the estimate for the time.

Now we prove \eqref{def:BoundS3}. To this end we use the construction of the set $\Lambda$
done in \cite{CollianderKSTT10}. Recall that the authors first construct the set $\Lambda$
inside the Gaussian rationals $\Q[i]$ and then multiplying by the least common multiple
they map it to the Gaussian integers $\ZZ[i]$, which is identified with $\ZZ^2$. Now, we want
to place the points in $\Q[i]$ keeping track of the denominators.
This
gives us the size of the harmonics we are dealing with and, therefore, the size of $S_3$.

The placement of the modes in $\Q[i]$ is done inductively generation by generation.
Namely, we first place $\Lambda_1$, then place $\Lambda_2$
checking that the conditions $1_\Lambda -6_\Lambda$ are satisfied, then
place $\Lambda_3$ and so on. Note that
the modes have to be close to the configuration called \emph{prototype embedding}
in \cite{CollianderKSTT10}, Sect. 4, since then we can ensure that \eqref{def:Growth} is satisfied.

{\it First generation:} To place the first generation we consider a grid of points in $\Q[i]$ with
denominator $60^N$. It is clear that we can place $\Lambda_1$ in this grid with the points close
to the first generation of the \emph{prototype embedding} in \cite{CollianderKSTT10}. It can be
done so that (co)tangent of a slope between any two points in $\Lambda_1$ has numerator and
denominator bounded by $Q_1:=60^N$.

{\it Second generation:} The set $\Lambda_1$ is divided in pairs of modes which are the parents of
different nuclear families. For each of these pairs, we need to place a pair of points of $\Lambda_2$
forming a rectangle with the  other pair. These new pair is going to be the children of the nuclear
family. To place it we consider the circle $C$ having as a diameter the segment between the considered
pair in $\Lambda_1$. Then, the children  have to be placed

\begin{itemize}
\item at the endpoints of a different diameter of $C$.
\item they should belong to $\Q[i]$ and
\item the conditions $1_\Lambda -6_\Lambda$ are satisfied.
\end{itemize}
To see that  the children belong to $\Q[i]$, we have to consider a diameter making a Pythagorean
angle with the previous diameter, that is an angle $\theta$ such that $e^{i\theta}\in\Q[i]$ (see
Figure \ref{fig:childrens-choice}).

Let $n=[\sqrt{R/2}]$ be the integer part of $\sqrt{R/2}$.
We lower bound the number of $\tet$'s whose tangent is rational with numerator and denominator
bounded by $R$ as $\sqrt{R/2}$. To see that notice that any triple of the form
$a=m^2-n^2,\ b=2mn,\ c=m^2+n^2$ with $m<n$ is Pythagorean.  Then there are $n-1$ values for $m$
giving a Pythagorean triple.

\begin{figure}[t]
  \centering
  \includegraphics[width=4in]{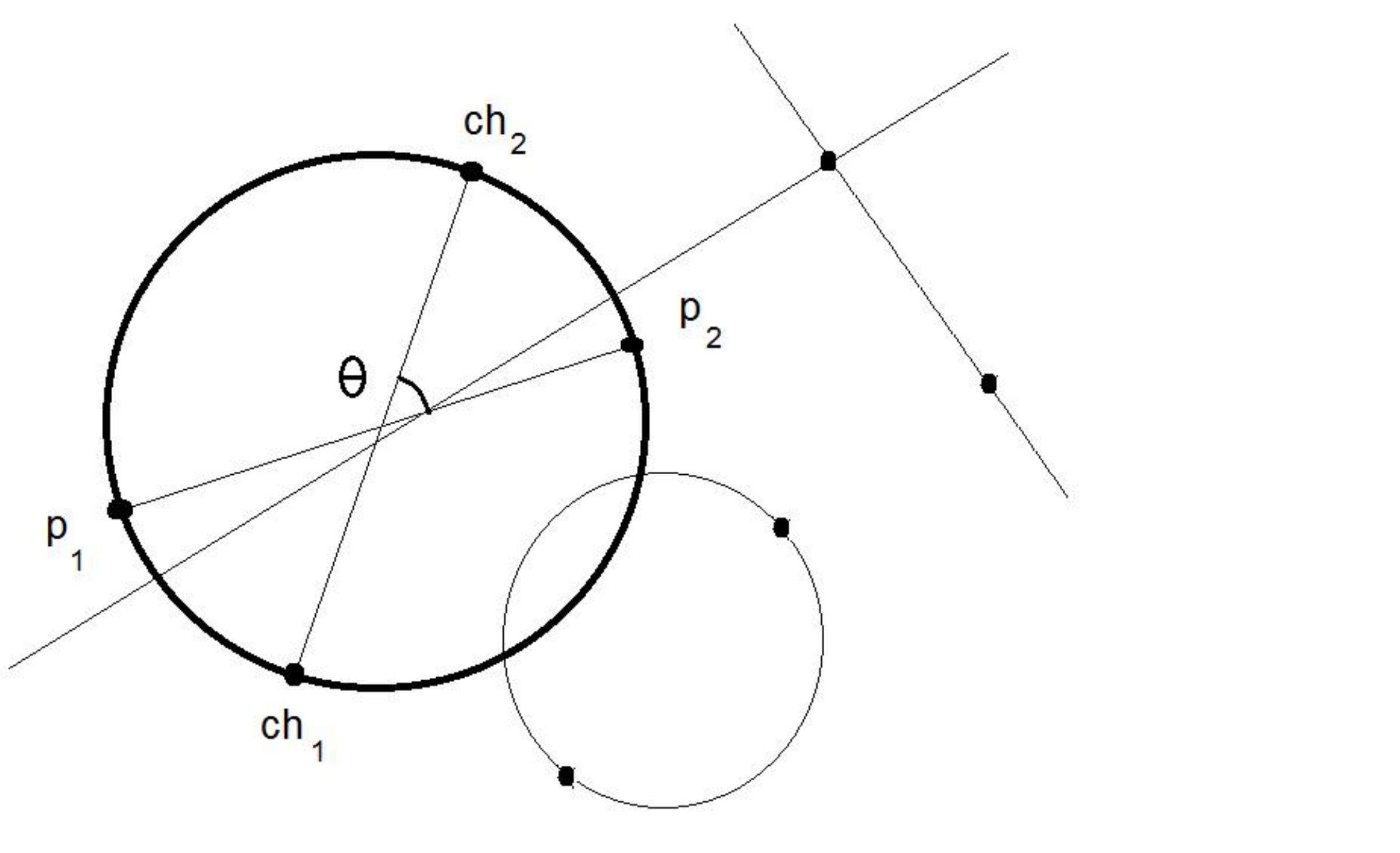}
  \caption{Proper children's choice}
  \label{fig:childrens-choice}
\end{figure}

Conditions $1_\Lambda -6_\Lambda$ are satisfied  provided the modes in $\Lambda_2$ are not placed
in certain points of the circle $C$. The number of these points is of order smaller than $60^N$.
Indeed, we have to exclude:
\begin{itemize}
\item The points of the previous generation ($2^N$ points).
\item The points of $\Lambda_2$ which have already been placed (at most $2^N$).
\item To avoid the existence of more rectangles besides the nuclear families, we proceed as follows.
We consider
\begin{itemize}
\item all the already placed points,
\item all the lines
perpendicular to lines containing two of these points and passing through one of them,
\item all the circles having as a diameter the segment between two of the already placed points
(see Figure \ref{fig:childrens-choice}).
\end{itemize}
Call $\LL$ the set of these lines and $\CCC$
the set of these circles.  The cardinality $|\LL\cup\CCC|$ is at most of order $5^N$.
Then, we have to exclude all the intersections between any object in $\LL\cup\CCC$ with the circle $C$.
\item To ensure that condition $6_\Lambda$ is satisfied, we consider the set $\PP$ of the  points which are
the intersection between any two objects in $\LL\cup\CCC$. It is easy to see that $|\PP|$ is of order at
most $25^N$. Consider the sets
\begin{itemize}
\item $\LL'$ containing the lines which are perpendicular to a line containing a point in $\PP$ and an already placed
point of $\Lambda$, and contain one of these two points,
\item  $\CCC'$ containing the circles having as a diameter a segment whose endpoints are
a point in $\PP$ and an already placed point of $\Lambda$.
\end{itemize}
The cardinality $|\LL'\cup\CCC'|$ is at most of order $60^N$.
Then, we have to exclude also the intersections between an element in $\LL\cup\CCC$ and $C$.
\end{itemize}
We can place the children of the nuclear family at rational points of the circle
$C$ away from the ones just mentioned. To estimate its denominator we apply our estimate
on the number of the Pythagorean triples. We have that the number of $\tet$'s with slopes
whose tangent is given by a rational whose numerator and denominator is bounded by $R$
lower bounded by $\sqrt{R/2}-1$. Thus, we can choose $R= 60^{2N}$. Formula
$\tan (\al + \beta)=(\tan \al + \tan \beta)/(1+\tan \al \tan \beta)$ implies that
$Q_2\le 2\,60^{2N}\,Q_1$. Thus, denominators and numerators in $\La_1\cup \La_2$
are upper bounded by $Q_2$. This grid is accurate enough to place the pairs of
$\Lambda_2$ in the corresponding circles. Iteratively, we can place the following
generations refining the grid at each step by dealing with Gaussian rationals whose
(co)tangent has numerator and denominator bounded by $60^{3jN}$ at the $j$ generation.
Therefore, after placing the $N$ generations and mapping the set $\Lambda$ from $\Q[i]$
to $\ZZ[i]$ we obtain that all the modes $n\in\Lambda$  satisfy
\[
 |n|\lesssim 60^{3N^2}.
\]
This procedure can be done so that the final configuration of modes is close to
the \emph{prototype embedding} in \cite{CollianderKSTT10} to ensure that condition
\eqref{def:Growth} is satisfied. Finally,  to obtain  the estimate \eqref{def:BoundS3},
it is enough to take any $B\ge 60^4$.

\section{Notations}
\begin{itemize}
\item $\KK$ --- growth of the Sobolev norm of
the solution $\|u(t)\|_{H^s}$ from Theorem \ref{thm:main};

\item $s$ --- index of the Sobolev space.

\item $\HH$ --- the Hamiltonian of (\ref{def:NLS}),
defined in (\ref{def:HamForFourier});

\item $\DDD$ --- quadratic part of the Hamiltonian $\HH$ defined in (\ref{def:HamForFourier});

\item $\GG$ --- quartic part of the Hamiltonian $\HH$ defined in (\ref{def:HamForFourier});

\item $\MM$ --- abusing notation, mass of both the solutions of the equation (\ref{def:NLS})
and of the toy model \eqref{def:model}

\item $\{a_n(t)\}_{n\in \ZZ^2}$ --- Fourier coefficients
of the solutions of (\ref{def:NLS})
or, equivalently, solution of system $\dot a_n=2\,i\,\partial_{\overline{a_n}}\HH$;

\item $\Gamma$ --- normal form change for the Hamiltonian \eqref{def:HamForFourier}.
It is given in Theorem \ref{thm:NormalForm}.

\item $\wt\GG$ --- resonant terms of $\GG$.

\item $\RRR$ --- remainder (of degree 5) of the Hamiltonian $\HH$ after performing
one step of normal form, that is remainder of the Hamiltonian $\HH\circ\Gamma$.

\item $\{\al_n(t)\}_{n\in \ZZ^2}$ --- Solutions of the normalized Hamiltonian
$\HH\circ \Gamma$, given in  Theorem \ref{thm:NormalForm};

\item $ \AAA_0(n)\subset (\ZZ^2)^3$ --- collection of the
resonance convolutions defined in (\ref{def:ResonantLatticeBeforeGauge});

%\item $G=-2\|\al\|_{\ell^2}^2$.
% V: we use different G's, and they are different.
\item $\{\beta_n(t)\}_{n\in \ZZ^2}$ --- rotated fourier coefficients,
$\beta_n= \al_n e^{-i(G+|n|^2)t}$. They satisfy
(\ref{eq:InftyODE:AfterVariation}).

\item $ \AAA(n)\subset (\ZZ^2)^3$ --- collection of reduced
resonance convolutions defined after (\ref{eq:InftyODE:AfterVariation});

\item $N-4$ --- number of energy cascades;

\item $\La \subset \ZZ^2$ essential Fourier coefficients given
as a disjoint union of $N$ pairwise disjoint generations:
$\La=\La_1\cup \cdots \La_N$. See Proposition \ref{thm:SetLambda}
and preceding discussion.

\item $\{b_j(t)\}_{j=1}^N$ solution to the Toy Model (\ref{def:model});

\item $h(b)$ --- Hamiltonian of the Toy Model, given in  (\ref{def:Hamiltonian});

\item $\TT_j$ --- periodic orbits of the Toy Model (\ref{def:model})

\item $\{c^{(j)}_k\}_{k\neq j}$ --- coordinates adapted to the periodic orbit
$\TT_j$ after symplectic reduction, given in Section \ref{sec:SaddleMapFirstChanges}.

\item $(p_1,q_1,p_2,q_2)$ --- hyperbolic variables adapted to the periodic
orbit $\TT_j$ after diagonalization, given in Section \ref{sec:SaddleMapFirstChanges}.

\item $\ZZZ_{\hyp,*}, \ZZZ_{\ell,*}, \ZZZ_{\mix,*}$ --- types of remainder
terms of the original Hamiltonian $H$ after symplectic reduction and diagonalization
near the periodic orbit $\TT_j$. Subscript means hyperbolic, elliptic and
mixed remainder respectively (see Lemma \ref{lemma:Diagonalization}).

\item $\Sigma_j^{\inn}$ --- transversal section to the stable manifold of $\TT_j$,
defined in \eqref{def:Section1Saddle}

\item $\Sigma_j^{\out}$ --- transversal section to the unstable manifold of $\TT_j$,
defined in \eqref{def:Section2Saddle}

\item $\BB^j$ --- map from $\Sigma_j^{\inn}$ to $\Sigma_{j+1}^{\inn}$ given by
the flow of the Toy Model (\ref{def:model})
%M
(see Section \ref{sec:ProofToyModelThm}).

\item $\BB^j_\loc$ --- local map from $\Sigma_j^{\inn}$ to $\Sigma_{j}^{\out}$
given by the flow of
%M
(\ref{def:model}), defined in \eqref{def:SaddleMap}.

\item $\BB^j_\glob$ --- global map from $\Sigma_j^{\out}$ to $\Sigma_{j+1}^{\inn}$
given by the flow
%M
(\ref{def:model}), defined in \eqref{def:HeteroMap}.

\item $a=\OO(b)$ means $|b|<Ka$ for some $K$ independent of
$\de,\sigma, N, j$.

\item $a=\OO_\sigma(b)$ means $|b|<Ka$ for some $K$ independent of
$\de, N, j$.

%M
\item $\Psi_{\hyp}$ --- the change of coordinates for the hyperbolic
toy model (see Lemma \ref{lemma:ToyModel:NormalForm}).

%\item $\ZZZ_{\hyp,*},\ZZZ_{\ell,*}, \ZZZ_{\mix,*}$ --- collection
%of remainder terms for the Toy Model after symplectic reduction
%and diagonalization (see Lemma \ref{lemma:Diagonalization}).

\item $\Psi$ --- the change of coordinates for the full
toy model (see Lemma \ref{lemma:FullModel:NormalForm}).

\item $R_{\hyp,*}$, $ R_{\mix,*}$, $,\ZZZ_{\ell,*}$ --- collection of
remainder terms for the Full Toy Model after normal form
transformation $\Psi$ (see Lemma \ref{lemma:FullModel:NormalForm}).

\item $\VV_j\subset \Sigma_j^{\inn}$ ---
%M
an open subset contained in the domain of definition of $\BB^j_\loc$
so that $\BB^j_\loc(\VV_j)\subset \UU_j$.

\item $\UU_j\subset \Sigma_j^{\out}$ --- an open subset
contained in the domain of definition of $\BB^j_\glob$ so that
$\BB^j_\glob(\UU_j)\subset \VV_{j+1}$.

\item $\NNN_j^\pm$ --- initial conditions inside $\Sigma^{\inn}_j$
whose orbits under the flow $\Phi^t$ have cancellation property
(see Lemma \ref{lemma:HypToyModel:SaddleMap})

\item $\WW_j$ --- an auxiliary set in the $(p,q,c)$--space
(see Corollary \ref{coro:SectionOldVariables})

\item $g_{\II_j}(p_2,q_2,\sigma,\de)$ --- the cancellation function,
defined in  (\ref{def:Function_g}) and used in the definition
of  $\NNN_j^\pm$.

\item $T_0$ --- time of evolution of the Toy Model in
Theorem \ref{thm:ToyModelOrbit}.

\item $\ga$ --- constant which gives the relation between $\de$ and $N$.

\item $\KKK$ --- constant from upper bound on time in
Theorem \ref{thm:ToyModelOrbit}.

\item $\la$ --- rescaling parameter see
(\ref{def:Rescaling});

\item $\kk$ --- constant wich gives the relation between $\la$ and $N$.

\item $T$ --- time of evolution after rescaling, see (\ref{def:Time:Rescaled});

\item $\{b^\la_j(t)\}_{j=1}^N$ rescaled solution to the Toy Model, given in
(\ref{def:Rescaling});

\item $\{\beta^\la_n(t)\}_{n\in \ZZ^2}$ the lift of the above
solution to the Toy Model to approximate solution to
(\ref{eq:InftyODE:AfterVariation});

%\item $\JJ$ is the remainder vector field in $\ell^1$
%defined in  (\ref{eq:AfterNFInRotating}).

%\item $\ZZZ^0(t), \ZZZ^2(\xi, t)$ the first and the second
%remainder in (\ref{eq:Approx:InfiniteODE})

%\item $\kk_0$ is a constant from Lemma \ref{lemma:Hyp:FirstIteration}.
\end{itemize}

\section*{Acknowledgements}
The authors would like to thank Terence Tao for pointing out an incorrect
choice of norms  in the original statement of Theorem \ref{thm:main} in
an earlier version of the preprint. We thank Jim Colliander for providing
several interesting references. A. Simon gave a detailed reading to the paper which
improved the exposition.

The authors warmly thank the Fields Insitute and the Institute for Advanced Study for their hospitality, stimulating atmosphere, and support. The first author is partially supported
by the Spanish MCyT/FEDER grant MTM2009-06973 and the Catalan SGR grant 2009SGR859.
The second author is partially supported by NSF grant DMS-0701271.

\bibliography{references}
\bibliographystyle{alpha}
\end{document}